\documentclass[10pt]{amsart}
\pdfoutput=1 
\usepackage{amssymb}
\usepackage{amsmath}
\usepackage{mathrsfs} 

\usepackage[english]{babel}

\usepackage{hyperref}
\usepackage{amsmath}
\usepackage{amsfonts,amssymb,verbatim}
\usepackage{mathrsfs}
\usepackage[english]{babel}

\newtheorem{theorem}{Theorem}[section]
 \newtheorem{corollary}[theorem]{Corollary}
 \newtheorem{lemma}[theorem]{Lemma}
 \newtheorem{proposition}[theorem]{Proposition}
 \theoremstyle{definition}
 \newtheorem{definition}[theorem]{Definition}
 \theoremstyle{remark}
 \newtheorem{remark}[theorem]{Remark}
 \newtheorem{ex}[theorem]{Example}
 \numberwithin{equation}{section}

\def \bC {\mathbb C}

\def \bN {\mathbb N}

\def \bR {\mathbb R}
\def \bS {\mathbb S}
\def \bT {\mathbb T}

\def \bZ {\mathbb Z}

\def \cB {\mathcal B}
\def \cC {\mathcal C}
\def \cD {\mathcal D}

\def \cF {\mathcal F}

\def \cH {\mathcal H}

\def \cL {\mathcal L}
\def \cM {\mathcal M}

\def \cS {\mathcal S}

\def \cV {\mathcal V}

\def \cX {\mathcal X}

\def \fg {\mathfrak g}

\def \tr {\text{\rm Tr}}

\def \id {\text{\rm I}}

\def \supp {\text{\rm supp}}

\def \spec {\text{\rm Spec}}
\def \RepG {\text{\rm Rep}(G)}
\def \FundG {\text{\rm Fund}(G)}
\def \Fund {\text{\rm Fund}}

\def\Op{{{\rm Op}}}

\def \sL{\mathscr L}

\def\Gh{{\widehat{G}}}
\def\Diff{{\rm Diff }}
\def\rank{{\rm rank }}
\def\L2f{L^2_{\mbox{\tiny finite}}(G)}

\begin{document}

\title[Intrinsic pseudo-differential calculi
on any compact Lie group]
{Intrinsic pseudo-differential calculi\\ 
on any compact Lie group}

\author[V. Fischer]
{V\'eronique Fischer}

\address{Department of Mathematics,
Imperial College London,\\
180 Queen's Gate, 
London SW7 2AZ, 
United Kingdom}


\begin{abstract}
In this paper, 
we define  operators  
  on a compact Lie group in an intrinsic way
by means of symbols using the representations of the group.
The main purpose is to show that
these operators form a symbolic pseudo-differential calculus
 which coincides or generalises 
  the (local) H\"ormander pseudo-differential calculus
on the group viewed as a compact manifold. 
\end{abstract}



\maketitle

\tableofcontents

\section{Introduction}

Over the past  five decades, 
pseudo-differential operators have become a powerful and versatile tool in the analysis of Partial Differential Equations (PDE's) in various contexts.
Although they may be used for global analysis
(essentially in the Euclidean setting), 
they can  be localised 
and this allows one to define them on closed manifolds. 
However, on a closed manifold, 
one can no longer attach a global symbol to a single operator in the calculus
(although one could recover a - partial - global definition of operators on manifolds for instance using linear connections, 
see \cite{safarov_97} and the references therein).
The subject of the present paper is to define globally and intrinsically symbolic calculi
 on a special class of manifolds, 
more precisely on any compact Lie group $G$.
Naturally the first aim of this article is to show that the fundamental properties of 
the calculi hold true, thereby justifying the vocabulary.
The second aim of this article is to prove that our calculi coincide 
with the H\"ormander calculi
localised on $G$ viewed as a compact manifold 
- when the H\"ormander calculi can be defined.
We will also show that it coincides with the calculi proposed by Michael Ruzhansky and Ville Turunen in \cite{ruzhansky+turunen_bk}.
Although this is not the purpose of this paper, 
let us mention that 
several applications to PDE's of the calculi 
have ben obtained by Michael Ruzhansky, Ville Turunen and Jens wirth, 
e.g. construction of parametrices, study of global hypoellipticity, 
see \cite{ruzhansky+turunen+wirth,ruzhansky+turunen_bk} and references therein. 

It is quite natural to define pseudo-differential operators 
globally on the torus
by using Fourier series and considering symbols 
as functions of a variable in the torus and another variable in the integer lattice, see for instance
\cite{ruzhansky+turunen_10} and all the references therein.
Michael Taylor argued
in his monograph \cite[Section I.2]{taylor_bk86}
 that an analogue quantisation is formally true 
on  any Lie group of type 1, considering again symbols as functions of a variable of the group $G$ and another variable of its dual $\Gh$
(which is the set of equivalence classes of the unitary irreducible representations of $G$).
Just afterwards,  Zelditch in \cite{zelditch}
defined a (compactly-supported) symbolic pseudo-differential calculus on a hyperbolic manifold
with a related quantisation.
Pseudo-differential calculi have also been defined on the Heisenberg group
by Taylor in \cite{taylor_bk86}, see also  \cite{bfg} and \cite{fischer+ruzhansky},
 and in other directions by Dynin, Folland, Beals, Greiner, Howe (see \cite{folland_meta} and the references therein).
See also \cite{CGGP} for  
a global pseudo-differential calculus on homogeneous Lie groups 
(although it may not qualify as symbolic, 
being defined in terms of properties of the kernels of the operators).

It would be nearly impossible to review in this introduction the vast literature on classes of operators defined on Lie groups
(especially if one has to include all the studies of spectral multipliers of sub-Laplacians).
Instead, in this article, we focus on pseudo-differential operators, 
in the sense that the operators are not necessarily of convolution type.
In this sense, studies of pseudo-differential calculi on Lie groups  
form a much shorter list and the ones known to the author were mentioned  directly or indirectly earlier in this introduction.  

Following the ideas  in the introduction of \cite{bfg},
let us formalise what  is meant here by a calculus:
\begin{definition}
\label{def_pseudo-diff_calculus}
For each $m\in \bR$, 
let $\Psi^m$  be a given Fr\'echet space of continuous operators $\cD(G)\to\cD(G)$.
We say that the space $\Psi^\infty:=\cup_m \Psi^m$ form a \emph{pseudo-differential calculus} 
when it is an algebra of operators satisfying: 
\begin{enumerate}
\item 
\label{item_def_pseudo-diff_calculus_inclusion}
The continuous inclusions
$\Psi^m \subset \Psi^{m'}$ hold for any $m\leq m'$.
\item
\label{item_def_pseudo-diff_calculus_product}
$\Psi^\infty$ is an algebra of operators.
Furthermore if $T_1\in \Psi^{m_1}$, $T_2\in \Psi^{m_2}$, 
then $T_1T_2\in \Psi^{m_1+m_2}$,
and the composition is continuous as a map
$\Psi^{m_1}\times \Psi^{m_2}\to \Psi^{m_1+m_2}$.
\item
\label{item_def_pseudo-diff_calculus_adjoint}
$\Psi^\infty$ is stable under taking the adjoint.
Furthermore if $T\in \Psi^{m}$  then $T^*\in \Psi^{m}$,
and taking the adjoint is continuous as a map
$\Psi^{m}\to \Psi^{m}$.
\item $\Psi^\infty$ contains the differential calculus on $G$.
More precisely, 
$\Diff^m(G)\subset \Psi^m(G)$
for every $m\in \bN_0$.
 \item
 \label{item_def_pseudo-diff_calculus_sobolev}
  $\Psi^\infty$ is continuous on the Sobolev spaces 
 with the loss of derivatives bounded by the order.
 Moreover, for any $s\in \bR$ and $T\in \Psi^m$, $\|T\|_{\sL(H^{s},H^{s-m})}$ is bounded by a semi-norm of $T\in \Psi^m$, up to a constant of $s,m$  and of the calculus.
 \end{enumerate}
 \end{definition}

The operator classes considered in this paper are defined in 
Section \ref{sec_calculus_def} and denoted by 
$$
\Psi^m_{\rho,\delta}(G),
\quad\mbox{or just}\quad
\Psi^m_{\rho,\delta}, \quad m\in \bR, \ 
1\geq \rho\geq \delta \geq 0, \ \rho\not=0, \ \delta\not=1.
$$
The (localised) H\"ormander class of operators defined on the group $G$ viewed as a manifold is denoted by
$$
\Psi^m_{\rho,\delta}(G,loc),
\quad m\in \bR, \ 
1\geq \rho> \delta \geq 0, \ \rho\geq 1-\delta.
$$
The conditions on the parameters $\rho,\delta$ for $\Psi^m_{\rho,\delta}(G,loc)$ comes from the necessary consistency when changing charts, 
and imply  $\rho>\frac12$.
In this paper, we  show that our classes of operators and the H\"ormander calculi coincide when the latter can be defined:

\begin{theorem}
\label{thm_main0}
Let  $\rho,\delta$ be real numbers with $1\geq \rho\geq\delta\geq 0$
with $\delta\not=1$.
Then $\Psi^\infty_{\rho,\delta}(G):=\cup_{m\in \bR}\Psi^m_{\rho,\delta}(G)$ 
is a calculus on $G$ in the sense of 
Definition \ref{def_pseudo-diff_calculus}.
Moreover, 
if $\rho>\delta$ and $\rho\geq 1-\delta$, 
then this calculus coincides with the H\"ormander calculus 
$\Psi^\infty_{\rho,\delta}(G,loc):=\cup_{m\in \bR}\Psi^m_{\rho,\delta}(G,loc)$ 
on $G$ 
viewed as a compact Riemannian manifold.
\end{theorem}

We will often abuse the vocabulary and refer to 
the collection of operators 
$\Psi^\infty_{\rho,\delta}(G)$ as a calculus 
although this is the main aim of this paper to show that it is indeed a calculus in the sense of Definition \ref{def_pseudo-diff_calculus}.

\medskip

The ideas and methods used in this article come from the `classical' harmonic analysis on Lie groups.
We show that multipliers in  the Laplace-Beltrami operator $\cL$ are also in the calculus in a uniform way (see Proposition \ref{prop_mult_t}).
For this, we use the well-known properties of the heat kernel of $\cL$ \cite{varo}
and methods regarding spectral multipliers \cite{alexo}.
This enables us to use Littlewood-Payley decompositions
with uniform estimates for the dyadic pieces.
This also allows us to obtain  precise estimates for the kernels 
of the operators in Section \ref{sec_kernel}.

It seems possible to generalise many of these ideas and methods  
to any Lie group of type-1 and with polynomial growth of the volume
and even to some of their quotients.
The resulting calculi would certainly depend on the choice of 
a fixed left-invariant sub-Laplacian.
An important technical problem would come from the fact
that, on a compact Lie group, 
we choose the Laplace-Beltrami operator
which has a scalar group Fourier transform.
This could no longer be assumed for a general  left-invariant sub-Laplacian.
Another technical issue is the use of weight theory in some parts of the proofs, for instance in see \ref{sec_bilinear}.

\medskip

This paper is organised as follows.
After the preliminaries in
Section \ref{sec_preliminary},
we define 
 the symbol and operator classes in Section \ref{sec_calculus_def}
 studied in this paper.
The main result is stated in Section \ref{subsec_main_result},
where the organisation of the proofs is also explained.
In Section \ref{sec_1stprop}, we present some first results.
In Section  \ref{sec_RT}, 
we recall the definition of the calculus proposed by Michael Ruzhansky and Ville Turunen in \cite{ruzhansky+turunen_bk},
and we show that it coincides with our intrinsic definition.
Section \ref{sec_kernel} is devoted to the study of the kernels associated with our symbols.
In Sections \ref{sec_calculus} and \ref{sec_L2bdd+commutator}, 
we show that our calculus indeed satisfies the properties listed in Definition \ref{def_pseudo-diff_calculus}
and that it can be characterised via commutators, 
thereby coinciding with the H\"ormander calculus. 
Some technical results are proved in 
 \ref{sec_multipliers} and \ref{sec_bilinear}.

\medskip

\noindent\textbf{Notation:}
$\bN_0=\{0,1,2,\ldots\}$ denotes the set of non-negative integers
and 
$\bN_0=\{1,2,\ldots\}$ the set of positive integers.
$\lceil \cdot\rceil$, $\lfloor \cdot\rfloor$ denote the upper and lower integer parts of a real number. 
We also set  $(r)_+:=\max(0,r)$ for any $r\in \bR$.
If $\cH_1$ and $\cH_2$ are two Hilbert spaces, we denote by $\sL(\cH_1,\cH_2)$ the Banach space of the bounded operators from $\cH_1$ to $\cH_2$. If $\cH_1=\cH_2=\cH$ then we write $\sL(\cH_1,\cH_2)=\sL(\cH)$.

\section*{Acknowledgement}

The author is very grateful to the anonymous referee for comments which led to the notion of difference operators presented in this paper.

\section{Preliminaries}
\label{sec_preliminary}

In this section, we set the notation for the group
and some of its natural structures,
such as the convolution, its representations, the Plancherel formula, 
and the Laplace-Beltrami  operator.
References for this classical material may include
 \cite{stein_topics} and \cite{knapp_bk}.  

\subsection{Notation and convention regarding objects on the group $G$}
\label{subsec_notG}
In this paper, $G$ always denotes a connected compact Lie group 
and $n$ is its dimension.
Its Lie algebra $\fg$ is the tangent space of $G$ at 
the neutral element $e_G$.
It is always possible to define a left-invariant Riemannian distance on $G$, denoted by $d(\cdot,\cdot)$.
We also denote by $|x|=d(x,e_G)$  the Riemannian distance on the Riemann between $x$ and the neutral element $e_G$
and by $B(r):=\{|x|<r\}$ the ball about $e_G$ of radius $r>0$. 
 In this paper, $R_0$ denotes the maximum radius of the ball around the neutral element, i.e. $B(R_0) = G$, and
$\epsilon_0\in (0,1)$ denotes 
the radius of a ball $B(\epsilon_0)$ 
which gives a chart around the neutral element for the exponential mapping $\exp_G : \fg \to G$.

We may identify the Lie algebra $\fg$ 
with  the space of left-invariant vector fields.
More precisely, if $X\in \fg$, then we denote by $X$ and $\tilde X$ 
the (respectively) left and right invariant vector fields  given by:
$$
X \phi(x) =\partial_{t=0} \phi(x \exp_G(tX)),
\quad\mbox{and}\quad
\tilde X \phi(x) =\partial_{t=0} \phi(\exp_G(tX)x),
$$
respectively, for $x\in G$ and $\phi\in \cD(G)$.
In this paper, $\cD(G)$ denotes the Fr\'echet space of smooth functions on $G$.
One easily checks 
\begin{equation}
\label{eq_X_tildeX}
X \{\phi(\cdot^{-1})\} (x)= -\left(\tilde X \phi\right)(x^{-1}).
\end{equation}

We denote by $\Diff^1(G)$ the space of smooth vector fields on $G$.
It is a left $\cD(G)$-module generated by any basis of left-invariant vector fields 
or by any basis of right-invariant vector fields.
More generally, for $k\in \bN$,
$\Diff^k(G)$ denotes the space of smooth differential operators of order $k$.
Any element of $\Diff^k(G)$ may be written as a linear combination of 
$a_\alpha(x) X^\alpha$, $|\alpha|= k$, where $a_\alpha\in \cD(G)$,
and 
$$
X^\alpha:= X_1^{\alpha_1}\ldots X_n^{\alpha_n},
$$
having fixed a basis $\{X_1,\ldots,X_n\}$ for $\fg$.
We have a similar property with the right-invariant vector fields 
$\tilde X_1,\ldots,\tilde X_n$.
We also set $\Diff^0(G)=\cD(G)$.
We denote by $\Diff(G)= \cup_{k\in \bN_0} \Diff^k(G)$ the $\cD(G)$-module of all the smooth differential operators on $G$.

The Haar measure is normalised to be a probability measure.
It is denoted by $dx$ for integration and the Haar measure of a set $E$ is denoted by $|E|$. 

If $f$ and $g$ are two integrable functions, i.e. in $L^1(G)$, we define their (non-commutative) convolution $f*g\in L^1(G)$ via  
$$
f*g (x) =\int_G f(y) g(y^{-1}x) dy.
$$
The Young's inequalities holds.
The convolution may be generalised to two distributions $f,g\in \cD'(G)$.

If $\kappa \in \cD'(G)$, 
we denote by $T_\kappa:\cD(G)\to \cD(G)$
given via $T_\kappa(\phi)= \phi * \kappa$
the associated convolution operator. 
More generally, in this paper, we will allow ourselves to keep the same notation for a (linear) operator $T:\cD(G)\to\cD'(G)$ and any of its possible extension as a bounded operator on the Sobolev spaces of $G$ since such an extension, when it exists, is unique. 

\subsection{Representations}

In this paper, a representation of $G$ is any continuous group homomorphism $\pi$ from $G$ to the set of automorphisms of a finite dimensional complex space. The continuity implies smoothness.
We will denote this space $\cH_\pi$ or identify it with $\bC^{d_\pi}$,
where $d_\pi=\dim \cH_\pi$, after the choice of a basis.
We see $\pi(g)$ as a linear endomorphism of $\cH_\pi$ or as a $d_\pi\times d_\pi$-matrix. 
It is said to be  \emph{irreducible} if the only sub-spaces invariant under $G$ are trivial.
If $\cH_\pi$ is equipped with an inner product 
(often denoted $(\cdot,\cdot)_{\cH_\pi}$), then  
the representation $\pi$ is \emph{unitary} if $\pi(g)$ is unitary for any $g\in G$.
For any representation $\pi$, one can always find an inner product on $\cH_\pi$ such that $\pi$ is unitary.
If $\pi$ is a representation of the group $G$, then 
$$
\pi(X)=\partial_{t=0} \pi(\exp_G(X))
$$
defines a representation also denoted $\pi$ of $\fg$ and therefore of its universal enveloping Lie algebra (with natural definitions).

If $\pi$ is a representation of $G$, 
then its \emph{coefficients} are any function of the form 
$x\mapsto (\pi(x)u,v)_{\cH_\pi}$.
These are smooth functions on $G$ and 
we denote by $L^2_\pi(G)$ the complex finite dimensional space of coefficients of $\pi$.
If a basis $\{e_1,\ldots, e_{d\pi}\}$ of $\cH_\pi$ is fixed, 
then the \emph{matrix coefficients} of $\pi$
are the coefficients $\pi_{i,j}$, $1\leq i,j\leq d_\pi$ given by 
$\pi_{i,j}(x)=(\pi(x)e_i,e_j)_{\cH_\pi}$.
If $f\in \cD'(G)$ is a distribution and $\pi$ is a unitary representation, 
we can always define its \emph{group Fourier transform} at $\pi$
denoted by
$$
\pi(f)\equiv \widehat f(\pi)\equiv \cF_G f(\pi) \in \sL(\cH_\pi)
$$
 via
$$
 \pi(f)  = \int_G f(x)\pi(x)^* dx,
\quad \mbox{i.e.}\quad
(\pi(f)u,v)_{\cH_\pi} = \int_G f(x)(u,\pi(x)v)_{\cH_\pi} dx,
$$
since the coefficient functions are smooth.
If $f$ is integrable and $\pi$ unitary, we have 
\begin{equation}
\label{eq_cF_L1}
\|\cF_G \kappa(\pi)\|_{\sL(\cH_\pi)}\leq \|\kappa\|_{L^1(G)}.
\end{equation}
One checks easily that the group Fourier transform maps the convolution  of two distributions $f_1,f_2\in \cD'(G)$
to the matrix product or composition of their group Fourier transforms: 
$$
\cF_G(f_1*f_2) = \widehat f_2\ \widehat f_1.
$$

Two representations $\pi_1$ and $\pi_2$ of $G$ are \emph{equivalent} when there exists a map $U:\cH_{\pi_1}\to \cH_{\pi_2}$ intertwining the representations, that is,
such that $\pi_2U =U\pi_1$.  
In this case, 
one checks easily that $L^2_{\pi_1}(G)=L^2_{\pi_2}(G)$.
If $\pi_1$ and $\pi_2$ are unitary, $U$ is also assumed to be unitary. 
The dual of the group $G$, denoted by $\Gh$, is the set of unitary irreducible representations of $G$
modulo unitary equivalence.
We also consider the set $\RepG$ of the equivalence class of unitary representations  modulo unitary equivalence.

\begin{remark}[Convention]
\label{rem_convention}
We will often identify a representation of $G$ and its class in $\Gh$ or $\RepG$. In particular, we consider the Fourier transform of a function to be defined on $\RepG$ and by restriction on $\Gh$.

If $S$ is a linear mapping on the representation space of a unitary representation $\pi_0$, 
then we can consider the set $\dot S$ of linear mappings $USU^{-1}$ over $\cH_{\pi_1}$ where $\pi_1$ runs over all the representation equivalent to $\pi_0 = U \pi_1U^{-1}$ via the intertwining operator $U$.
We will often identify $S$ with the set  $\dot S$
which will be then referred as a linear mapping on $\cH_\pi$ where $\pi\in \Gh$ is the equivalence class 
of $\pi_0$. 
\end{remark}

\begin{theorem}[Peter-Weyl Theorem]
\label{thm_PW}
The dual $\Gh$ is discrete.
The Hilbert space $L^2(G)$ decomposes as the Hilbert direct sum $\oplus_{\pi\in \Gh} L^2_\pi(G)$.
Moreover, if for each $\pi\in \Gh$, one fixes a realisation as a representation with an orthonormal basis of $\cH_\pi$, 
then the functions $\sqrt d_\pi \pi_{i,j}$, $1\leq i,j \leq d_\pi$, $\pi\in \Gh$, form an orthonormal basis of $G$.
\end{theorem}

The Peter-Weyl theorem yields the Plancherel formula:
\begin{equation}
\label{eq_Plancherel}
\int_G|f(x)|^2 dx
=
\sum_{\pi\in \Gh} d_\pi \|\pi(f)\|_{HS(\cH_\pi)}^2,
\quad f\in L^2(G),
\end{equation}
and the Fourier inversion formula
\begin{equation}
\label{eq_inversion}
f(x)
=
\sum_{\pi\in \Gh} d_\pi \tr\left( \pi(x) \pi(f)\right),
\quad f\in \cC(G), \ x\in G.
\end{equation}
Here $\cC(G)$ denotes the (Banach) space of the continuous functions on $G$.

We denote by 
$$
\L2f := \sum_{\pi\in \Gh} L^2_\pi(G),
$$
the vector space 
formed of finite linear sum of vectors in some $L^2_\pi(G)$, $\pi\in \Gh$.
As each $L^2_\pi(G)$ is a finite dimensional subspace of $\cD(G)$, 
$\L2f\subset \cD(G)$.
The Peter-Weyl Theorem can be stated equivalently as follows:
$\L2f$ is dense in $L^2(G)$ and 
\begin{equation}
\label{eq_cq_Peter+Weyl_thm}
d_\pi \widehat \pi (\pi')= \delta_{\pi=\pi'} \id_{\cH_\pi'},
\end{equation}
for any two representations $\pi,\pi'\in \Gh$,
in the sense that 
$d_\pi \widehat \pi_{i,j} (\pi')= \delta_{\pi=\pi'} \delta_{i,j}$
for any $1\leq i,j\leq d_\pi$, when $\pi$ is realised as a matrix representation.

We will also use specific properties of representations on compact Lie groups 
 in relation with the Laplace-Beltrami operator, 
 see below and in \ref{sec_bilinear}.

\subsection{The Laplace-Beltrami operator}

We can decompose the Lie algebra $\fg$ of $G$ as the direct sum
$\fg=\fg_{ss}\oplus \fg_{ab}$
where $\fg_{ss}$ is semi-simple and $\fg_{ab}$ is abelian.
Note that the group $G$ can be written as the direct product of 
the semi-simple Lie group $G_{ss}$ whose Lie algebra is $\fg_{ss}$
together with the torus $\bT^{\dim \fg_{ab}}$ with the same dimension as $\fg_{ab}$:
$G=G_{ss}\times \bT^{\dim \fg_{ab}}$.
Fixing a scalar product on $\fg_{ab}$
and considering the Killing form on $\fg_{ss}$
yield a scalar product on $\fg=\fg_{ss}\oplus^\perp \fg_{ab}$.
The (positive) \emph{Laplace-Beltrami operator} of the compact Lie group $G$ is 
$$
\cL:= -X_1^2-\ldots- X_n^2,
$$
where $X_1,\ldots,X_n$ are left invariant vector fields which form an orthonormal basis of $\fg$.
However $\cL$ does not depend on a particular choice of such a basis. 
Being invariant under left and right translations,
$\cL$ is a central operator
and its group Fourier transform is scalar:
\begin{equation}
\label{eq_piLlambda}
\forall \pi\in \Gh\quad \exists ! \lambda_\pi\in [0,\infty)\quad
\pi(\cL) =  \lambda_\pi \id_{\cH_\pi}.
\end{equation}

We keep the same notation for $\cL$ and its self-adjoint extension on $L^2(G)$
having as domain of definition the space of all functions $f\in L^2(G)$ 
such that $\cL f\in L^2(G)$. 
Then $\cL$ is a positive self-adjoint operator on $L^2(G)$.
The Peter-Weyl Theorem
yields an explicit spectral decomposition for $\cL$ and of its spectrum:
$$
\spec(\cL)=\{\lambda_\pi,\pi\in \Gh\}.
$$
For any $\lambda\in \bR$, we set:
\begin{equation}
\label{eq_cHlambdacL}
\cH_\lambda^{(\cL)}
:= \ker (\cL-\lambda\id).
\end{equation}
The eigenspace corresponding to the eigenvalue $\lambda\in \spec (\cL)$ is:
\begin{equation}
\label{eq_cHlambdacL1}
\cH_\lambda^{(\cL)}=\oplus_{\lambda_\pi=\lambda} L^2_\pi(G).
\end{equation}
If $\lambda\not\in \spec (\cL)$, $\cH_\lambda^{(\cL)}=0$.
Note that $\cH_\lambda^{(\cL)}$ must be finite dimensional. 
Indeed, the operator  $(\id +\cL)^{s/2}$ is Hilbert-Schmidt
as its kernel is square integrable for any $s>n/2$
by Lemma \ref{lem_sob_embedding}.
Alternatively, this can be viewed as a general property of  an elliptic operator on the compact manifold $G$.

The spectral decomposition 
$L^2(G)=\oplus_{\lambda\in \spec(\cL)} \cH_\lambda^{(\cL)}$ 
shows that
for any function $f:[0,\infty)\to \bC$ 
the operator $f(\cL)$ is densely defined on $L^2(G)$.
By the Schwartz kernel theorem, it admits a distributional convolution kernel which we denote by $f(\cL)\delta_e\in \cD'(G)$:
\begin{equation}
\label{eq_not_f(L)delta}
f(\cL)\phi = \phi * (f(\cL)\delta_e),
\quad \phi\in \cD(G).
\end{equation}
The group Fourier transform of this kernel is 
$$
\cF(f(\cL)\delta_e)(\pi) = f(\lambda_\pi), \quad \pi\in \Gh.
$$

The Sobolev spaces $H^s(G)=H^s$ may be defined as the Hilbert space 
which is the closure of $\cD(G)$ for the norm 
$$
\phi\mapsto \|(\id+\cL)^{s/2} \phi\|_{L^2(G)}=\|\phi\|_{H^s}.
$$
If $s=0$ then $H^0=L^2(G)$. 
If $s\in \bN$, then $H^s$ coincides with the space of function $f\in L^2(G)$ such that $Df\in L^2(G)$ for any $D\in \Diff^k$, $k\leq s$
and an equivalent norm is $\sum_{|\alpha|\leq s} \|X^\alpha \cdot\|_{L^2(G)}$.

\begin{proposition}
\label{prop_L2finite_density}
The space
$\L2f$ is dense in each Hilbert space $H^s$ and in the Fr\'echet space $\cD(G) = \cap_{s\in \bR}H^s = \cap_{s\in \bN}H^s $.
\end{proposition}

\begin{proof}[Sketch of the proof of Proposition \ref{prop_L2finite_density}]
If $f\in H^s$, we set $f_s:=(\id+\cL)^{-s/2} f\in L^2(G)$ and 
$f_{s,\ell}$ the orthogonal projection of $f_s$ onto $\oplus_{\lambda_\pi \leq \ell} L^2_\pi(G) \subset \L2f$.
Then one checks easily that $f_\ell:=(\id+\cL)^{-s/2} f_{s,\ell} \in  \oplus_{\lambda_\pi \leq \ell} L^2_\pi(G)$ converges in $H^s$ to $f$.
The rest of the proof is routine using Lemma \ref{lem_sob_embedding}.
\end{proof}

\section{The symbolic calculus}
\label{sec_calculus_def}

The operator classes which are the subject of this paper are presented in this section. 
We introduce the natural quantisation and our notion of symbols 
in Section \ref{subsec_symbol_quantisation}, 
then  in Section \ref{subsec_my_diff_op} 
our concept  of difference operators and symbol classes.
Eventually, in Section \ref{subsec_main_result},
 the main theorem of this paper is stated
 and we present the organisation of its proof.
 
\subsection{Symbols and quantisation}
\label{subsec_symbol_quantisation}

The natural quantisation and notion of symbols on (type 1 locally compact) groups  is due to Michael Taylor \cite{taylor_bk86}.
On compact Lie groups, $\Gh$ is discrete and the natural quantisation 
is greatly  simplified greatly.
In fact, it may be viewed as a generalisation of the Fourier series on tori.

\begin{definition}
\label{def_inv_symbol}
An \emph{invariant symbol} 
is a collection $\sigma=\{\sigma(\pi), \pi\in \Gh\}$ 
where for each $\pi\in \Gh$, $\sigma(\pi)$ 
is a linear map over $\cH_\pi$
(see Remark \ref{rem_convention}).
\end{definition}

Using a different vocabulary, 
an invariant symbol may be defined as a field of operators over 
$\oplus_{\pi\in \Gh} \cH_\pi$ modulo unitary equivalence.

The space of invariant symbols is denoted by
$$
\Sigma=\Sigma(G)=\{\sigma \ \mbox{invariant symbol}\}.
$$
One checks easily that $\Sigma(G)$ is an algebra for the product of linear mappings.

Since $\pi\in \RepG$ may be written as a finite direct sum 
$\pi= \oplus_j \tau_j$ 
of $\tau_j \in \Gh$, 
any invariant symbol may be naturally extended over $\RepG$
via $\sigma(\pi) := \oplus_j \sigma(\tau_j)$.
We will often identify an invariant symbol with its natural extension as a collection over $\RepG$.

\begin{ex}
\label{ex_widehatf_sigma}
The group Fourier transform of a distribution is an invariant symbol:
$$
\widehat f=\{\pi(f), \pi\in \Gh\}\in \Sigma, \qquad
 f\in \cD'(G).
$$
As already noticed, $\widehat f$ may equally be viewed as a collection over $\RepG$.  
\end{ex}

The set $\cF_G \cD'(G)$ is sometimes called the \emph{space of 
Fourier transform} or,  in the case of the tori, of Fourier coefficients.
Example \ref{ex_widehatf_sigma} shows 
$$
\cF_G \cD'(G)\subset \Sigma(G).
$$
The inclusion is strict as the following description of the image of the Sobolev spaces implies:
\begin{lemma}
\label{lem_hs}
\begin{enumerate}
\item Let $s\in \bR$.
An invariant symbol $\sigma\in \Sigma$ is in $\cF_G H^s$ if and only if 
$\|\sigma\|_{h_s(\Gh)}:=(\sum_{\pi\in \Gh} d_\pi (1+\lambda_\pi)^s\|\sigma\|_{HS(\cH_\pi)}^2)^{1/2}$ is finite.
\item 
An invariant symbol $\sigma\in \Sigma$ is in $\cF_G \cD'(G)$ if and only if 
there exists $s\in \bR$ satisfying
$\|\sigma\|_{h_s(\Gh)}<\infty$.
\end{enumerate}
\end{lemma}

The proof of this statement follows readily from the Plancherel formula
\eqref{eq_Plancherel}, 
the definition \eqref{eq_piLlambda} of the eigenvalue $\lambda_\pi$ of $\cL$, 
and the fact (which follows from Proposition \ref{prop_L2finite_density})
that
$\cD'(G)=\cup_{s\in \bR} H^s(G)$.

\begin{definition}
\label{def_symbol+op}
A \emph{symbol} is a collection 
$\sigma=\{\sigma(x,\pi), (x,\pi)\in G\times \Gh\}$
such that for each $x\in G$, $\sigma(x,\cdot)=\{\sigma(x,\pi), (x,\pi)\in G\times \Gh\}$ is an invariant symbol.

The \emph{operator associated} with $\sigma$  is the operator $\Op(\sigma)$ defined on $\L2f$ via
$$
\Op(\sigma) \phi (x)=\sum_{\pi\in \Gh}
d_\pi \tr \left(\pi(x) \sigma(x,\pi) \widehat \phi(\pi)\right),
\quad
\phi\in \L2f, \ x\in G.
$$
\end{definition}

Naturally an invariant symbol is a symbol `which does not depend on $x$'.
In this case, the corresponding operator is a Fourier multiplier.

The Peter-Weyl theorem implies that if an invariant symbol $\sigma$ is bounded in the  sense that the quantity
\begin{equation}
\label{eq_LinftyGh}
\|\sigma\|_{L^\infty(\Gh)}
:=
\sup_{\pi\in \Gh} \|\sigma(\pi)\|_{\sL(\cH_\pi)}
=
\sup_{\pi\in \RepG} \|\sigma(\pi)\|_{\sL(\cH_\pi)},
\end{equation}
is finite, then the corresponding Fourier mulitplier $\Op(\sigma)$
is bounded on $L^2(G)$ with operator norm 
\begin{equation}
\label{eq_L2bdd_LinftyGh}
\|\Op(\sigma)\|_{\sL(L^2(G))}=\|\sigma\|_{L^\infty(\Gh)}
\end{equation}
The converse holds easily: if $\Op(\sigma)$
is bounded on $L^2(G)$ then $\|\sigma\|_{L^\infty(\Gh)}$ is finite.

Note that, using the notation of Lemma \ref{lem_hs},  
the properties of the Hilbert-Schmidt norm easily imply 
that for any invariant symbol $\sigma$ we have (with quantities possibly unbounded):
\begin{equation}
\label{eq_hs_Linfty_norm}
\|\sigma\|_{h_s(\Gh)}
\leq 
C_s
\|\sigma\|_{L^\infty(\Gh)}
\end{equation}
where
$C_s:=\|(1+\lambda_\pi)^{s/2}\|_{h_0(\Gh)}$ is finite whenever 
$s<-n/2$ by Lemma \ref{lem_sob_embedding}.

Naturally, any convolution operators may be viewed as a Fourier multiplier:
\begin{ex}
\label{ex_symbol_op}
If $\kappa\in \cD'(G)$, then 
$\Op(\widehat \kappa)$ extends to the group Fourier multiplier $T_\kappa:\cD(G) \to \cD'(G)$ associated with $\kappa$, that is,
$$
\widehat  {T_\kappa\phi} = \widehat \kappa \phi, \qquad \phi\in \cD(G).
$$
Equivalently, $T_\kappa$ is the convolution operator $T_\kappa:\phi \mapsto \phi *\kappa$.

For instance, if $\kappa=\delta_{e_G}$ is the Dirac mass at the neutral element then $T_\kappa=\id$ is the identity operator on $\cD(G)$.
More generally, for any $\beta\in \bN_0^n$, 
if $\kappa=(X^\beta)^t \delta_{e_G}(y^{-1})$ then $T_\kappa=X^\beta$.

If an operator $T\in \sL(L^2(G))$ is invariant under left-translation, 
that is, $T(f(x_0\cdot))(x)= (Tf) (x_0x)$, $x,x_0\in G$, $f\in L^2(G)$, 
then the Schwartz kernel theorem implies that 
it is a right convolution operator in the sense that there exists $\kappa\in \cD'(G)$ such that $T=T_\kappa:\phi\mapsto \phi*\kappa$ on $\cD(G)$.
Equation \eqref{eq_L2bdd_LinftyGh} yields
\begin{equation}
\label{eq_Tkappa_sup}
\|T_\kappa\|_{\sL(L^2(G))}
=\sup_{\pi\in \Gh} \|\cF_G \kappa(\pi)\|_{\sL(\cH_\pi)}.
\end{equation}
\end{ex}

If $T$ is a linear operator defined on $\L2f$
(and with image some complex-valued functions of $x\in G$), 
then one recovers the symbol via
\begin{equation}
\label{eq_sigma_T}
\sigma(x,\pi)=
\pi (x)^* (T \pi)(x), 
\quad \mbox{that is,}\quad
[\sigma(x,\pi)]_{i,j}=
\sum_{k} \overline{\pi_{ki} (x)} (T \pi_{kj})(x) ,
\end{equation}
when one has fixed a matrix realisation of $\pi$.
This can be easily checked using \eqref{eq_cq_Peter+Weyl_thm}.
This shows that the quantisation $\Op$ defined above is injective.
Moreover  \eqref{eq_sigma_T} makes sense for any $\pi\in \RepG$
and one checks easily that this coincides with the natural extension of $\sigma(x,\cdot)$ to a collection over $\RepG$.

\begin{definition}
\label{def_extended_symbol}
If $\sigma=\{\sigma(x,\pi), (x,\pi)\in G\times \Gh\}$  is a symbol, 
then it extends naturally to the collection 
$\{\pi (x)^* (\Op(\sigma) \pi)(x), (x,\pi)\in G\times \RepG\}$.
We will often keep the same notation for $\sigma$ and 
the extended collection over  $G\times \RepG$.
\end{definition}

\begin{definition}
\label{def_smooth_entries}
A symbol $\sigma=\{\sigma(x,\pi), (x,\pi)\in G\times \Gh\}$ has (resp.)
\emph{continuous, smooth, integrable, square-integrable entries} in $x$ 
when, having fixed one (and then all) matrix realisation of each $\pi\in \Gh$, 
the entries of $\sigma(x,\pi)$ are respectively continuous, smooth, integrable, square-integrable in $x$.
\end{definition}

\subsection{Difference operators and symbol classes}
\label{subsec_my_diff_op}

Here we introduce our concepts of difference operators and of classes of symbols.

For each $\tau,\pi\in \RepG$ and $\sigma\in\Sigma(G)$,
we define  the linear mapping $\Delta_\tau \sigma (\pi)$ on $\cH_\tau\otimes \cH_\pi$ via:
\begin{equation}
\label{eq_def_Deltatau}
\Delta_\tau \sigma (\pi) := \sigma(\tau\otimes \pi) - 
\sigma(\id_{\cH_\tau}\otimes \pi).
\end{equation}
The restriction of $\Delta_\tau \sigma (\pi)$ to any occurrence of  
$\rho\in \Gh$ in a decomposition of 
$\tau \otimes \pi$, $\pi\in \Gh$ defines the same mapping over $\cH_\rho$. 
Therefore \eqref{eq_def_Deltatau} defines a `partial invariant symbol' on any $\rho\in \Gh$ occurring in $\tau \otimes \pi$, $\pi\in \Gh$.
Let us extend this trivially by defining the mapping to be zero for any  $\rho\in \Gh$ never appearing in any $\tau \otimes \pi$, $\pi\in \Gh$.

\begin{definition}
\label{def_Delta_tau}
The operation $\Delta_\tau$ defined via \eqref{eq_def_Deltatau}
and extended trivially
acts on $\Sigma(G)$ 
and is called  the \emph{difference operator} associated with $\tau\in \RepG$.
\end{definition}

\begin{ex}
\label{ex_diff_op_torus}
The dual of  the torus $\bT =\bR/2\pi\bZ$,
is  $\widehat \bT=\{ e_\ell, \ell\in\bZ\}$ 
where $e_\ell(x)=e^{i \ell x}$, $x\in \bT$.
Note that $e_\ell \otimes e_m = e_{\ell+m}$.
If the invariant symbol $\sigma$ is the Fourier transform of $f\in \cD'(\bT)$
as in Example \ref{ex_widehatf_sigma},
$$
\mbox{that is,}\quad
\sigma=\widehat f, 
\quad
\sigma(e_\ell)=\widehat f(\ell) =\frac 1{2\pi}\int_{0}^{2\pi} f(x) \bar e_\ell(x) dx,
\quad \ell\in \bZ,
$$
then the difference operator $\Delta_{e_\ell}$ is given via
$$
\Delta_{e_\ell} \widehat f(e_m) 
=
\sigma  ( e_\ell \otimes e_m  ) 
-
\sigma ( 1\otimes e_m  ) 
=
\widehat f ( \ell+m  ) 
-
\widehat f( m  ). 
$$
Hence, for $\ell=\pm1 $, $\Delta_{e_\ell}$ is the usual discrete (forward or backward) difference operator on the lattice $\bZ$.
\end{ex}

We also define the iterated difference operators as follows.
For any $a\in \bN$ and  for any $\alpha=(\tau_1,\ldots,\tau_a)\in \RepG^a$, we write
$$
\Delta^\alpha :=\Delta_{\tau_1}\ldots\Delta_{\tau_a},
\qquad
|\alpha|:=a.
$$
If $\pi\in \RepG$ and $\sigma\in \Sigma$, 
then $\Delta^{\alpha}\sigma(\pi)$ is a mapping over 
$$
\cH^{\otimes \alpha }_\pi:= \cH_{\tau_1}\otimes \ldots\otimes \cH_{\tau_a}\otimes \cH_\pi.
$$
We adopt the following conventions:
if $a=0$ and $\alpha=\emptyset$, we define $\Delta^\alpha$
to be the identity operator on $\Sigma(G)$. We also set 
$$
\RepG^0 =\emptyset \quad\mbox{and}\quad
\RepG^* := \cup_{a\in \bN_0} \RepG^a.
$$

We can now define our classes of symbols.

\noindent\textbf{Convention:}
In this paper, $\rho$ and $\delta$ are two real numbers satisfying
$$
1\geq \rho\geq\delta\geq0.
$$

\begin{definition}
\label{def_Smrhodelta}
Let $m\in \bR$.
The set $S^m_{\rho,\delta}(G)$  is the space of all the symbols $\sigma=\{\sigma(x,\pi), (x,\pi)\in G\times\Gh\}$
with smooth entries in $x$ (in the sense of Definition \ref{def_smooth_entries})
such that for each $\alpha \in \RepG^a$ and $D\in \Diff^b$
there exists $C>0$  satisfying
\begin{equation}
\label{eq_def_Smrhodelta}
\forall (x,\pi)\in (x,\Gh)\qquad
 \| D_x \Delta^\alpha \sigma(x,\pi)\|_{\sL(\cH_\pi^{\otimes\alpha})}
\leq C (1+\lambda_\pi)^{\frac{m-\rho a +\delta b}2}.
\end{equation}
\end{definition}

In this definition, it appears that one should check a  non-countable number of conditions for each symbol.
Let us show that it is in fact countable and furthermore that this defines a Fr\'echet structure on $S^m_{\rho,\delta}$.

As the group $G$ is compact, any differential operator $D\in \Diff^b$ may be written as a linear combination of $X^\beta$, $|\beta|=b$, with smooth coefficients on $G$, see Section \ref{subsec_notG}. Thus 
$\sigma\in S^m_{\rho,\delta}(G)$ if and only if the symbol $\sigma$ has smooth entries in $x$ and satisfies
the condition in \eqref{eq_def_Smrhodelta} 
for any $D=X^\beta$, $\beta\in \bN_0^n$, and any $\alpha\in \RepG^*$.

As any representation in $\RepG$ is a finite sum of irreducible representations in $\Gh$, it suffices to check the condition in \eqref{eq_def_Smrhodelta} only for $\alpha\in \Gh^*:=\cup_{a\in \bN_0} \Gh^a$.
We can restrict this even more: recall that the (compact) group $G$  admits a finite set of fundamental representations:
$$
\FundG \subset \Gh \subset \RepG,
$$
in the sense that any representation in $\Gh$ will occur in a tensor product $\otimes _j \tau_j$ of $\tau_j \in \FundG$.
Hence it suffices to check the condition in \eqref{eq_def_Smrhodelta} only for $\alpha\in \FundG^*:=\cup_{a\in \bN_0} \FundG^a$.

These observations imply that a symbol $\sigma$ with smooth entries in $x$
is in $S^m_{\rho,\delta}(G)$ if and only if the following quantities
are finite for all $a,b\in \bN_0$:
$$
\|\sigma\|_{S^m_{\rho,\delta}(G), a,b}
:=
\max_{\substack{\alpha\in \FundG^*,\ \beta\in \bN_0^n\\
|\alpha|\leq a, |\beta|\leq b}}
\sup_{(x,\pi))\in G\times\Gh}
(1+\lambda_\pi)^{-\frac{m-\rho|\alpha| +\delta|\beta|}2}
 \| X^\beta_x \Delta^\alpha \sigma(x,\pi)\|_{\sL(\cH_\pi^{\otimes\alpha})}.
 $$
 
 It is a routine exercise to show that the functions
 $\|\cdot\|_{S^m_{\rho,\delta}(G), a,b}$, $a,b\in \bN_0$,
 are semi-norms on $S^m_{\rho,\delta}(G)$ 
 and that $S^m_{\rho,\delta}(G)$ 
 then becomes a Fr\'echet space.
 One checks easily that if 
\begin{equation}
\label{eq_rhodelta_inclusion}
m_1<m_2,\quad \rho_1\geq \rho_2, \qquad \delta_1\leq \delta_2,
\quad 1\geq \rho_i\geq \delta_i\geq 0, \ i=1,2,
\ \Longrightarrow \
S^{m_1}_{\rho_1,\delta_1}\subset S^{m_2}_{\rho_2,\delta_2},
\end{equation}
and this inclusion continuous.
This shows the property in
Part \eqref{item_def_pseudo-diff_calculus_inclusion}, 
 of Definition \ref{def_pseudo-diff_calculus}.

\begin{definition}
We say that a symbol is \emph{smoothing} 
when it is in 
$$
S^{-\infty}(G) = \cap_{m\in \bR}S^m_{\rho,\delta}(G).
$$
\end{definition}
One checks easily that indeed, $S^{-\infty}(G)$ does not depend on  $\rho$ and $\delta$. It is naturally endowed with a projective topology.

\begin{remark}
In the case of the torus (see Example \ref{ex_diff_op_torus}), 
Fund$(\bT)=\{e_{\pm 1}\}$ and 
the class of symbol $S^m_{\rho,\delta}(\bT)$ coincides with the one considered in \cite{ruzhansky+turunen_10}.
\end{remark}

\subsection{The main result}
\label{subsec_main_result}

We can now define the classes of operators on $G$ we are studying:
$$
\Psi^m_{\rho,\delta}(G):=
\Op(S^m_{\rho,\delta}(G)),
\qquad m\in \bR \cup\{-\infty\}.
$$
and restate our main result.

For $m\in \bR$, the space $\Psi^m_{\rho,\delta} (G)$ inherits the Fr\'echet topology via  the semi-norms  $\|\cdot\|_{\Psi^m_{\rho,\delta} (G), a,b}$ defined by:
$$
\|T\|_{\Psi^m_{\rho,\delta} (G), a,b}
:=
\|\sigma\|_{S^m_{\rho,\delta} (G), a,b}
\quad\mbox{when}\ T=\Op(\sigma).
$$
The properties of inclusion similar to \eqref{eq_rhodelta_inclusion} hold.
The smoothing operators are defined in a similar manner as well.

Let us now restate the main result of this paper (which was also given in the introduction):

\begin{theorem}
\label{thm_main}
Let  $\rho,\delta$ be real numbers with $1\geq \rho\geq\delta\geq 0$
with $\delta\not=1$.
Then $\Psi^\infty_{\rho,\delta}(G):=\cup_{m\in \bR}\Psi^m_{\rho,\delta}(G)$ 
is a calculus on $G$ in the sense of 
Definition \ref{def_pseudo-diff_calculus}.
Moreover, 
if $\rho>\delta$ and $\rho\geq 1-\delta$, 
then this calculus coincides with the H\"ormander calculus 
$\Psi^\infty_{\rho,\delta}(G,loc):=\cup_{m\in \bR}\Psi^m_{\rho,\delta}(G,loc)$ 
on $G$ 
viewed as a compact Riemannian manifold.
\end{theorem}

Implicit in the theorem is the fact that any operator $T\in \Psi^\infty_{\rho,\delta}(G)$ extends uniquely to a continuous operator $\cD(G)\to\cD(G)$.
This is proved in Lemma \ref{lem_op_cD2cD}.

Although it is the aim of this paper to show that $\Psi^\infty_{\rho,\delta}(G)$ is a calculus, we will abuse the vocabulary and refer to it as \emph{the intrinsic $(\rho,\delta)$-calculus}.

Another important  result of the paper is the fact that the Laplace operator and its spectral calculus are part of the calculus:
\begin{proposition}
\label{prop_mult}
 For any function  $f:\spec(\cL) \to \bC$,
the spectral multiplier $f(\cL)$ is in $\Psi^m_{1,0}$
provided that $\sup_{\lambda\in \spec(\cL)} 
(1+\lambda)^{- \frac {m}  2} 
| f(\lambda)|<\infty$.

Moreover, the symbol given by $f(x,\lambda_\pi)$
is in $S^m_{1,\delta}$ 
provided that 
the function $f:G\times \spec(\cL) \to \bC$ 
satisfies
$$
\forall \beta\in \bN_0^n\qquad
\sup_{\lambda\in \spec(\cL)} 
(1+\lambda)^{- \frac {m+\delta|\beta|}  2} 
| X^\beta_x f(x,\lambda)| <\infty.
$$
\end{proposition}

In fact, in \ref{sec_multipliers},
we will also prove a property as in Proposition \ref{prop_mult} but for multipliers in $t\cL$, uniformly in $t\in (0,1)$;
this property is stated in Proposition \ref{prop_mult_t}
and this is the main technical argument of this paper.
It enable us to use Littlewood-Payley decompositions and analyse precisely the singularity of the kernels,
and these two results are the keys to show the rest of the properties of the calculus.

\medskip

The proof of Theorem \ref{thm_main} is organised as follows.
In Section \ref{sec_1stprop}, 
we show that the symbol classes form an algebra, that the differential calculus is in the intrinsic calculus and we define our notion of kernels associated with a symbol.
In Section  \ref{sec_RT}, we recall the definition of the calculus proposed by Michael Ruzhansky and Ville Turunen and we show that it coincides with our intrinsic definition.
Section \ref{sec_kernel} is devoted to the study of the kernels associated with our symbols.
In Section \ref{sec_calculus}, we show that our calculus satisfy the properties of composition and adjoint 
as in Parts 
\eqref{item_def_pseudo-diff_calculus_product}
and
\eqref{item_def_pseudo-diff_calculus_adjoint}
of Definition \ref{def_pseudo-diff_calculus}.
In Section \ref{sec_L2bdd+commutator}, 
we show that our operators are bounded on Sobolev spaces 
as in Part \eqref{item_def_pseudo-diff_calculus_sobolev}
of Definition \ref{def_pseudo-diff_calculus}
and that it can be characterised via commutators. 
This implies that our calculus coincides with the H\"ormander calculus when the latter is defined and concludes the proof of  Theorem \ref{thm_main}.
In \ref{sec_multipliers}, we prove 
Propositions \ref{prop_mult} and \ref{prop_mult_t}.
In  \ref{sec_bilinear}, we show a bilinear estimate used in 
Section \ref{sec_L2bdd+commutator}.

\section{First properties}
\label{sec_1stprop}

\subsection{The algebra of symbols}
\label{subsec_symbol_algebra}

In this section, we summarise properties of the classes of symbols which are easily obtained.

\begin{proposition}
\label{prop_symbol_1stprop}
\begin{enumerate}
\item 
If $\sigma\in S^m_{\rho,\delta}(G)$, 
then for any $\alpha,\beta\in \bN_0^n$,
 $X^\beta \Delta^\alpha \sigma \in S^{m-\rho |\alpha|+\delta|\beta|}_{\rho,\delta} (G)$ and $$
\| X_x^\beta \Delta^\alpha \sigma\|_{S^{m-\rho |\alpha|+\delta|\beta|}_{\rho,\delta} , a,b}
\lesssim_{a,b,\alpha,\beta, m} \| \sigma\|_{S^{m-\rho |\alpha|+\delta|\beta|}_{\rho,\delta}, a+|\alpha|,b+|\beta|}.
$$
\item 
If $\sigma\in S^m_{\rho,\delta}(G)$, 
then the symbol 
$$
\sigma^*=\{\sigma(x,\pi)^*, (x,\pi)\in G\times\Gh\}
$$
is in $S^m_{\rho,\delta}(G)$ 
and 
$$
\| \sigma^*\|_{S^{m}_{\rho,\delta} , a,b,}
= \| \sigma\|_{S^{m}_{\rho,\delta} , a,b}.
$$
\item 
If $\sigma_1\in S^{m_1}_{\rho,\delta}(G)$ 
and $\sigma_2\in S^{m_2}_{\rho,\delta}(G)$ 
then the symbol 
$\sigma=\sigma_1\sigma_2$
is in $S^{m_1+m_2}_{\rho,\delta}(G)$ 
and 
$$
\| \sigma\|_{S^{m_1+m_2}_{\rho,\delta} , a,b}
\lesssim_{a,b, m} 
\| \sigma_1\|_{S^{m_1}_{\rho,\delta} , a,b}
\| \sigma_2\|_{S^{m_2}_{\rho,\delta} , a,b}.
$$
\end{enumerate}
\end{proposition}
\begin{proof}
The first property in this statement is straightforward from the properties of the tensor product and of the representations.
The second one follows from 
$$
\{\Delta_\tau \sigma (\pi)\}^*=
\Delta_\tau (\sigma^*) (\pi).
$$
For the last one, we notice that our  difference operators
generally do not satisfy  exactly a Leibniz property since
one can check that for any $\sigma_1,\sigma_2\in \Sigma(G)$ 
and $\tau,\pi\in \Gh$,
$$
\Delta_\tau (\sigma_1\sigma_2)(\pi)=
\Delta_\tau (\sigma_1) (\pi)\ \sigma_2(\id_\tau\otimes\pi)
+
\sigma_1 (\tau\otimes\pi)\ \Delta_\tau (\sigma_2)(\pi).
$$
However 
taking taking the supremum over $\pi\in \Gh$ of the $\sL(\cH_{\pi\otimes\tau})$- norm of the expression above,
this yields 
 (see \eqref{eq_LinftyGh}):
$$
\|\Delta_\tau (\sigma_1\sigma_2)\|_{L^\infty(\Gh)}
\leq
\|\Delta_\tau (\sigma_1) \|_{L^\infty(\Gh)} 
\|\sigma_2\|_{L^\infty(\Gh)}
+
\|\sigma_1\|_{L^\infty(\Gh)} 
\| \Delta_\tau (\sigma_2)\|_{L^\infty(\Gh)} ,
$$
with quantities possibly infinite.
More generally, it is not difficult to prove recursively that we have
for any $\alpha\in \FundG$:
\begin{equation}
\label{eq_leibniz_LinftyGh}
\|\Delta^\alpha (\sigma_1\sigma_2)\|_{L^\infty(\Gh)}
\leq C_\alpha
\sum_{|\alpha_1|+|\alpha_2| = |\alpha|} 
\|\Delta^{\alpha_1} (\sigma_1) \|_{L^\infty(\Gh)} 
\|\Delta^{\alpha_2} (\sigma_2) \|_{L^\infty(\Gh)} 
\end{equation}
And this easily implies the last property in the statement above.
\end{proof}

Consequently, we have:
\begin{corollary}
\label{cor_algebra_of_symbol}
The classes of symbols $\cup_{m\in \bR}S^m_{\rho,\delta}$ form an algebra stable under taking the adjoint. 
Moreover the operations of composition and taking the adjoint are continuous.

Furthermore if $\sigma_0$ is smoothing,
then for any $\sigma\in S^m_{\rho,\delta}$, 
the symbols
$\sigma\sigma_0$ and $\sigma_0\sigma$ are also smoothing.
\end{corollary}

Note that the calculus is invariant under translations in the following sense:
\begin{lemma}
\label{lem_Psi0_inv_left_translation}
If $T\in \Psi^m_{\rho,\delta}$ then for all $x_o \in G$, 
 the operator $\tau_{x_o} T\tau_{x_o}^{-1}$
  is in $\Psi^m_{\rho,\delta}$
 where $\tau_{x_o}:f\mapsto f(x_o\ \cdot)$ is the left translation.
Furthermore, if $\kappa_x$ is the kernel of $T$ and $\sigma=\Op^{-1}(T)$ is its symbol,
then $\tau_{x_o}T\tau_{x_o}^{-1}$ has $\kappa_{x_ox}$ as kernel and $\sigma(x_ox,\pi)$ as symbol, and 
$$
\|T\|_{\Psi^m_{\rho,\delta},a,b}=\|\tau_{x_o}T\tau_{x_o}^{-1}\|_{\Psi^m_{\rho,\delta},a,b}.
$$
\end{lemma}

\subsection{The differential calculus}

We can now give important examples of operators in the intrinsic calculus.
Namely  we prove that the differential calculus,
that is, $\cup_{k\in \bN_0} \Diff^k$,
is included in $\Psi^\infty_{1,0}$.
We start with studying the case of the operator $X^\beta$:

\begin{lemma}
\label{lem_Deltaq_piXbeta}
Let $\beta\in \bN_0^n$ and $\alpha\in \FundG^*$.
Then if $|\beta|< |\alpha|$ then $\Delta^\alpha \sigma=0$.
If $|\beta|\geq |\alpha|$ then
there exists $C=C_{\alpha,\beta}$ such that
$$
\forall \pi\in \Gh\qquad
\|\Delta^\alpha \pi(X)^\beta\|_{\sL(\cH_\pi^{\otimes\alpha})} 
\leq C (1+\lambda_\pi)^{\frac{|\beta|}2(|\alpha|+1)}.
$$
\end{lemma}

\begin{proof}
We may assume $\beta\not=0$.
Since  $X^{\beta}$ maps $H^s$ to $H^{s-|\beta|}$, 
the map $(\id+\cL)^{-|\beta|/2} X^{\beta}$ is bounded on $L^2(G)$ 
and this implies
(see \eqref{eq_L2bdd_LinftyGh}) 
\begin{equation}
\label{eq_piXbeta_sup}
\sup_{\pi\in \Gh}
(1+\lambda_\pi)^{-|\beta|/2}
\|\pi(X)^{\beta}\|_{\sL(\cH_\pi)} <\infty.
\end{equation}
This shows the case $\alpha=\emptyset$, i.e $\Delta^\alpha=\id$.

Let us now consider any $\tau\in \Gh$ and $|\beta|=1$, that is $X^\beta=X_j$ for some $j=1,\ldots,n$.
To avoid confusions, let us define $\sigma\in \Sigma$ via $\sigma(\pi)=\pi(X_j)$.
One computes easily
$$
(\tau\otimes\pi)( X_j)=\tau(X_j)\otimes \id_{\cH_\pi}
+
\id_{\cH_\tau}\otimes\pi(X_j)
$$
 for any $\tau,\pi\in \RepG$,
thus
\begin{equation}
\label{eq_pf_lem_Deltaq_piXbeta}
\Delta_\tau \sigma(\pi) = 
\tau(X_j)\otimes  \id_{\cH_\pi},
\end{equation}
and by \eqref{eq_piXbeta_sup},
$$
\|\Delta_\tau \sigma(\pi)\|_{\sL(\cH_{\tau\otimes\pi})} 
\leq 
\|\tau(X_j)^\beta\|_{\sL(\cH_\tau)} 
\leq C_j (1+\lambda_\tau)^{\frac 12}.
$$
If $\tau_1,\tau_2,\pi\in \Gh$, 
we have by definition of $\Delta_{\tau_1}$:
$$
\Delta_{\tau_1}
\Delta_{\tau_2}
 \sigma(\pi) = 
\Delta_{\tau_2}
 \sigma (\tau_1 \otimes \pi)
  -
 \Delta_{\tau_2}
 \sigma (\id_{\cH_{\tau_1}}\otimes \pi),
$$
but by \eqref{eq_pf_lem_Deltaq_piXbeta}, 
both terms 
$\Delta_{\tau_2}\sigma (\tau_1 \otimes \pi)$
and 
$ \Delta_{\tau_2}\sigma (\id_{\cH_{\tau_1}}\otimes \pi)$
are equal to $\tau_2(X_j) \otimes \id_{\cH_{\tau_1}}\otimes \id_{\cH_\pi}$.
Therefore $\Delta_{\tau_1}
\Delta_{\tau_2}
 \sigma=0$.
 This shows Lemma \ref{lem_Deltaq_piXbeta} in the case $|\beta|=1$.

Writing a general $X^\beta$ as a product of various $X_j$'s
and using \eqref{eq_leibniz_LinftyGh}
imply easily the general statement in Lemma \ref{lem_Deltaq_piXbeta}.
\end{proof}

Lemma \ref{lem_Deltaq_piXbeta}
 implies  that $\pi(X^\beta)\in S^{|\beta|}_{1,0}(G)$.
 More generally we readily obtain that the differential calculus is included in  $\Psi^\infty$:

\begin{corollary}
\label{cor_Diff_Psi}
Any $T\in \Diff^m$ may be written as 
$T=\sum_{|\alpha|\leq m} a_\alpha X^\alpha$
where $a_\alpha\in \cD(G)$
and its symbol is then 
$$
\sigma_{T}(x,\pi)=\sum_{|\alpha|\leq m} a_\alpha(x) \pi(X)^\alpha.
$$
Moreover $T\in \Psi^m_{1,0}(G,\Delta)$.
\end{corollary}

\subsection{Kernels and smooth symbols}

An important notion in the analysis of our operators in the intrinsic calculus 
is the following notion of kernel. 

\begin{definition}
The symbol $\sigma=\{\sigma(x,\pi), (x,\pi)\in G\times \Gh\}$ 
admits an \emph{associated kernel} when for each $x\in G$, we have $\sigma(x,\pi)\in \cF_G (\cD'(G))$.
Then its associated kernel is $\kappa_x:=\cF_G^{-1} \sigma(x,\cdot)$.
\end{definition}

If  $\kappa_x$ is the associated kernel of  $\sigma=\{\sigma(x,\pi), (x,\pi)\in G\times \Gh\}$, the Fourier inversion formula (see \eqref{eq_inversion}) implies then
\begin{equation}
\label{eq_Opsigma_*}
\Op(\sigma) \phi (x)=\phi*\kappa_x(x)
=\sum_{\pi\in \Gh}
d_\pi \tr \left(\pi(x) \sigma(x,\pi) \widehat \phi(\pi)\right),
\end{equation}
for $\phi\in \L2f$, $x\in G$.

\begin{remark}
\label{rem_symbol_distrib}
We could have only assumed some distributional dependence in $x$, 
i.e. the coefficients of $x\mapsto\sigma(x,\pi)$ are in $\cD'(G)$, 
then the quantisation formula in \eqref{eq_sigma_T} would still make sense and be valid. 
Moreover in this case, 
by the Schwartz kernel theorem,
a sufficient condition for a symbol to admit an associated kernel is that 
$\Op(\sigma)(\L2f)\subset \cD'(G)$ and that
$\Op(\sigma)$ extends to a linear continuous operator $\cD(G)\to\cD'(G)$, this extension being unique as $\L2f$ is dense in $\cD(G)$
by Proposition \ref{prop_L2finite_density}.
However in our analysis, we will usually assume regularity in $x$, 
see below.
So we do not seek the greatest generality 
and we prefer assuming that each symbol makes sense at each point $x\in G$.
The only exception in this paper is in the proof of Proposition \ref{prop_converse_commutator}.
\end{remark}

\begin{definition}
\label{def_continuoussymbol}
A \emph{continuous symbol} is a collection 
$\sigma=\{\sigma(x,\pi), (x,\pi)\in G\times \Gh\}$
such that  the associated kernel $\kappa_x $ is a distribution depending continuously on $x$.
\end{definition}

In fact, if the symbol $\sigma$ is continuous, 
then $\Op(\sigma)$ extends (uniquely) as a continuous linear operator $\cD(G)\to \cC(G)$ and the quantisation formula in \eqref{eq_Opsigma_*} holds for any $\phi\in \cD(G)$.

\begin{definition}
\label{def_smooth_symbol}
A \emph{smooth symbol} is a continuous symbol with smooth entries
 and such that for any $D\in \Diff$, 
$\{D_x \sigma(x,\pi)\}$ is a continuous symbol.
\end{definition}

If the symbol $\sigma$ is smooth then 
$x\mapsto \kappa_x\in \cD'(G)$ is smooth
and   $\Op(\sigma):\cD(G)\to\cD(G)$ is continuous as an operator valued in $\cD(G)$.

Naturally if the symbol $\sigma$ is invariant
and if $\sigma \in \cF_G(\cD'(G))$, then it is smooth
and its associated kernel is $\cF_G^{-1}\sigma$, see Example
\ref{ex_symbol_op}. In particular, we have:

\begin{ex}
\label{ex_kernel_Xbeta}
For any $\beta\in \bN_0^n$, 
the operator $X^\beta$ admits for symbol 
$\pi(X)^\beta$ which is invariant, i.e. does not depend on $x$.
The associated kernel 
is $\kappa(y)=(X^\beta)^t \delta_{e_G}(y^{-1})$.
\end{ex}

\begin{lemma}
\label{lem_op_cD2cD}
Any symbol $\sigma$ in $S^m_{\rho,\delta}(G)$ is  smooth in the sense of definition \ref{def_smooth_symbol}.
Therefore any operator in $\Psi^\infty _{\rho,\delta}(G)$
is continuous $\cD(G)\to\cD(G)$.
\end{lemma}

\begin{proof}
We fix for instance $s=-\lceil n/2\rceil$.
By \eqref{eq_hs_Linfty_norm}, for  $\beta\in \bN_0^n$, we have
$$
\|X_x^\beta \sigma(x,\cdot)\|_{h_{s-m}(\Gh)}
=
\|(1+\lambda_\pi)^{-\frac {m +\delta|\beta|}2} X_x^\beta \sigma(x,\cdot)\|_{h_{s}(\Gh)}
\lesssim 
\|(1+\lambda_\pi)^{-\frac {m +\delta|\beta|} 2} X_x^\beta \sigma(x,\cdot)\|_{L^\infty(\Gh)}.
$$
This shows in particular for $\beta=0$ that 
the distribution $\kappa_x:=\cF_G^{-1} \sigma(x,\cdot)$ is in  
the Sobolev space $H^{s-m}$ by Lemma \ref{lem_hs}.
We also have
$$
\max_{x\in G}\|X_x^\beta \kappa_x\|_{H^{s-m}}
=
\max_{x\in G}\|X_x^\beta \sigma(x,\cdot)\|_{h_{s-m}(\Gh)}
\lesssim 
\|\sigma\|_{S^m_{\rho,\delta}, 0,|\beta|}.
$$
The continuous inclusion of any Sobolev Space $H^{s_1}$ in $\cD'(G)$
implies that
 $x\mapsto  X^\beta_x \kappa_x$ is continuous from $G$ to $\cD'(G)$
and this concludes the proof of the statement.
\end{proof}

\medskip

The following easy lemma implies that one can always approximate an operator 
of a smooth symbol  by an operator with a smooth kernel in the following way:
\begin{lemma}
\label{lem_approximation_kernel}
Let $\sigma$ be a symbol.
For each $\ell\in \bN$, 
we define the  symbol $\sigma_\ell$ via
$$
\sigma_\ell(x,\pi) = 
\left\{\begin{array}{ll}
\sigma(x,\pi) & \mbox{if} \ \lambda\leq \ell \\
0&\mbox{if} \ \lambda> \ell \\
\end{array}\right.
$$ 
Then for a fixed $\ell\in \bN$, 
$\sigma_\ell$ admits a kernel $\kappa_{\ell,x}\in L^2(G)\cap C^\infty(G)$.

For each $x\in G$ and $\phi\in \L2f$, 
we have the convergence
$\Op(\sigma_\ell) \phi\to\Op(\sigma) \phi$
as $\ell\to\infty$
since
$\Op(\sigma_\ell) \phi-\Op(\sigma) \phi=0$
for $\ell> \ell_0$ where $\ell_0$ is such that
$\supp\, \widehat \phi\subset\{\pi\in \Gh:\lambda_\pi\leq\ell_0\}$.

If $\sigma$ is continuous or smooth, then so is $\sigma_\ell$.
\end{lemma}

\begin{proof}
The Plancherel formula \eqref{eq_Plancherel} yields
the square-integrability of $\kappa_{\ell,x}$.
The convergence follows from
$$
\Op(\sigma_\ell) \phi(x)-\Op(\sigma) \phi(x)
=
-\sum_{\pi\in \Gh: \lambda_\pi>\ell}
d_\pi \tr \left(\pi(x) \sigma(x,\pi) \widehat \phi(\pi)\right).
$$
 \end{proof}

\section{An equivalent characterisation of our operator classes}
\label{sec_RT}

In this section, we recall the definition of the differential calculus 
proposed by Michael Ruzhansky and Ville Turunen in \cite{ruzhansky+turunen_bk}. 
We then show that this coincides $\Psi^\infty_{\rho,\delta}$.

\subsection{The Ruzhansky-Turunen difference operators $\Delta_q$}

Here we recall  the difference operators $\Delta_q$, called RT-difference operators, introduced by Michael Ruzhansky and Ville Turunen \cite{ruzhansky+turunen_bk} with slight modifications.
These RT-difference operators are different from our concept of difference operators explained in Section \ref{subsec_my_diff_op}. 
The notation is close but the context should always prevent any ambiguity.

\begin{definition}
\label{def_Delta_RT}
If $q\in \cD(G)$, 
then the corresponding \emph{RT-difference operator}
$\Delta_q$ is the operator acting on the space of Fourier transforms  $\cF_G(\cD'(G))$ via
$$
\Delta_q \widehat f = \cF_G \{q f\},
\quad f\in \cD'(G).
$$ 
\end{definition}

This definition is motived by the abelian case. Indeed, in the case of $\bR$, 
if we denote  the Euclidean Fourier transform of a (reasonable) function $g:\bR\to\bC$ by 
$$
\widehat g = \cF_\bR g,
\qquad
\widehat g(\zeta) = \int_\bR  g(x) e^{-i x\zeta} dx, 
\quad \zeta\in \bR,
$$
then $\partial_\xi^\alpha \widehat g= \cF_\bR \{(-ix)^\alpha g\}$.
The torus case is even more compelling:

\begin{ex}[Continuation of Example \ref{ex_diff_op_torus}]
\label{ex_diff_op_torus1}
In the case of the torus $\bT$, 
we see that the difference operator $\Delta_{e_\ell} $ associated with the one dimensional  representation $e_\ell$, $\ell\in \bZ$, is given on a Fourier transform $\widehat f\in \cF_G\cD'(G)$  by:
$$
\Delta_{e_\ell} \widehat f(e_m) 
=
\widehat f (m +\ell ) 
-
\widehat f( m  )
=
\int_0^{2\pi} f(x) e^{ixm} (e^{i\ell x} -1) \frac{dx}{2\pi}
=
\widehat{f q_\ell}(m)
=
\Delta_{q_\ell} \widehat f(m), 
$$
where $q_{\ell}(x)=e^{ i\ell x} -1$.
Hence $\Delta_{e_{\ell}} $ coincides with $\Delta_{q_\ell}$ on Fourier transforms.
In particular the backward and forward difference operators correspond to the function $q_{\pm 1}$.
\end{ex}

\medskip

We will adopt the following notation and vocabulary:

\begin{definition}
\label{def_Delta}
A \emph{collection $\Delta=\Delta_Q$ of RT-difference operators}
is the collection of RT-difference operators associated with the element of a finite ordered family $Q$ of smooth functions, 
that is:
$$
Q=Q_\Delta= \{q_{1,\Delta},\ldots, q_{n_\Delta,\Delta}\},
\quad
\Delta=\Delta_Q=\{\Delta_1,\ldots, \Delta_{n_\Delta}\},
$$
where $\Delta_{Q,j}=\Delta_{q_j}$.
\end{definition}
For such a collection $\Delta=\Delta_Q$, we set
$$
\Delta_Q^\alpha := 
\Delta_{Q,1}^{\alpha_1}
\ldots
\Delta_{Q,n_\Delta}^{\alpha_{n_\Delta}},
\quad
\mbox{for any multi-index}\
\alpha=(\alpha_1,\ldots,\alpha_{n_\Delta})\in \bN_0^{n_\Delta}.
$$
Note that $\Delta_Q^\alpha$ is the RT-difference operator corresponding to 
$$
q^\alpha_\Delta
:=q_1^{\alpha_1}\ldots q_{n_\Delta}^{\alpha_{n_\Delta}},
$$
and that 
this notation is consistent as any two RT-difference operators commute.

\medskip

Let us recall the definition of admissibility for a collection of RT-difference operators with a slight modification with respect to 
 \cite[Section 2]{ruzhansky+turunen+wirth}:

\begin{definition}
\label{def_admissible}
The collection $\Delta=\Delta_Q$ of RT-difference operators is \emph{admissible} 
when  the gradients  at $e_G$ of  the  functions in $Q$
  span the tangent space of $G$ (viewed as a manifold) at $e_G$:
$$
\rank (\nabla_{e_G} q_1,\ldots, \nabla_{e_G} q_{n_\Delta} )=n \quad (=\dim G).
$$

The collection $\Delta$ of RT-difference operators  is said to be \emph{strongly admissible} 
when it is admissible 
and furthermore when $e_G$ is the only common zero of the corresponding functions: 
$$
\{e_G\}=\cap_{j=1}^{n_\Delta} \{x\in G: q_j(x) =0\} .
$$
 \end{definition}

\begin{remark}
In the definition of admissibility in \cite[Section 2]{ruzhansky+turunen+wirth}, each gradient $\nabla q_j(e_G)$
is assumed to be non-zero so that the RT-difference operator is of order one (in the sense of Definition \ref{def_q_vanish_order}).
We do not assume this here 
hence our definition might appear to be more general.
However from a strongly admissible collection in the sense of Definition \ref{def_admissible},
we can always extract one which is admissible in the sense of \cite[Section 2]{ruzhansky+turunen+wirth}.
As proved in Theorem \ref{thm_Deltaeq+coincide}, 
they yield the same symbol classes.
The advantage in considering this relaxed definition lies in its convenience in various proofs.
\end{remark}

We can easily construct a strongly admissible collection:
\begin{lemma}
\label{lem_q0}
The exponential mapping is a diffeomorphism from a neighbourhood of 
0 in $\fg$ onto a neighbourhood of $e_G$.
We may assume that this neighbourhood is the ball $B(\epsilon_0)$
about $e_G$.
Let $\chi,\psi\in \cD(G)$ be valued in $[0,1]$ and  such that 
$$
\chi|_{B(\epsilon_0/2)} \equiv 1,
\quad
\chi|_{B(\epsilon_0)^c} \equiv 0,
\quad
\psi|_{B(\epsilon_0/8)} \equiv 0,
\quad
\psi|_{B(\epsilon_0/4)^c} \equiv 1.
$$
We fix a basis $\{x_1,\ldots,X_n\}$ of $\fg$.
For each $j=1,\ldots,n$, 
we define a function $p_j:G\to \bR$
$$
p_j(y) :=
\left\{\begin{array}{ll}
y_j  
&\mbox{if}\ B(\epsilon_0) \ni y =\exp(\sum_j y_j X_j),   \\
1&\mbox{if}\ y\not\in  \bar B(\epsilon_0) , \\
\end{array}\right.
$$
and then  a smooth function $q_j:=p_j \chi + \psi$.
The collection of RT-difference operators corresponding to $Q=\{q_j\}_{j=1}^n$ is strongly admissible.
\end{lemma} 

Note that $\nabla q_j (e_G)\not=0$ in Lemma \ref{lem_q0}.
 
 We can perform the following operations on collection of RT-difference operators:
\begin{lemma}
\label{lem_q_tildeq_barq_q*}
Let $\Delta_Q$ be a collection of RT-difference operators.
We denote by $\tilde \Delta=\Delta_{\tilde Q}$, $\bar \Delta =\Delta_{\bar Q}$ and $\Delta^*=\Delta_{Q^*}$ 
the collections of RT-difference operators with corresponding family of functions 
$\tilde Q:=\{q_{j,\Delta}(\cdot^{-1})\}_j$, 
$\bar Q:=\{\bar q_{j,\Delta}\}_j$
and $Q^*:=\{\bar q_{j,\Delta}(\cdot^{-1})\}_j$.

If $\Delta_Q$ is strongly admissible, 
then so are $\tilde \Delta$, $\bar \Delta$ and $\Delta^*$ 
\end{lemma}
 
\subsection{The Ruzhansky-Turunen classes of symbols}

Let us recall the  symbol classes  
introduced by M. Ruzhansky and V. Turunen \cite{ruzhansky+turunen_bk}.

\begin{definition}
\label{def_Smrhodelta_RT}
Let $\Delta=\Delta_Q$ be a collection of RT-difference operators.
A smooth symbol $\sigma=\{\sigma(x,\pi), (x,\pi)\in G\times\Gh\}$
is in  $S^m_{\rho,\delta}(G,\Delta)$ 
when for each $\alpha \in \bN_0^{n_\Delta}$ and $D\in \Diff^b$
there exists $C>0$  such that
\begin{equation}
\label{eq_def_SmrhodeltaRT}
\forall (x,\pi)\in (x,\Gh)\qquad
 \| X^\beta_x \Delta_Q^\alpha \sigma(x,\pi)\|_{\sL(\cH(\pi)}
\leq C (1+\lambda_\pi)^{\frac{m-\rho|\alpha| +\delta b}2}.
\end{equation}
\end{definition}

As the group $G$ is compact and $\sigma$ is smooth in $x$, 
it suffices to check \eqref{eq_def_SmrhodeltaRT} only for $D=X^\beta$, 
$\beta\in \bN_0^n$.

For $a,b\in \bN_0$, we set
$$
\|\sigma\|_{S^m_{\rho,\delta} (G,\Delta), a,b}
:=
\sup_{\substack{(x,\pi)\in G\times\Gh\\ |\alpha|\leq a, |\beta|\leq b}}
(1+\lambda_\pi)^{-\frac{m-\rho|\alpha| +\delta|\beta|}2}
 \| X^\beta_x \Delta_Q^\alpha \sigma(x,\pi)\|_{\sL(\cH(\pi)},
 \quad
\sigma \in S^m_{\rho,\delta} (G,\Delta).
 $$
 
If $x\in G$ is fixed (and if there is no ambiguity), we may use the notation $$
\|\sigma (x,\cdot) \|_{S^m_{\rho,\delta} (G,\Delta), a,b}
:=
\sup_{\substack{\pi\in \Gh\\ |\alpha|\leq a, |\beta|\leq b}}
(1+\lambda_\pi)^{-\frac{m-\rho|\alpha| +\delta|\beta|}2}
 \| X^\beta_x \Delta_Q^\alpha \sigma(x,\pi)\|_{\sL(\cH(\pi)}.
 $$

We denote by $\Psi^m_{\rho,\delta}(G,\Delta)$ the corresponding operator classes:
$$
\Psi^m_{\rho,\delta}(G,\Delta):=
\Op(S^m_{\rho,\delta}(G,\Delta)),
$$
and we define  $\|\cdot\|_{\Psi^m_{\rho,\delta} (G,\Delta), a,b}$ via
$$
\|T\|_{\Psi^m_{\rho,\delta} (G,\Delta), a,b}
:=
\|\sigma\|_{S^m_{\rho,\delta} (G,\Delta), a,b}
\quad\mbox{when}\ T=\Op(\sigma).
$$

It is not difficult to show that 
$\|\cdot\|_{S^m_{\rho,\delta} (G,\Delta), a,b}$ is a seminorm on
$S^m_{\rho,\delta}(G,\Delta)$ 
and that equipped with $\|\cdot\|_{S^m_{\rho,\delta} (G,\Delta), a,b}$,
$a,b\in \bN_0$, $S^m_{\rho,\delta}(G,\Delta)$ becomes a Fr\'echet
space.
The space $\Psi^m_{\rho,\delta} (G,\Delta)$ inherits the Fr\'echet topology.
One shows easily that the usual $\rho,\delta$-inclusions 
similar to \eqref{eq_rhodelta_inclusion} hold for the classes of symbols and operators.

Note that if a symbol has smooth entries 
and satisfies \eqref{eq_def_Smrhodelta}
then $\sigma$ is a smooth symbol in the sense of Definition \ref{def_smooth_symbol},
and the operator $\Op(\sigma)$ 
is a continuous operator $\cD(G)\to\cD(G)$, 
see Section \ref{subsec_symbol_quantisation}.

One important result of this paper is that the Ruzhansky-Turunen classes of operators coincide with our intrinsic pseudo-differential calculus:

\begin{theorem}
\label{thm_Deltaeq+coincide}
Let $m\in \bR$ and  $1\geq \rho\geq \delta\geq 0$.
\begin{enumerate}
\item 
If $\Delta$ and $\Delta'$ are two strongly admissible collections
of RT-difference operators, 
then the Fr\'echet spaces 
$S^m_{\rho,\delta}(G,\Delta)$
and $S^m_{\rho,\delta}(G,\Delta')$ coincide, 
that is, the vector spaces together with their topologies coincide.
\item 
Moreover, they coincide with the Fr\'echet space $S^m_{\rho,\delta}(G)$
defined in Definition \ref{def_Smrhodelta}.
\end{enumerate}
\end{theorem}
 
In other words,  the intrinsic calculus can be described 
with symbols in  $S^m_{\rho,\delta}(G,\Delta)$ for any strongly admissible collection $\Delta$ of  RT-difference operators.

The next section is devoted to 
the proof of Theorem \ref{thm_Deltaeq+coincide} 
and its corollary.

\subsection{Proof of Theorem \ref{thm_Deltaeq+coincide}}

The proof of Theorem \ref{thm_Deltaeq+coincide}  uses the following property:

\begin{lemma}
\label{lem_prop_indep_Delta}
Let  $q,q'\in \cD(G)$  be two functions 
such that $q/q'$ extends to a smooth function on $G$.
Let $s\in \bR$ and let $\sigma\in \Sigma(G)$ be  such that
$$
\exists C>0 \quad\forall \pi\in \Gh
\quad
\|\Delta_q \sigma(\pi)\|\leq C (1+\lambda_\pi)^{-\frac s2}.
$$
Then we have the same property for $\Delta_{q'} \sigma$ with the same $s$.
More precisely,  there exists $C'=C'_{q,q',s}>0$ (independent of $\sigma$) such that
   $$
\|(1+\lambda_\pi)^{\frac s2} \Delta_{q'} \sigma \|_{\sL(\cH_\pi)} 
\leq C'
\|(1+\lambda_\pi)^{\frac s2} \Delta_ q \sigma \|_{\sL(\cH_\pi)}.
$$
\end{lemma}

\begin{proof}[Lemma \ref{lem_prop_indep_Delta}]
Let $q,q'$ as in the statement. 
Let $\kappa \in \cD'(G)$ and $s\in \bR$. 
Denoting by $T_{q\kappa}$ and $T_{q'\kappa}$ the convolution operators with kernels $q\kappa$ and $q'\kappa$ respectively, we  have to prove 
\begin{equation}
\label{eq_lem_prop_indep_Delta}
\|T_{q\kappa}\|_{\sL(L^2,H^s)} \lesssim_{s,q,q'}
  \|T_{q'\kappa}\|_{\sL(L^2,H^s)}.
\end{equation}
Let $\phi\in \cD(G)$. We have 
$$
T_{q' \kappa}(\phi)(x)
=
\int_G \phi(y)  (q'\kappa)(y^{-1}x) dy
=
\int_G \phi(y) \psi_x(y) (q\kappa)(y^{-1}x)dy,
$$
where
the function $\psi_x\in \cD(G)$ is defined for each $x\in G$
via
$$
\psi_x(y):= \frac{q'}q(y^{-1}x), 
\quad y\in G.
$$
Then
\begin{eqnarray*}
\int_G |T_{q' \kappa}(\phi)(x)|^2 dx
\leq
\int_G \sup_{x_1\in G}
\left| \int_G \phi(y) \psi_{x_1}(y) (q\kappa)(y^{-1}x)dy\right|^2
dx\\
\lesssim
\int_G \sum_{|\gamma|\leq \frac n2 +1}
\int_G \left|X_{x_1}^\gamma \int_G \phi(y) \psi_{x_1}(y) (q\kappa)(y^{-1}x)dy\right|^2 dx_1
dx,
\end{eqnarray*}
having used the Sobolev inequalities (cf. Lemma \ref{lem_sob_embedding}). 
We have obtained
\begin{eqnarray*}
\|T_{q' \kappa}(\phi)\|_{L^2(G)}^2
&\lesssim&
 \sum_{|\gamma|\leq \frac n2 +1}
\int_G \left\|T_{q\kappa}(\phi X_{x_1}^\gamma\psi_{x_1})\right\|_{L^2(G)}^2 dx_1
\\
&\lesssim&
\|T_{q\kappa}\|_{\sL(L^2,H^s)}^2
 \sum_{|\gamma|\leq \frac n2 +1}
\int_G \|\phi X_{x_1}^\gamma\psi_{x_1}\|_{H^s}^2 dx_1.
\end{eqnarray*}
One can see easily that 
$$
\forall s\in \bN_0 
\quad \forall \phi,\psi\in \cD(G)
\quad
\|\phi \psi\|_{H^s} \lesssim_s 
\max_{|\alpha|\leq s} \|X^\alpha \psi\|_{L^\infty(G)}
\|\phi\|_{H^s},
$$
and thus by duality and interpolation, we also have the same property for any $s\in \bR$, with the slight modification that the maximum is now over $|\alpha|\leq |s|+1$.
Hence in our case, we obtain that 
$$
\sum_{|\gamma|\leq \frac n2 +1}
\int_G \|\phi X_{x_1}^\gamma\psi_{x_1}\|_{H^s}^2 dx_1
\lesssim_s
\max_{|\alpha|\leq |s|+n/2+2} \|X^\alpha \psi\|_{L^\infty(G)}
\ \|\phi\|_{H^s}.
$$
We have obtained \eqref{eq_lem_prop_indep_Delta}.
This concludes the proof of Lemma \ref{lem_prop_indep_Delta}.
\end{proof}

\begin{proof}[Theorem \ref{thm_Deltaeq+coincide}, Part 1.]
Let $\Delta$ be a strongly admissible collections of  difference operators with corresponding functions $q_1,\ldots,q_{n_\Delta}$.
Up to reordering $\Delta$, we may assume that the rank of $(\nabla_{e_G} q_1,\ldots, \nabla_{e_G} q_n )$ is $n=\dim G$.
Furthermore the basis of $\fg$ is chosen to be
$(X_1,\ldots, X_n)=(\nabla_{e_G} q_1,\ldots, \nabla_{e_G} q_n )$.
For each $q_j$, $j=1,\ldots,n$, 
we use the notation of Lemma \ref{lem_q0} to  construct $q_{j,0}:=p_j \chi + \psi$.
We adapt the argument of Lemma \ref{lem_q0} for the other functions. That is for $j>n$, we know that $\nabla_{e_G} q_j$ 
may be written as a linear combination $\sum_{\ell=1}^n c_\ell^{(j)} \nabla_{e_G} q_\ell$ and we define then  
$$
p_j(y) :=
\left\{\begin{array}{ll}
\sum_{\ell=1}^n c_\ell^{(j)}  y_\ell  
&\mbox{if}\ B(\epsilon_0) \ni y =\exp(\sum_j y_j X_j),   \\
1&\mbox{if}\ y\not\in  \bar B(\epsilon_0) , \\
\end{array}\right.
\quad\mbox{and}\quad q_{j,0}:=p_j \chi + \psi.
$$
Clearly the functions $q_{j,0}$, $j=1,\ldots,n_\Delta$, 
 yield a strongly admissible collections 
and for each $j =1,\ldots,n_\Delta$, the functions  
 $q_j/q_{j,0}$ and $q_{j,0}/q_j$ are smooth on $G$.
By Lemma \ref{lem_prop_indep_Delta}, 
the Fr\'echet spaces $S^m_{\rho,\delta}(G,\Delta)$
and  $S^m_{\rho,\delta}(G,\{\Delta_{q_{j,0}}\}_{j=1}^{n_\Delta})$
coincide for each $m,\rho,\delta$.
Moreover, 
the functions $ q_{j,0}$, $j=1,\ldots,n$ (only), 
 yield also a strongly admissible collections 
and for each $j >n$, the functions  
 $(\sum_{\ell=1}^n c_\ell^{(j)}  q_{\ell,0})/q_{j,0}$ are smooth on $G$.
By Lemma \ref{lem_prop_indep_Delta} again, 
the Fr\'echet spaces 
  $S^m_{\rho,\delta}(G,\{\Delta_{q_{j,0}}\}_{j=1}^{n_\Delta})$
and   $S^m_{\rho,\delta}(G,\{\Delta_{q_{j,0}}\}_{j=1}^{n})$
coincide for each $m,\rho,\delta$.
This shows that any class $S^m_{\rho,\delta}(G,\Delta)$ with $\Delta$ strongly admissible coincides with $S^m_{\rho,\delta}(G,\Delta_0)$ 
with a strongly admissible collection $\Delta_0$ constructed in Lemma \ref{lem_q0}. 

Let $\Delta_1$ and $\Delta_2$ be two collections constructed  in Lemma \ref{lem_q0} out of two bases $(X_j^{(1)})$ and $(X_j^{(2)})$ of $\fg$.
Let $P$ be  a $n\times n$ real matrix mapping 
$(X_j^{(1)})$ to $(X_j^{(2)})$. We construct the two corresponding collections of functions  $(q_j^{(1)})$ and $(q_j^{(2)})$ 
as in Lemma \ref{lem_q0}. We check easily that for each $j$, 
$(\sum_k P_{j,k} q_k^{(2)}) /q_j^{(1)}$ and
$(\sum_k (P^{-1})_{j,k} q_k^{(1)}) /q_j^{(2)}$ are smooth on $G$.
By Lemma \ref{lem_prop_indep_Delta}, 
the Fr\'echet spaces 
  $S^m_{\rho,\delta}(G,\Delta_1)$
and   $S^m_{\rho,\delta}(G,\Delta_2)$
coincide for each $m,\rho,\delta$.

Hence $S^m_{\rho,\delta}(G,\Delta)$ do not depend on a choice of strongly admissible collection $\Delta$.
This concludes the proof of the first part of Theorem \ref{thm_Deltaeq+coincide}.
\end{proof}

In the proof of the second part of Theorem \ref{thm_Deltaeq+coincide}, 
we will need the following lemma:
\begin{lemma}
\label{lem_linkDeltas}
\begin{enumerate}
\item 
Let $\tau\in \RepG$.
For any $\sigma\in \cF_G(\cD'(G))$, we have
$$
\Delta_\tau \sigma = \left[ \Delta_{q_{i,j}^{(\tau)}} \sigma\right]_{1\leq i,j\leq d_\tau},
$$
where the functions $q^{(\tau)}_{i,j}$ are the coefficients of a matrix realisation of $\tau -\id_{\cH_\tau}$, i.e. $q^{(\tau)}_{i,j}(x)=\tau_{i,j}(x)$ if $i\not=j$ and $q^{(\tau)}_{j,j}(x)=\tau_{j,j}(x)-1$.


\item We fix  a matrix realisation of each representation $\tau \in \FundG$, 
and we consider the functions $q^{(\tau)}_{i,j}$ as in Part 1.
We then consider the family $Q:=\{q^{(\tau)}_{i,j}, 1\leq i,j\leq d_\tau, \tau \in \FundG\}$.
The resulting collection $\Delta_Q$ of RT-difference operators 
is strongly admissible in the sense of Definition \ref{def_admissible}.
\end{enumerate}
\end{lemma}

\begin{proof}[Lemma \ref{lem_linkDeltas}]
One easily checks the first formula in the statement. Let us show the second part.
Each function $q_{i,j}^{(\tau)} \in \cD(G)$  vanishes at $e_G$ since $\tau(e_G)=\id_{\cH_\tau}$.
Its gradient at $e_G$ is 
\begin{equation}
\label{eq_pf_lem_linkDeltas}
\nabla_{e_G} q_{i,j}^{(\tau)} 
=(\tau_{i,j} (X_1),\ldots,\tau_{i,j}(X_n)),
\end{equation}
having kept the same notation for the representation $\tau$ of the group $G$ and the corresponding infinitesimal representation of the Lie algebra $\fg$.

Recall that $G$ can be written as the direct product of a torus with a semi-simple Lie group, that is,  $G=\bT^{n'}\times G_{ss}$
with $n_t=\dim \fg_{ab}$.
The set $\FundG$ can be written as the disjoint union of $\Fund (\bT^{n'})$
with $\Fund (G_{ss})$.
Let us define $Q$, $Q_{ab}$ and $Q_{ss}$
as the collections of functions $q^{(\tau)}_{i,j}, 1\leq i,j\leq d_\tau$, 
as $\tau$ runs over $\FundG$, $\Fund (\bT^{n'})$ and $\Fund (G_{ss})$ respectively.
Naturally  we can write the family 
$Q$ as the disjoint union of 
$Q_{ab}$ with $Q_{ss}$.
We write $\rank (\nabla_{e_G}  Q)
:=\rank \{\nabla_{e_G}  q, q\in Q\}$ and similarly for $Q_{ab}$ and $Q_{ss}$.

With the notation of Examples \ref{ex_diff_op_torus} and \ref{ex_diff_op_torus1},
the fundamental representations of the torus $\bT$ are $e_{\pm 1}$
and we see that $e_{\pm 1}'(0)=\pm 1$. 
This shows that $\rank(\nabla_{e_G}  Q_{ab})=n'$.
This implies the statement when $G=\bT^{n'}$ has no semi-simple part. 
If $G_{ss}$ is non trivial and $\rank(\nabla_{e_G}  Q)\not=n$, 
then $\rank(\nabla_{e_G}  Q_{ss}) < \dim \fg_{ss}$.
As any representation of $\fg_{ss}$ appears in the decomposition of some tensor products of fundamental representations, this together with \eqref{eq_pf_lem_linkDeltas} would imply that any representation of the semi-simple Lie algebra $\fg_{ss}$ is not injective and this is impossible.
Hence in any case, we have $\rank(\nabla_{e_G}  Q)=n$.

The zero set of $Q$ is 
$$
\cap _{q\in Q} \{x:q(x)=0\}
=
\cap_{\tau \in \FundG} \{x:\tau(x)-\id_{\cH_\tau} =0\}
=
\cap_{\tau \in \RepG } \{x:\tau(x)-\id_{\cH_\tau} =0\}.
$$
The inversion formula
\eqref{eq_inversion} implies that
if $x_0\in G$ is a zero of $Q$, then 
$f(x_0)=f(e_G)$ for any  function  $f\in\cC(G)$.
This implies that $x_0=e_G$ and $Q=\{e_G\}$.
This shows that $\Delta_Q$ is strongly admissible 
and concludes the proof of Lemma \ref{lem_linkDeltas}.
\end{proof}

\begin{proof}[Theorem \ref{thm_Deltaeq+coincide}, Part 2.]
If $\sigma\in S^m_{\rho,\delta}(G,\Delta)$ for some strongly admissible collection of RT-difference operator $\Delta$, 
then by Part 1., we may assume that $\Delta=\Delta_Q$ defined in Lemma \ref{lem_linkDeltas}.
The properties of the tensor easily implies for $\alpha\in \FundG^a$ 
$$
\|\Delta^\alpha \sigma\|_{\cH_\pi^{\otimes\alpha}}
\leq C_\alpha
\sum_{\alpha'\in \bN_0^{n_{\Delta_Q}}, |\alpha'|=a}
\|\Delta^{\alpha'}_Q \sigma\|_{\cH_\pi}.
$$
This shows that $\sigma\in S^m(\rho,\delta)(G)$.

Conversely, let $\sigma\in S^m(\rho,\delta)(G)$.
Then $\sigma$ is smooth by Lemma \ref{lem_op_cD2cD}.
Let $\Delta=\Delta_Q$ defined in Lemma \ref{lem_linkDeltas}.
The properties of the tensor easily implies for $\alpha' \in \bN_0^{n_{\Delta_Q}}$
$$
\|\Delta^{\alpha'}_Q \sigma\|_{\cH_\pi}
\leq C_{\alpha'}
\sum_{\alpha \in \FundG, |\alpha|=|\alpha'|}
\|\Delta^\alpha \sigma\|_{\cH_\pi^{\otimes\alpha}}
$$
This shows that $\sigma\in S^m(\rho,\delta)(G,\Delta)$.
The proof of Theorem \ref{thm_Deltaeq+coincide} is now complete.
\end{proof}

\medskip

From the proof of Theorem \ref{thm_Deltaeq+coincide}, 
we can obtain a corollary which was noticed by Ruzhansky, Turunen and Wirth via other means in \cite{ruzhansky+turunen+wirth}.
It concerns the Leibniz rule which is  a useful (and sometimes defining) property of derivatives.
The difference operators in the sense of Definitions 
\ref{def_Delta_tau} or \ref{def_Delta_RT} 
generally do not satisfy this exactly. 
Our difference operators satisfy the estimate \eqref{eq_leibniz_LinftyGh}.
In the Ruzhansky-Turunen viewpoint,
the following notion of Leibniz property was introduced in \cite{ruzhansky+turunen+wirth}:
\begin{definition}
\label{def_leibniz}
A  collection $\Delta=\Delta_Q$
of  RT-difference operators satisfies the \emph{Leibniz-like} property
when for any Fourier transforms $\widehat f_1$ and $\widehat f_2$
(with $f_1, f_2\in \cD'(G)$)
$$
\Delta_{Q,j} (\widehat f_1 \widehat f_2)
=
\Delta_{Q,j} (\widehat f_1) \ \widehat f_2
+
\widehat f_1\ \Delta_{Q,j} (\widehat f_2)
+
\sum_{1\leq l,k \leq n_\Delta} c_{l,k}^{(j)}
\Delta_{Q,l} (\widehat f_1) \  \Delta_{Q,k} (\widehat f_2)
$$
for some coefficients $c_{l,k}^{(j)}\in \bC$ depending only on $l,k,j$ and $\Delta$.
 \end{definition}

Note that this is equivalent to saying that $Q=Q_\Delta$
 satisfyies:
\begin{equation}
\label{eq_def_leibniz_q}
q_j (xy)
=
q_j (x) 
+
q_j (y)
+
\sum_{1\leq l,k \leq n_\Delta} c_{l,k}^{(j)}
q_l (x) \  q_k (y).
\end{equation}

Recursively on any multi-index $\alpha\in \bN_0^{n_\Delta}$, 
if $\Delta$ satisfies the Leibniz-like property, then 
$$
\Delta_Q^\alpha (\widehat f_1 \widehat f_2) 
= \sum_{|\alpha|\leq |\alpha_1|+|\alpha_2|\leq 2|\alpha|}
c_{\alpha_1,\alpha_2}^{\alpha} 
\Delta_Q^{\alpha_1} (\widehat f_1) \ \Delta_Q^{\alpha_2} (\widehat f_2),
$$
for some coefficients $c_{\alpha_1,\alpha_2}^{\alpha} \in \bC$ depending only on $\alpha_1,\alpha_2,\alpha$ and $\Delta$, 
with $c_{\alpha,0}^{\alpha}=c_{0,\alpha}^{\alpha}=1$.

The proof of Theorem \ref{thm_Deltaeq+coincide} yields:
\begin{corollary}
\label{cor_choice_Delta}
A strongly admissible collection of  RT-difference operators which satisfies the Leibniz-like formula always exists. An example is the strongly admissible family $Q$ considered in Lemma \ref{lem_linkDeltas}.
\end{corollary}

\begin{proof}
We notice that the coefficients of $\tau-\id$ for any $\tau\in \RepG$ satisfies
$$
q_{i,j}^{(\tau)} (xy) = 
q_{i,j}^{(\tau)} (x) 
+q_{i,j}^{(\tau)} (y)
+\sum_{k=1}^{d_\pi}  q_{ik}(x) q_{kj}(y),
\qquad 1\leq i,j\leq d_\pi, \ x,y\in G,
$$
with the notation of Lemma \ref{lem_linkDeltas}
since $\tau(xy)=\tau(x)\tau(y)$.
This together with \eqref{eq_def_leibniz_q} shows the statement.
\end{proof}

\section{Properties of the  kernels}
\label{sec_kernel}

In this section, 
we show that
the kernels of the symbols we have considered 
can only have a singularity at the neutral element 
and we obtain estimates near this singularity.
We also show that these distribution may be approximated by smoother kernels. 

We will use the following property whose proof is 
provided in \ref{sec_multipliers}:

\begin{proposition}
\label{prop_mult_t}
Let $\Delta=\Delta_Q$ be a strongly admissible collection of   RT-difference operators.
For any $m\in \bR$ and muti-index $\alpha\in\bN_0^{n_\Delta} $,
there exists $d\in \bN_0$ and $C>0$
such that for all
$f\in C^d[0,\infty)$, $\pi\in\Gh$ and $ t\in (0,1)$,
 we have
$$
\|\Delta_Q^\alpha \{f(t\lambda_\pi)\}\|_{\sL(\cH_\pi)} 
\leq C t^{\frac {m}  2}
(1+\lambda_\pi)^{\frac{m -|\alpha|}2}
\sup_{\substack{\lambda\geq 0 \\ \ell =0,\ldots, d}}
(1+\lambda)^{-m +\ell} |\partial^\ell_\lambda f(\lambda)|,
$$
in the sense that if the supremum in the right hand-side is finite, 
then the left hand-side is finite in the inequality holds.
\end{proposition}

\subsection{Singularities of the kernels}

Let us show that the singularities of the convolution kernels in 
$\Psi^\infty_{\rho,\delta}$ can be located only at the neutral element in the following sense:
\begin{proposition}
\label{prop_kernel_regularity}
We consider the symbol class 
$S^m_{\rho,\delta}(G,\Delta)$
with $1\geq \rho\geq \delta\geq 0$, $\rho\not=0$,
and a collection $\Delta$ such that
if $\cap_{q\in \Delta}\{x\in G : q(x) =0\}=0$.

If  $\sigma\in S^m_{\rho,\delta}$,
then its associated kernel $(x,y)\mapsto \kappa_x(y)$ 
is smooth on $G\times (G\backslash\{e_G\})$.

If  $\sigma\in S^{-\infty}$ is smoothing, 
then its associated kernel $(x,y)\mapsto \kappa_x(y)$ 
is smooth on $G\times G$. The converse is true: 
if the kernel associated with a symbol is smooth on $G\times G$ 
(as a function of $(x,y)$)
then the symbol is smoothing, i.e. it is in $S^{-\infty}$.
\end{proposition}

The proof relies on  the following two lemmata and their corollary:
\begin{lemma}
\label{lem_kernel_inL2}
If  $\kappa \in \cD'(G)$ then 
$$
\|\kappa\|_{L^2(G)}\lesssim_s \sup_{\pi\in \Gh} 
(1+\lambda_\pi)^{\frac s2} \|\widehat \kappa\|_{\sL(\cH_\pi)},
\quad s> n / 2,
$$
in the sense that $\kappa \in L^2(G)$  when there exists $s>n/ 2$ 
such that  the right-hand side is finite.
\end{lemma}
\begin{proof}
By Corollary \ref{lem_sob_embedding} and its proof,
we have for $s>0$,
$$
\kappa = (\id+\cL)^{s/2} (\kappa *\cB_s)
\quad\mbox{thus}\quad
\widehat \kappa(\pi)=
(1+\lambda_\pi)^{s/2} 
\pi(\cB_s)\pi(\kappa).
$$
and, together with the Plancherel formula (see \eqref{eq_Plancherel}),
$$
\|\cB_s\|_{L^2(G)}^2
=\sum_{\pi\in \Gh} d_\pi \|\pi(\cB_s) \|_{HS(\cH_\pi)} ^2
<\infty \quad\mbox{whenever} \ s>n/2.
$$
The properties of the Hilbert-Schmidt operators and the Plancherel formula  yield
$$\|\kappa\|_{L^2(G)}^2
=
\sum_{\pi\in \Gh} d_\pi \|\widehat \kappa(\pi)\|_{HS(\cH_\pi)} ^2\\
\leq
\|\cB_s\|_{L^2(G)} 
\sup_{\pi\in \Gh}\|(1+\lambda_\pi)^{\frac s2} \widehat\kappa(\pi)\|_{\sL(\cH_\pi)}^2.
$$
This shows Lemma \ref{lem_kernel_inL2}.
\end{proof}

The following properties are straightforward.
See also Proposition \ref{prop_symbol_1stprop} for notation.

\begin{lemma}
\label{lem_symbol_kernel}
\begin{enumerate}
\item Let $\sigma$ be a smooth symbol
with associated kernel $\kappa_x$.

If $\Delta=\Delta_Q$ and $D\in \Diff^b$
then  the kernel associated with 
 $D_x \Delta_Q^\alpha \sigma \in S^{m-\rho |\alpha|+\delta b}_{\rho,\delta} (G)$
for any $\alpha\in \bN_0^{n_\Delta}$ is $q^\alpha D_x \kappa_x$.

The kernel associated with 
$\sigma^*$
is  kernel $\kappa_x(y)=\bar \kappa_x(y^{-1})$.
\item 
If $\sigma_1$ and $\sigma_2$ are smooth symbols 
with associated kernel $\kappa_{1x}$
and $\kappa_{2x}$, 
then the kernel of the symbol 
$\sigma=\sigma_1\sigma_2$ is
 $\kappa_x=\kappa_{2x}*\kappa_{1x}$.
\end{enumerate}
\end{lemma}

\begin{corollary}
\label{cor_lem_kernel_inL2_1}
If $\sigma\in S^m_{\rho,\delta}$ 
with $1\geq \rho\geq \delta\geq 0$
and $\Delta=\Delta_Q$ a collection of difference operators,
then  
for any differential operators 
$D_z\in \Diff^b$ and $D'_x\in \Diff^{b'}$, 
the function $D'_x D_z \{q^\alpha_\Delta (z) \kappa_x(z)\}$ is continuous on $G$ and bounded, up to a constant of $m,\rho,\delta,\Delta,b,b'$ by 
$ \sup_{\pi\in \Gh} \|\sigma(x,\pi)\|_{S^m_{\rho,\delta}(G,\Delta),|\alpha|, b'}$
as long as $b +m +n +\delta b' <  \rho |\alpha|$.
\end{corollary}

\begin{proof}
If $s\in \bR$, using Lemma \ref{lem_symbol_kernel} and
the properties of the Sobolev spaces,
we have:
\begin{eqnarray}
\label{eq_pf_cor_lem_kernel_inL2_1}
&&\| (\id+\cL)^{\frac s2} D \{q^\alpha_\Delta D'_x \kappa_x(\cdot)\}\|_{L^2(G)}
\\
&&\qquad
\lesssim_{s,D}
\| (\id+\cL)^{\frac {s+b}2}  \{q^\alpha_\Delta (z) D'_x \kappa_x(z)\}\|_{L^2(G)}\nonumber
\\
&&\qquad
\lesssim_{s'}
\sup_{\pi\in \Gh} 
(1+\lambda_\pi)^{\frac {s'+s+b}2}  \|D'_x\Delta_Q^\alpha \sigma(x,\pi)\|_{\sL(\cH_\pi)},\label{eq_pf_cor_lem_kernel_inL2_2}
\end{eqnarray}
by Lemma \ref{lem_kernel_inL2} with $s'>n/2$.
By the Sobolev inequality (cf. Lemma \ref{lem_sob_embedding}), 
the function $D_z \{q^\alpha_\Delta (z) \kappa_x(z)\}$ is continuous if 
there exists $s>n/2$ such that
\eqref{eq_pf_cor_lem_kernel_inL2_1} 
is finite
and this quantity also provides a bound for the supremum over $z$.
As $\sigma \in S^m_{\rho,\delta}(G,\Delta)$,
\eqref{eq_pf_cor_lem_kernel_inL2_2} is indeed finite when $s'+s+b \leq -m +\rho |\alpha| -\delta b'$
and it suffices that $n+b  +m +\delta b' <  \rho |\alpha|$. 
\end{proof}

Corollary \ref{cor_lem_kernel_inL2_1} clearly implies Proposition \ref{prop_kernel_regularity}.

\subsection{Approximations by nice kernels}

We have already seen that the kernel associated with a continuous symbol 
can be approximated by a smooth kernel in the sense of Lemma \ref{lem_approximation_kernel}.
In many proofs below, we will use the following slightly different version 
for the  symbols in $S^m_{\rho,\delta}$.

\begin{lemma}
\label{lem_approximation_kernel2}
Let $\chi\in \cD(\bR)$ be a given function valued in $[0,1]$ and such that $\chi\equiv 1$ on a neighbourhood of 0.
Let $\sigma \in S^m_{\rho,\delta}$ with associated kernel $\kappa_x$.
For each $\ell\in \bN$, 
we define the symbol $\sigma_\ell$ via
$$
\sigma_\ell(x,\pi) = \sigma(x,\pi)\chi (\ell^{-1}\lambda_\pi),
$$ 
Then $\sigma_\ell\in S^{-\infty}$ and for any $a,b\in \bN_0$, 
there exists $C=C_{G,m,a,b,\chi}$ such that
$$
\|\sigma_\ell\|_{S^m_{\rho,\delta}, a,b}
\leq C
\|\sigma\|_{S^m_{\rho,\delta}, a,b}.
$$
Moreover 
the kernel $(x,y)\mapsto \kappa_{\ell,x}(y)$ associated with $\sigma_\ell$ is smooth 
on $G\times G$
and for any $\beta\in \bN_0^n$, 
$X^\beta_x\kappa_{x,\ell}\to X^\beta_x\kappa_x$ in $\cD'(G)$ uniformly in $x\in G$ as $\ell\to\infty$.
\end{lemma}
\begin{proof}
By Proposition \ref{prop_mult_t},
$\chi (\ell^{-1}\lambda_\pi)$ is smoothing.
Thus  the properties of the symbol classes
 (see Proposition \ref{prop_symbol_1stprop})
implies the membership $\sigma_\ell \in S^{-\infty}$.
By Proposition \ref{prop_kernel_regularity}, 
  $(x,y)\mapsto \kappa_{\ell,x}(y)$ is smooth.
The estimates for the semi-norms follows easily from Proposition \ref{prop_symbol_1stprop} and \eqref{eq_leibniz_LinftyGh}.
The only point to prove is the convergence of the kernels.
 For this, we proceed by adapting the proof of Lemma \ref{lem_op_cD2cD}.
Setting $s=-\lceil n/2\rceil$, we have
\begin{eqnarray*}
\|X_x^\beta (\kappa_{\ell,x}-\kappa_x)\|_{H^{s-m-\delta|\beta|-1}}
&=&
\|X_x^\beta (\sigma_\ell-\sigma)(x,\cdot)\|_{h_{s-m-\delta|\beta|-1}(\Gh)}
\\
&=&
\|X_x^\beta \sigma(x,\cdot) (1-\chi)(\ell^{-1}\lambda_\pi)\|_{h_{s-m-\delta|\beta|-1}(\Gh)}
\\
&\lesssim& 
\|(1+\lambda_\pi)^{-\frac {m+1 +\delta|\beta|} 2} (1-\chi)(\ell^{-1}\lambda_\pi)
X_x^\beta \sigma(x,\cdot) 
\|_{L^\infty(\Gh)}.
\end{eqnarray*}
By hypothesis, for some $0<\epsilon_\chi<\Lambda$,
the function $\chi$ is identically equal to 1 on $[0,\epsilon_\chi]$
and to 0 on $[\Lambda,+\infty)$.
Consequently, $\chi(\ell^{-1}\lambda_\pi)=1$ whenever $\lambda_\pi \geq \epsilon_\chi \ell$ and we have:
\begin{eqnarray*}
\|X_x^\beta (\kappa_{\ell,x}-\kappa_x)\|_{H^{s-m -\delta|\beta|- 1}}
\lesssim 
\max_{\pi)\in \Gh : \lambda_\pi \geq \epsilon_\chi \ell}
\|(1+\lambda_\pi)^{-\frac {m +1+\delta|\beta|} 2} (1-\chi)(\ell^{-1}\lambda_\pi)
X_x^\beta \sigma(x,\cdot) 
\|_{\cH_\pi}.
\\
\lesssim 
(1+\epsilon_\chi \ell)^{-1}
\max_{\pi)\in \Gh }
\|(1+\lambda_\pi)^{-\frac {m +\delta|\beta|} 2} X_x^\beta \sigma(x,\cdot) 
\|_{\cH_\pi}.
\end{eqnarray*}
Taking the supremum over $x\in G$, we obtain:
$$
\max_{x\in G}
\|X_x^\beta (\kappa_{\ell,x}-\kappa_x)\|_{H^{s-m}}
\lesssim 
(1+\epsilon_\chi \ell)^{-1}
\|\sigma\|_{S^m_{\rho,\delta},0,|\beta|}<\infty.
$$
The properties of the Sobolev spaces easily implies the stated convergence
of the kernels. This concludes the proof of Lemma \ref{lem_approximation_kernel2}.
 \end{proof}

\subsection{Estimates for the kernel}

In this section, we study the behaviour of the kernels near the origin.
More precisely, we show:

\begin{proposition}
\label{prop_estimate_kernel}
Let $\sigma\in S^m_{\rho,\delta}$
with $1\geq \rho\geq \delta\geq 0$, $\rho\not=0$.
Then its associated kernel $(x,y)\mapsto \kappa_x(y)
\in \cC^\infty(G\times (G\backslash\{e_G\})$ satisfies the following estimates:
\begin{itemize}
\item if $n+m>0$ then 
there exists 
$C$ and $a,b\in \bN$ (independent of $\sigma$) such that
$$
|\kappa_x(y)|\leq C 
\sup_{\pi\in \Gh} \|\sigma(x,\pi)\|_{S^m_{\rho,a,b}} |y|^{-\frac{n+m}\rho}.
$$
\item if $n+m=0$ then 
there exists 
$C$ and $a,b\in \bN$ (independent of $\sigma$) such that
$$
|\kappa_x(y)|\leq C \sup_{\pi\in \Gh} \|\sigma(x,\pi)\|_{S^m_{\rho,a,b}} 
|\ln |y||.
$$
\item if $n+m<0$ then $\kappa_x$ is continuous on $G$ and bounded  
$$
|\kappa_x(y)|\lesssim_m \sup_{\pi\in \Gh} \|\sigma(x,\pi)\|_{S^m_{\rho,0,0}}.
$$
\end{itemize}
\end{proposition}

By Lemma \ref{lem_symbol_kernel}, 
we also obtain similar properties for any derivatives in $x$ and $y$ of $\kappa_x(y)$ multiplied  by a smooth function $q$.

First we need to understand a `dyadic piece' of a symbol in the calculus:
\begin{lemma}
\label{lem_dyadicpiece}
Let $\sigma\in S^m_{\rho,\delta}$
with $1\geq \rho\geq\delta\geq0$.
Let $\eta\in \cD(\bR)$.
For any $t\in (0,1)$ we define the symbol $\sigma_t$
via $\sigma_t(x,\pi):=\sigma(x,\pi) \eta(t\lambda_\pi)$.
Then  for any $m_1\in \bR$ we have
$$
\|\sigma_t\|_{S^{m_1}_{\rho,\delta},a,b}
\leq C  
\|\sigma\|_{S^{m}_{\rho,\delta},a,b} t^{\frac{m_1-m}2}
$$
where $C=C_{m,m_1,a,b,\eta}$ does not depend on $\sigma$ or $t\in (0,1)$.
\end{lemma}

\begin{proof}[Lemma \ref{lem_dyadicpiece}]
This follows easily from the Leibniz property \eqref{eq_leibniz_LinftyGh}
together with Proposition \ref{prop_mult_t}  for the strongly admissible collection of RT-difference operators given in Lemma \ref{lem_linkDeltas}.
We naturally have used the equivalence of description of the symbols,
cf. Theorem \ref{thm_Deltaeq+coincide}.
\end{proof}

\begin{proof}[Proposition \ref{prop_estimate_kernel}]
The case $n+m<0$ follows readily from Corollary \ref{cor_lem_kernel_inL2_1} .
Hence we just have to study the case $m+n\geq0$.
We fix a dyadic decomposition of $\spec(\cL)$:
we choose two functions $\eta_0,\eta_1\in \cD(\bR)$ supported in 
$[-1,1]$ and $[1/2,2]$ respectively, both valued in $[0,1]$
and satisfying 
\begin{equation}
\label{eq_dyadic_dec}
\forall \lambda\geq 0\qquad
\sum_{\ell=0}^\infty \eta_\ell (\lambda)=1,
\quad\mbox{where for each}\ \ell\in \bN, \quad
\eta_\ell(\lambda):=\eta_1(2^{-(\ell-1)} \lambda).
\end{equation}
For each $\ell\in \bN_0$,
we set $\sigma_\ell(x,\pi) =\sigma (x,\pi) \eta_\ell(\lambda_\pi)$
and we denote by $\kappa_x$ and $\kappa_{\ell,x}$ the kernels associated with $\sigma$ and $\sigma_\ell$.
By Proposition \ref{prop_mult},  each symbol $\eta_\ell(\lambda_\pi)$ is smoothing, 
thus each $\sigma_\ell$ is also smoothing by Corollary \ref{cor_algebra_of_symbol}.
By Proposition \ref{prop_kernel_regularity}, 
the mapping
$(x,y)\mapsto\kappa_x(y)$ is smooth 
on $G\times (G\backslash\{e_G\})$
and
$\eta_\ell(\cL)\delta_e$
is smooth on $G$ thus $(x,y)\mapsto\kappa_{\ell,x}(y) = \kappa_x * (\eta_\ell(\cL)\delta_e)$ is in fact smooth on $G$.

One can easily show the convergence in $\cC^\infty(G\backslash\{e_G\})$ of 
$$
\kappa_x(y) = \lim_{N\to\infty} \sum_{\ell=0}^N \kappa_{\ell,x}.
$$
and the (possibly unbounded) summation, 
$$
\forall y\in G\backslash\{e_G\}\qquad |\kappa_x(y)| 
\leq \sum_{\ell=0}^\infty |\kappa_{\ell,x}(y)|.
$$

We suppose that a strongly admissible collection $\Delta$ has been fixed.
Applying Corollary \ref{cor_lem_kernel_inL2_1} and its proof
for any $\alpha\in \bN_0^n$ (but no $x$-derivatives), 
for any $m_1\in \bR$, 
whenever $m_1 +n  <  \rho |\alpha|$
we have 
$$
\sup_{z\in G} |q^\alpha(z) \kappa_{\ell,x} (z)|
\lesssim 
\sup_{\pi\in \Gh} \| \sigma_\ell (x,\pi)\|_{S^{m_1}_{\rho,\delta},|\alpha|,0}
\lesssim 
 \|\sigma\|_{S^{m}_{\rho,\delta},|\alpha|,0}
 2^{-(\ell-1)\frac{m_1-m}2 },
$$
by Lemma \ref{lem_dyadicpiece}.
As in Lemma \ref{lem_q0},
the strong admissibility implies 
$$
\forall z\in G, a\in 2\bN_0,\quad
|z|^a \lesssim_{\Delta,a} \sum_{|\alpha|=a} |q^\alpha (z)|.
$$
Hence  for any $a\in 2\bN_0$
and $m_1\in \bR$ satisfying $m_1+n<\rho a$,
we have obtained:
\begin{equation}
\label{eq_prop_estimate_kernel}
|z|^{a} |\kappa_{\ell,x} (z)|
\lesssim 
  \|\sigma\|_{S^{m}_{\rho,\delta},a,0}
  2^{\ell \frac{m-m_1}2 }.
\end{equation}

We may assume $|z|<1$ and choose $\ell_0\in \bN_0$ such that 
$$
|z|\sim 2^{-\ell_0}
\quad\mbox{in the sense that}\quad 2^{-\ell_0}\leq |z|< 2^{-\ell_0+1}.
$$

\medskip

\noindent\textbf{Case of $m+n>0$.}
For $\ell\leq \ell_0$, 
 we choose the real number $m_1\in \bR$ and the integer 
even  $a\in 2\bN_0$ to be such that 
\begin{equation}
\label{eq_relation_m_m1_a}
\frac{m+n} \rho > a\geq \frac{m+n}\rho -2
\qquad\mbox{and}\qquad
\frac{m-m_1}{2} =\frac{m+n}\rho -a.
\end{equation}
Hence $m>m_1$ so 
$$
\sum_{\ell=0}^{\ell_0} |\kappa_{\ell,x} (z)|
\lesssim 
  \|\sigma\|_{S^{m}_{\rho,\delta},a,0}
  |z|^{-a} 2^{\ell_0 \frac {m-m_1}2}  
$$
with 
$$
    |z|^{-a} 2^{\ell_0 \frac {m-m_1}2}  \lesssim 
  |z|^{-a - \frac {m-m_1}{2}} 
   \lesssim 
 |z|^{-\frac{m+n}\rho}  . 
$$

For $\ell> \ell_0$, we replace $a,m_1$ by $a',m'_1$
where  $a'=a+2$
and  $m'_1$ satisfies the same relation as  \eqref{eq_relation_m_m1_a} 
with $a$ replaced with $a'$.  
This time $m<m_1$ so 
$$
\sum_{\ell>\ell_0} |\kappa_{\ell,x} (z)|
\lesssim 
  \|\sigma\|_{S^{m}_{\rho,\delta},a,0}
  |z|^{-a'} 2^{\ell_0 \frac {m-m'_1}2}  
$$
and again $|z|^{-a} 2^{\ell_0 \frac {m-m'_1}2}    \lesssim 
 |z|^{-\frac{m+n}\rho} $ 
This shows the statement  in the case $m+n>0$.

\medskip

\noindent\textbf{Case of $m+n=0$.}
For $\ell\leq \ell_0$, we choose $a=0$ and $m_1=m$
and proceed as above: 
$$
\sum_{\ell=0}^{\ell_0} |\kappa_{\ell,x} (z)|
\lesssim 
  \|\sigma\|_{S^{m}_{\rho,\delta},a,0}
 \ell_0
\lesssim 
  \|\sigma\|_{S^{m}_{\rho,\delta},a,0}
|\ln |z| | .
$$
For $\ell> \ell_0$, we choose $a=2$ and 
$m_1=m-4$ (as in \eqref{eq_relation_m_m1_a})
$$
\sum_{\ell>\ell_0} |\kappa_{\ell,x} (z)|
\lesssim 
  \|\sigma\|_{S^{m}_{\rho,\delta},a,0}
 |z|^{-a}   2^{\ell_0  \frac{m_1-m} 2}
  \lesssim 
  \|\sigma\|_{S^{m}_{\rho,\delta},a,0}
$$
so 
$$
|\kappa_x(z)|
\lesssim 
  \|\sigma\|_{S^{m}_{\rho,\delta},a,0}
(1+ |\ln |z| |)
\lesssim 
  \|\sigma\|_{S^{m}_{\rho,\delta},a,0}
 |\ln |z| |.
$$
This shows the statement  in the case $m+n=0$
and concludes the proof of Proposition \ref{prop_estimate_kernel}.
\end{proof}

\section{The calculus}
\label{sec_calculus}

In this section, we prove that 
 $\cup_{m\in \bR}\Psi^m_{\rho,\delta}$ 
 satisfies the properties for the adjoint 
and the composition, that is
 Parts 
  \eqref{item_def_pseudo-diff_calculus_product}
  and \eqref{item_def_pseudo-diff_calculus_adjoint} of Definition \ref{def_pseudo-diff_calculus}.
We will also obtain the usual properties of asymptotic expansions in the case $\rho\not=\delta$.

\subsection{Adjoint}
\label{subsec_adjoint}

This section is devoted to showing
\begin{proposition}
\label{prop_adjoint}
Let $1\geq \rho\geq\delta\geq 0$ and $m\in \bR$.
If $T\in \Psi^m_{\rho,\delta}$ then its formal adjoint $T^*$ is 
also in $\Psi^m_{\rho,\delta}$.
Moreover $T\mapsto T^*$ is continuous on $\Psi^m_{\rho,\delta}$.
\end{proposition}

By Lemma \ref{lem_approximation_kernel2}, 
we may assume that all the associated kernels are smooth on $G\times G$.
This justifies the following formal manipulations.
One computes easily that if $T=\Op(\sigma)\in \Psi^m_{\rho,\delta}$
with associated kernel $\kappa_x$
then $T^*$ has associated kernel $\kappa^{(*)}_x$
given by
\begin{equation}
\label{eq_kappa(*)}
\kappa_{x}^{(*)}(y)=\bar \kappa_{xy^{-1}}(y^{-1}).
\end{equation}
We denote its symbol by $\sigma^{(*)}$:
$$
T^*=\Op(\sigma^{(*)}).
$$
Note that the kernel $\kappa_x^{(*)}$ and the symbol $\sigma^{(*)}$ 
are usually different from the kernel 
$\kappa_x^*:y\mapsto \bar \kappa_{x}(y^{-1})$ and its associated symbol $\sigma^*$ (unless, for instance, the symbol does not depend on $x$)
but we have $\kappa_{x}^{(*)}(y)= \kappa^*_{xy^{-1}}(y)$.

\begin{proof}[Proposition \ref{prop_adjoint}]
If $\Delta=\Delta_Q$ is a collection of RT-difference operators, 
given the formula in \eqref{eq_kappa(*)} for 
the kernel of $\sigma^{(*)}$,
one checks easily that 
\begin{equation}
\label{eq_pf_adj_XDelta}
\tilde X^\beta_x\Delta_Q^\alpha \sigma^{(*)}(x,\pi)
=
\{\tilde X^{\beta_0}_x (\Delta_Q^*)^{\alpha_0} \sigma\}^{(*)}(x,\pi)
\quad\mbox{for all multi-indices} \ \alpha_0,\beta_0.
\end{equation}
Thus, by Lemmata \ref{lem_q_tildeq_barq_q*}
and  \ref{lem_symbol_kernel}, it suffices to show that there exists $b\in \bN_0$ such that
\begin{equation}
\label{eq_pf_adj_suff}
\|\sigma^{(*)}(x,\pi)\|_{\sL(\cH_\pi)}\leq 
C \|\sigma\|_{S^m_{\rho,\delta}(G,\Delta),0,b}
(1+\lambda_\pi)^{\frac m2}.
\end{equation}

From  \eqref{eq_kappa(*)}, it is easy to check  
using integration by parts that we have
\begin{eqnarray*}
\{\lambda_\pi^N \sigma(x,\pi)\}^{(*)}
&=&
\int_G (\cL^N \bar \kappa_{x_1})(y_1)|_{x_1=xy^{-1}, y_1=y^{-1}}
\pi(y)^* dy
\\&=&
\sum_{|\beta_1|+|\beta_2|=2N}
c_{\beta_1,\beta_2}
\int_G \tilde X^{\beta_1}_{x_1=xy^{-1}} 
\bar \kappa_{x_1}(y^{-1})
\pi(y)^* \pi(X)^{\beta_2} dy
\\&=&
\sum_{|\beta_1|+|\beta_2|=2N}
c_{\beta_1,\beta_2}
\{\tilde X^{\beta_1}_x\sigma(x,\pi)\}^{(*)} \pi(X)^{\beta_2}.
\end{eqnarray*}
Thus 
\begin{equation}
\label{eq_pf_adj_lambdaNsigma}
\|\{\lambda_\pi^N \sigma(x,\pi)\}^{(*)}\|_{\sL(\cH_\pi)}
\lesssim 
\sum_{|\beta_1|+|\beta_2|=2 N}
(1+\lambda_\pi)^{\frac {|\beta_2|}2}
\|\{\tilde X^{\beta_1}_x\sigma(x,\pi)\}^{(*)}\|_{\sL(\cH_\pi)}.
\end{equation}

Now suppose  that one can write 
$\sigma(x,\pi) = (1+\lambda_\pi)^N \tau(x,\pi)$
with $N\in \bN_0$ 
and $\tau \in S^{m-2N}_{\rho,\delta}$ satisfying \eqref{eq_pf_adj_suff} 
with order $m-2N$.
Then applying \eqref{eq_pf_adj_lambdaNsigma} to $\tau$ yields
\begin{eqnarray*}
&&\|\sigma(x,\pi)^{(*)}\|_{\sL(\cH_\pi)}
=
\|\{(1+\lambda_\pi)^N \tau(x,\pi)\}^{(*)}\|_{\sL(\cH_\pi)}\\
&&\qquad\lesssim 
\|\tau\|_{S^{m-2N}_{\rho,\delta},0,b}
\sum_{|\beta_1|+|\beta_2|\leq 2 N}
(1+\lambda_\pi)^{\frac {|\beta_2|}2 +\frac{m-2N+\delta|\beta_1|}2 }
\lesssim 
\|\sigma\|_{S^{m}_{\rho,\delta},0,b'}
(1+\lambda_\pi)^{\frac{m}2 }.
\end{eqnarray*}
and $\sigma$ also satisfies $\eqref{eq_pf_adj_suff}$.
This shows that 
 it suffices to prove \eqref{eq_pf_adj_suff} for $m<<0$
and we may assume $m<-n$.

From  \eqref{eq_kappa(*)}, we  also observe  that
the kernel of $\sigma^{(*)}$ is continuous and bounded 
in $(x,y)\in G\times G$
by Corollary \ref{cor_lem_kernel_inL2_1}
provided that $m<-n$.
Thus, by 
\eqref{eq_cF_L1},
we have the  crude implication:
\begin{equation}
\label{eq_pf_adj_rho=rho1}
m <-n
\Longrightarrow
\sup_{\pi\in \Gh,  x\in G}
\|\sigma^{(*)}(x,\pi)\|_{\sL(\cH_\pi)}
\lesssim \|\sigma\|_{S^m_{\rho,\delta},0,0}.
\end{equation}

We can now start the proof of \eqref{eq_pf_adj_suff} for $m<-n$.
We consider a dyadic decomposition of $\spec(\cL)$, 
for instance the same as for the proof of Proposition \ref{prop_estimate_kernel}:
we choose two functions $\eta_0,\eta_1\in \cD(\bR)$ supported in 
$[-1,1]$ and $[1/2,2]$ respectively, both valued in $[0,1]$
and satisfying \eqref{eq_dyadic_dec}. 
We set $\sigma_\ell(x,\pi) =\sigma (x,\pi) \eta_\ell(\lambda_\pi) $
for each $\ell\in \bN_0$.
We easily obtain
$$
\|\sigma^{(*)}(x,\pi)\|_{\sL(\cH_\pi)}
\leq
\sum_{\ell=0}^\infty
\|\sigma_\ell^{(*)}(x,\pi)\|_{\sL(\cH_\pi)},
$$
with possibly infinite non-negative quantities.
Combining \eqref{eq_pf_adj_rho=rho1} and  Lemma \ref{lem_dyadicpiece}
already provides an estimate for each $\|\sigma_\ell^{(*)}(x,\pi)\|_{\sL(\cH_\pi)}$, 
$\ell\in \bN_0$.
This can be improved for $\ell>0$ in the following way.
For any $N\in \bN$ and each $\ell\in \bN$, we define 
$$
\tilde \eta_1^{(N)} (\lambda): = \lambda^{-N} \eta_1(\lambda),
\qquad
\tilde \eta_\ell^{(N)} (\lambda) = \tilde \eta_1^{(N)} (2^{-(\ell-1)}\lambda)
\quad \mbox{and}\quad
\tilde \sigma_\ell^{(N)} (x,\pi) = \tilde \eta_\ell^{(N)} (\lambda_\pi)
\sigma_\ell (x,\pi).
$$
Simple manipulations show 
\begin{equation}
\label{eq_sigmaell_tildesigmaell}
\sigma_\ell(x,\pi) = 2^{-(\ell-1) N}\lambda_\pi^N \tilde \sigma_\ell^{(N)}  (x,\pi),
\end{equation}
and using \eqref{eq_pf_adj_lambdaNsigma}:
\begin{eqnarray*}
\|\{\lambda_\pi^N \tilde\sigma_\ell^{(N)} (x,\pi)\}^{(*)}\|_{\sL(\cH_\pi)}
&\lesssim &
\sum_{|\beta_1|+|\beta_2|= 2N}
(1+\lambda_\pi)^{\frac {|\beta_2|}2}
\|\{\tilde X^{\beta_1}_x\tilde \sigma_\ell^{(N)}  (x,\pi)\}^{(*)}\|_{\sL(\cH_\pi)}
\\
&\lesssim &
\|\sigma\|_{S^m_{\rho,\delta},0,2N}
(1+\lambda_\pi)^{N}
2^{-(\ell-1) \frac{m_1-m}2},
\end{eqnarray*}
by \eqref{eq_pf_adj_rho=rho1} and Lemma \ref{lem_dyadicpiece}, 
for any choice of $m_1<-n$.
Hence we have obtained 
$$
\forall \ell\in \bN_0\qquad
\| \sigma_\ell (x,\pi)^{(*)}\|_{\sL(\cH_\pi)}
\lesssim 
\|\sigma\|_{S^m_{\rho,\delta},0,2N}
(1+\lambda_\pi)^{N}
2^{-\ell (N+\frac{m_1-m}2)},
$$
for any fixed $N,m_1$ satisfying $N\in \bN_0$ and $m_1<-n$.
Let us apply this for $N=N_1$, $m_1<-n$ if $\ell<\ell_0$, 
and for $N=N_2$, $m_2<-n$ if $\ell\geq \ell_0$
for $\ell_0$ to be chosen suitably with respect to $\pi$.
Setting $N=\max(N_1,N_2)$, we have obtained:
\begin{eqnarray*}
\|\sigma^{(*)}(x,\pi)\|_{\sL(\cH_\pi)}
\lesssim
\|\sigma \|_{S^{m}_{\rho,\delta},0,2N} 
\left(\sum_{\ell=0}^{\ell_0-1}
(1+\lambda_\pi)^{N_1}
2^{-\ell (N_1+\frac{m_1-m}2)}
+
\sum_{\ell=\ell_0}^{\infty}
(1+\lambda_\pi)^{N_2}
2^{-\ell (N_2+\frac{m_2-m}2)}\right)
\\
\lesssim
\|\sigma \|_{S^{m}_{\rho,\delta},0,2N} 
\left(
(1+\lambda_\pi)^{N_1}
2^{-\ell_0 (N_1+\frac{m_1-m}2)}
+
(1+\lambda_\pi)^{N_2}
2^{-\ell_0 (N_2+\frac{m_2-m}2)}\right)
\end{eqnarray*}
provided that $N_1,N_2\in \bN$ satisfy 
$N_1+\frac{m_1-m}2<0$ and 
$N_2+ \frac{m_2-m}2 >0$.
Now we choose $\ell_0\in \bN$ such that 
$2^{\ell_0}\sim (1+\lambda_\pi)$, in the sense that
$2^{\ell_0-1} \leq (1+\lambda_\pi)< 2^{\ell_0}$, 
together with 
$m_1= m_2=2m$,
 $N_1:= \lfloor \frac{m-m_2}2\rfloor$
 and $N_2:= \lceil \frac{m-m_2}2\rceil$.
This shows \eqref{eq_pf_adj_suff} for $m<-n/2$ and  concludes the proof of Proposition \ref{prop_adjoint}.
\end{proof}

\subsection{Composition}

This section is devoted to showing
\begin{proposition}
\label{prop_composition}
Let $1\geq \rho\geq \delta\geq 0$.
If $T_1\in \Psi^{m_1}_{\rho,\delta}$
and $T_2\in \Psi^{m_2}_{\rho,\delta}$,
 then the composition $T_1T_2$ is  in $\Psi^{m_1+m_2}_{\rho,\delta}$.
 Moreover the map $(T_1,T_2)\mapsto T_1T_2$ is continuous
 $\Psi^{m_1}_{\rho,\delta}\times \Psi^{m_2}_{\rho,\delta}\to \Psi^{m_1+m_2}_{\rho,\delta}$.
\end{proposition}

We proceed in a similar way as in Section \ref{subsec_adjoint}.
One computes easily that if $T_i=\Op(\sigma_i)\in \Psi^m_{\rho,\delta}$, 
with associated kernel $\kappa_{i,x}$
$i=1,2$,
(which we assume smooth on $G\times G$)
then $T_1T_2$ has associated kernel $\kappa_x$
given by
\begin{equation}
\label{eq_kernel_composition}
\kappa_{x}(y)=\int_G \kappa_{2,xz^{-1}}(yz^{-1})  \kappa_{1,x}(z) dz, 
\quad x,y\in G,
\end{equation}
and symbol
\begin{equation}
\label{eq_symbol_composition}
\sigma(x,\pi) :=\sigma_1 \circ \sigma_2 (x,\pi) :=
\int_G \kappa_x(z) \pi(z)^* dz=
\int_G \kappa_{1x}(z) \pi(z)^* \sigma_2(xz^{-1},\pi) dz.
\end{equation}
Note that $\kappa_x$ and $\sigma=\sigma_1\circ\sigma_2$ are usually different from 
$\kappa_{2x} * \kappa_{1x}$ and $\sigma_1\sigma_2$,
unless, for instance, $\sigma_2$ does not depend on $x$.

\begin{proof}[Proposition \ref{prop_composition}]
Let $\Delta=\Delta_Q$ be  a strongly admissible collection of difference operators satisfies the Leibniz like property
(see Theorem \ref{thm_Deltaeq+coincide} and Corollary \ref{cor_choice_Delta}).
This Leibniz property (see \eqref{eq_def_leibniz_q})
 together with  the formulae in \eqref{eq_kernel_composition} and \eqref{eq_symbol_composition}, 
imply easily
\begin{equation}
\label{eq_pf_comp_XDelta}
\tilde X^{\beta_0}_x \Delta_Q^{\alpha_0} \sigma
=
\sum_{\substack{|\beta_1|+|\beta_2|=|\beta_0|\\
|\alpha_0|\leq |\alpha_1|+|\alpha_2|\leq 2|\alpha_0|}}
c_{\alpha_1,\alpha_2,\beta_1,\beta_2}
(\tilde X^{\beta_1}_x\Delta_Q^{\alpha_1}\sigma_1)\circ  (\tilde X^{\beta_2}_x\Delta_Q^{\alpha_2} \sigma_2),
\end{equation}
Hence it suffices to show that there exists $b\in \bN_0$ such that
\begin{equation}
\label{eq_pf_comp_suff}
\|\sigma_1\circ \sigma_2 (x,\pi)\|_{\sL(\cH_\pi)}\leq 
C \|\sigma\|_{S^m_{\rho,\delta},0,b}
(1+\lambda_\pi)^{\frac {m_1+m_2}2}.
\end{equation}

From \eqref{eq_kernel_composition} and \eqref{eq_symbol_composition},
it is easy to check using integration by parts that
\begin{eqnarray*}
(\lambda_\pi^N \tau_1) \circ \sigma_2 (x,\pi)  
&=&
\int_G (\tilde \cL^N \kappa_{1x})(z) \pi(z)^* \sigma_2(xz^{-1},\pi) dz\\
&=&
\sum_{|\beta_1|+|\beta_2|=2N}
c_{\beta_1,\beta_2}
\int_G \kappa_{1x}(z) \pi(z)^* \pi(X)^{\beta_1}   X_{x_1=xz^{-1}}^{\beta_2} \sigma_2(x_1,\pi) dz\\
&=&
\sum_{|\beta_1|+|\beta_2|=2N}
c_{\beta_1,\beta_2}
\tau_1\circ ( \pi(X)^{\beta_1}   X_{x}^{\beta_2} \sigma_2)(x,\pi).
\end{eqnarray*}
Thus 
\begin{equation}
\label{eq_pf_comp_lambdaNsigma1}
\|(\lambda_\pi^N \tau_1) \circ \sigma_2 (x,\pi)\|_{\sL(\cH_\pi)}
\lesssim 
\sum_{|\beta_1|+|\beta_2|=2 N}
\| \tau_1\circ ( \pi(X)^{\beta_1}   X_{x}^{\beta_2} \sigma_2)(x,\pi)
\|_{\sL(\cH_\pi)}.
\end{equation}

Now suppose  that one can write 
$\sigma_1(x,\pi) = (1+\lambda_\pi)^N \tau_1(x,\pi)$
with $N\in \bN_0$ 
and that $\tau_1 \in S^{m_1-2N}_{\rho,\delta}$ satisfies\eqref{eq_pf_adj_suff} 
with order $m_1-2N$ for any $\sigma_2\in S^{m_2}_{\rho,\delta}$.
Then applying \eqref{eq_pf_comp_lambdaNsigma1} to $\tau_1$ yields
that $\sigma$ also satisfies $\eqref{eq_pf_comp_suff}$.
This shows that 
 it suffices to prove \eqref{eq_pf_comp_suff} for $m_1<<0$
and we may assume $m_1<-n$.

From \eqref{eq_symbol_composition}, we  also observe  that
$$
\|\sigma(x,\pi)\|_{\sL(\cH_\pi)}
\leq
\sup_{x_1\in G} \|\sigma_2(x_1,\pi)\|
\int_G |\kappa_{1,x}(z) dz|
$$.
By Corollary \ref{cor_lem_kernel_inL2_1},
we have the  crude implication:
\begin{equation}
\label{eq_pf_comp_rho=rho1}
m_1 <-n
\Longrightarrow
\|\sigma_1\circ \sigma_2 (x,\pi)\|_{\sL(\cH_\pi)}
\lesssim 
\|\sigma_2\|_{S^{m_2}_{\rho,\delta},0,0}
\|\sigma_1\|_{S^{m_1}_{\rho,\delta},0,0}
(1+\lambda_\pi)^{\frac {m_2}2}.
\end{equation}

We can now start the proof of \eqref{eq_pf_comp_suff} for $m<-n$.
We consider the same dyadic decomposition of $\spec(\cL)$
as in the first proof of Proposition \ref{prop_adjoint}:
we choose two functions $\eta_0,\eta_1\in \cD(\bR)$ supported in 
$[-1,1]$ and $[1/2,2]$ respectively, both valued in $[0,1]$
and satisfying \eqref{eq_dyadic_dec}. 
We set $\sigma_{1,\ell}(x,\pi) =\sigma_1 (x,\pi) \eta_\ell(\lambda_\pi) $
for each $\ell\in \bN_0$.
For any $N\in \bN$, we also define $\tilde \eta_1^{(N)} (\lambda): = \lambda^{-N} \eta_1(\lambda)$, and the corresponding 
$\tilde \eta_\ell^{(N)}$ and 
$\tilde \sigma_{1,\ell}^{(N)}$.

We easily obtain
$$
\|\sigma(x,\pi)\|_{\sL(\cH_\pi)}
\leq
\sum_{\ell=0}^\infty
\|\sigma_{1,\ell}\circ \sigma_2 (x,\pi)\|_{\sL(\cH_\pi)},
$$
with possibly infinite non-negative quantities.
Combining \eqref{eq_pf_comp_rho=rho1} and  Lemma \ref{lem_dyadicpiece}
already provides an estimate for each 
$\|\sigma_{1,\ell}\circ \sigma_2 (x,\pi)\|_{\sL(\cH_\pi)}$,
$\ell\in \bN_0$.
Using \eqref{eq_pf_comp_lambdaNsigma1}, we also have:
\begin{eqnarray*}
\|\{\lambda_\pi^N \tilde\sigma_{1,\ell}^{(N)}\}\circ \sigma_2 (x,\pi)\|_{\sL(\cH_\pi)}
\lesssim 
\sum_{|\beta_1|+|\beta_2|= 2N}
\|\tilde\sigma_{1,\ell}^{(N)}\circ ( \pi(X)^{\beta_1}   X_{x}^{\beta_2}  \sigma_2) (x,\pi)\|_{\sL(\cH_\pi)}
\\
\lesssim 
C_N(\sigma_1,\sigma_2)
\sum_{|\beta_1|+|\beta_2|= 2N}
(1+\lambda_\pi)^{\frac{m_2 +|\beta_1| +\delta|\beta_2|} 2}
2^{-\ell \frac{m'_1-m_1}2},
\end{eqnarray*}
by \eqref{eq_pf_comp_rho=rho1} and Lemma \ref{lem_dyadicpiece}, 
for any choice of $m'_1<-n$, 
with $C_N(\sigma_1,\sigma_2):=
\|\sigma_2\|_{S^{m_2}_{\rho,\delta},0,2N}
\|\sigma_1\|_{S^{m_1}_{\rho,\delta},0,0}$.
Hence, using \eqref{eq_sigmaell_tildesigmaell}, we have obtained 
$$
\forall \ell\in \bN_0\qquad
\| \sigma_{1,\ell}\circ \sigma_2  (x,\pi)\|_{\sL(\cH_\pi)}
\lesssim C_N(\sigma_1,\sigma_2)
(1+\lambda_\pi)^{N +\frac {m_2}2}
2^{-\ell (N+\frac{m'_1-m_1}2)},
$$
for any fixed $N,m'_1$ satisfying $N\in \bN_0$ and $m'_1<-n$.
Let us apply this for $N=N_1$, $m'_1<-n$ if $\ell<\ell_0$, 
and for $N=N_2$, $m''_1<-n$ if $\ell\geq \ell_0$
for $\ell_0$ to be chosen suitably with respect to $\pi$.
Setting $N=\max(N_1,N_2)$,  $\|\sigma(x,\pi)\|_{\sL(\cH_\pi)}$ is then bounded, up to a constant, by
\begin{eqnarray*}
C_N(\sigma_1,\sigma_2)
\left(\sum_{\ell=0}^{\ell_0-1}
(1+\lambda_\pi)^{N_1+\frac {m_2}2}
2^{-\ell (N_1+\frac{m'_1-m_1}2)}
+
\sum_{\ell=\ell_0}^{\infty}
(1+\lambda_\pi)^{N_2+\frac {m_2}2}
2^{-\ell (N_2+\frac{m''_1-m_1}2)}\right)
\\
\lesssim
C_N(\sigma_1,\sigma_2)
\left(
(1+\lambda_\pi)^{N_1+\frac {m_2}2}
2^{-\ell_0 (N_1+\frac{m'_1-m_1}2)}
+
(1+\lambda_\pi)^{N_2+\frac {m_2}2}
2^{-\ell_0 (N_2+\frac{m''_1-m_1}2)}\right)
\end{eqnarray*}
provided that $N_1,N_2\in \bN$ satisfy 
$N_1+\frac{m'_1-m_1}2<0$ and 
$N_2+ \frac{m''_1-m_1}2 >0$.
Now we choose $\ell_0\in \bN$ such that 
$2^{\ell_0}\sim (1+\lambda_\pi)$, in the sense that
$2^{\ell_0-1} \leq (1+\lambda_\pi)< 2^{\ell_0}$, 
together with 
$m'_1= m''_1=2m_1$,
 $N_1:= \lfloor \frac{m'_1-m_1}2\rfloor$
 and $N_2:= \lceil \frac{m''_1-m_1}2\rceil$.
This shows \eqref{eq_pf_comp_suff} for $m_1<-n/2$ and  concludes the proof of Proposition \ref{prop_composition}.
\end{proof}

\subsection{Asymptotic  expansions}

The analysis to prove the properties for the adjoint 
and the composition will also yield a familiar (but matrix valued) expansion
in the case $\rho>\delta$.
This section is devoted to understand the meaning of the expansion and the coefficients in it. 

For the asymptotic expansion, we first prove:
\begin{proposition}
\label{prop_asymptotic}
Let $\{\sigma_j\}_{j\in \bN_0}$ be a sequence of symbols such that
$\sigma_j\in S^{m_j}_{\rho,\delta}$
with $m_j$ strictly increasing to $-\infty$.
Then there exists $\sigma \in S^{m_0}_{\rho,\delta}$, 
unique modulo $S^{-\infty}$ such that 
\begin{equation}
\label{eq_asymptotic}
\forall M\in \bN 
\qquad
\sigma-\sum_{j=0}^M \sigma_j \in S^{m_{M+1}}_{\rho,\delta}.
\end{equation}
\end{proposition}

Under the hypotheses and conclusions of Theorem \ref{prop_asymptotic}, 
we write
$$
\sigma\sim \sum_j \sigma_j.
$$

\begin{proof}[Proposition \ref{prop_asymptotic}]
Let $\psi\in \cC^\infty(\bR)$ valued in $[0,1]$ satisfying 
$\psi\equiv 0$ on $(-\infty,1/2)$ and $\psi\equiv 1$ on $(1,\infty)$.
Let $\Delta=\Delta_Q$ be  a strongly admissible collection of difference operators satisfies the Leibniz like property
(see Theorem \ref{thm_Deltaeq+coincide} and Corollary \ref{cor_choice_Delta}).
Hence we have
\begin{eqnarray*}
&&\|\Delta_Q^\alpha X^\beta \{\sigma_j(x,\pi) \psi(t\lambda_\pi)\}\|_{\sL(\cH_\pi)}
\lesssim
\sum_{|\alpha_1|+|\alpha_2|=|\alpha|}
\|\Delta_Q^{\alpha_1} X^\beta \sigma_j(x,\pi) \|_{\sL(\cH_\pi)}
\|\Delta_Q^{\alpha_2}  \psi(t\lambda_\pi)\|_{\sL(\cH_\pi)}
\\&&\qquad\lesssim
\|\sigma_j\|_{S^{m_j}_{\rho,\delta}, |\alpha|,|\beta|}
\sum_{|\alpha_1|+|\alpha_2|=|\alpha|}
(1+\lambda_\pi)^{\frac{m-\rho|\alpha_1|+\delta |\beta|}2}
t^{\frac{m_2}2} (1+\lambda)^{\frac{m_2-|\alpha_2|}2},
\end{eqnarray*}
by Proposition \ref{prop_mult_t}.
We choose $m_2=m_0-m_j$ and obtain easily
$$
\|\Delta_Q^\alpha X^\beta \{\sigma_j(x,\pi) \psi(t\lambda_\pi)\}\|_{\sL(\cH_\pi)}
\lesssim
\|\sigma_j\|_{S^{m_j}_{\rho,\delta}, |\alpha|,|\beta|}
t^{\frac{m_0-m_j}2}. 
$$
This implies that for any $a,b\in \bN_0$, we have:
$$
\|\sigma_j \psi(t\lambda_\pi)\}\|_{S^{m_0}_{\rho,\delta}, a,b}
\leq C_{a,b,m_0,\sigma_j}
t^{\frac{m_0-m_j}2}. 
$$
We now choose a decreasing sequence of numbers $\{t_j\}$ such that for any $j\in \bN_0$, we have
$$
t_j\in (0,2^{-j})
\qquad\mbox{and}\qquad
C_{j,j,m_0,\sigma_j}
t_j^{m_0-m_j}\leq 2^{-j}.
$$
We then define the symbol $\tilde \sigma_j$
via 
$\tilde\sigma_j(x,\pi)=\sigma_j(x,\pi) \psi(t_j\lambda_\pi)$.

For any $\ell\in \bN_0$, the sum 
$$
\sum_{j=0}^\infty \|\tilde\sigma_j\|_{S^{m_0}_{\rho,\delta}, \ell,\ell}
\leq
\sum_{j=0}^\ell \|\tilde\sigma_j\|_{S^{m_0}_{\rho,\delta}, \ell,\ell}
+
\sum_{j=\ell+1}^\infty  2^{-j},
$$
is finite. As $S^{m_0}_{\rho,\delta}$ is a Fr\'echet space,
we obtain that $\sigma = \sum_{j=0}^\infty \tilde \sigma_j$ is a symbol in $S^{m_0}_{\rho,\delta}$.

Starting the summation at $j=M+1$, the same proof gives
$\sum_{j=M+1}^\infty \tilde \sigma_j \in S^{m_{M+1}}_{\rho,\delta}$.
Hence the symbol given via
$$
\sigma(x,\pi)-\sum_{j=0}^M \sigma_j(x,\pi)
=
\sum_{j=0}^M \sigma_j(x,\pi) (1-\psi)(t_j\lambda_\pi)
+
\sum_{j=M+1}^\infty \tilde \sigma_j ,
$$
is in $S^{m_{M+1}}_{\rho,\delta}$ as 
the symbol $(1-\psi)(t_j\lambda_\pi)$ is smoothing by Proposition \ref{prop_mult_t}
and so is $\sigma_j (1-\psi)(t_j\lambda_\pi)$
by Corollary \ref{cor_algebra_of_symbol}.

The property in \eqref{eq_asymptotic} is proved but 
it remains to show that the symbol $\sigma$ is unique modulo smoothing operator. If $\tau$ is another symbol as in the statement of the theorem, then for any $M\in \bN$, 
$\sigma-\tau \in S^{m_{M+1}}_{\rho,\delta}$
as this symbol is the difference of $\sigma -\sum_{j=0}^M \sigma_j$
with $\tau-\sum_{j=0}^M \sigma_j$, both is in $S^{m_{M+1}}_{\rho,\delta}$ by \eqref{eq_asymptotic}. Hence $\sigma=\tau$ modulo $S^{-\infty}$.
\end{proof}

In the expansion given for adjoint and composition, 
we will need to identify a suitable choice of $\Delta=\Delta_Q$ 
together with a choice of vector fields.
This is the purpose of the next lemma, whose proof is left to the reader:
\begin{lemma}
\label{lem_choice_q_X}
Let $\Delta=\Delta_Q$ be a strongly admissible collection difference operators. We may assume that $n_\Delta=n$.
There exists an adapted basis $\cX_\Delta:=X_{\Delta,1}, \ldots, X_{\Delta,n}$
such that $X_j \{q_k(\cdot^{-1})\}(e_G)=\delta_{j,k}$.
The following Taylor estimates hold
 for any integer $N\in \bN_0$ and $y\in G$:
$$
\big| R^{f}_{x,N} \big|
\leq C |y|^{N} 
\max_{|\alpha| \leq N} \|X_\Delta^\alpha f \|_\infty,
$$
where the constant $C>0$ depends in $N,G,\Delta$
but not on $f\in \cD(G)$.
%
Furthermore for any $\beta\in \bN_0^n$, 
we have on the one hand
$\{\tilde X_{\Delta}^{\beta}\}|_{x_1=x} R^{f}_{x_1,N}
=
 R^{\tilde X_{\Delta}^{\beta} f}_{x,N}$
and on the other hand, 
$\{X_{\Delta}^{\beta}\}|_{y_1=y} \{R^{f}_{x,N}(y_1)\}$
satisfies the same estimates as 
$R^{X_{\Delta}^\beta f}_{x,N -|\beta|}(y_1)$
above if $N-|\beta|\geq 0$.
\end{lemma}

Here and in the rest of the paper, 
if $N\in \bN_0$, then
$R^{f}_{x,N}$ denotes the Taylor remainder of $f$ at $x$ of order $N-1$ 
(adapted to the fixed collection $\Delta$):
$$
R^{f}_{x,N} (y)= 
f(xy) -\sum_{|\alpha|< N} q^\alpha  (y^{-1})X_\Delta^\alpha f(x)
$$
and  $X_\Delta^\alpha=X_{\Delta,1}^{\alpha_1}\ldots X_{\Delta,n}^{\alpha_n}$.
If $N<0$ then $R^{f}_{x,N}\equiv f(x\, \cdot)$.

\begin{proof}
The proof is straightforward. 
The properties of the remainder follow from the facts that left and right invariant vector fields commute and that the Taylor expansion is essentially unique.
\end{proof}

\subsection{Adjoint property for $\rho\not=\delta$}
\label{subsec_adjoint1}

This section is devoted to showing Proposition \ref{prop_adjoint} 
with  a more classical proof 
in the case $\rho>\delta$.
It will yields asymptotic expansions. 
In the rest of this section, we assume that $\Delta$ and $\cX_\Delta$ are fixed and chosen as in Lemma \ref{lem_choice_q_X}. We also simplify slightly the notation by setting $X_{\Delta,j}=X_j$.

\begin{lemma}
\label{lem_adjoint}
We assume that $1\geq \rho>\delta\geq 0$.
Let $\sigma\in S^m_{\rho,\delta}$ and let $\kappa_x$ be its associated kernel.
We assume that $(x,y)\mapsto \kappa_x(y)$ is smooth on $G\times G$.
Then for any multi-indices $\beta,\beta_0,\alpha_0\in \bN_0^n$,
there exists $N_0\in \bN_0$ such that for any integer $N>N_0$, 
we have
$$
\int_G \left| \tilde X_y ^{\beta} \tilde X_x^{\beta_0} 
\left\{ q_{\alpha_0}(y) \left(
\kappa_{x}^{(*)}(y) - 
\sum_{|\alpha|< N} q^\alpha  (y) X_x^\alpha \kappa_{x}^*(y)
\right)\right\} \right| dy
\leq C \|\sigma\|_{S^m_{\rho,\delta}, a,b},
$$
where the constant $C>0$ and the semi-norm $\|\cdot\|_{S^m_{\rho,\delta}, a,b}$ are independent on $\sigma$
(but may depend on $N,m,\rho,\delta,\Delta,\alpha_0,\beta_0,\beta$).
\end{lemma}

\begin{proof}
The idea is to use the estimate 
given in Lemma \ref{lem_choice_q_X} 
for the Taylor reminder 
\begin{equation}
\label{eq_remainder_lem_adjoint}
R^{\kappa_\cdot^* (y)}_{x,N}(y^{-1})=
\kappa_{x}^{(*)}(y) - 
\sum_{|\alpha|< N} q^\alpha  (y) X_x^\alpha \kappa_{x}^*(y) 
\end{equation}
in the case $\beta=\beta_0=\alpha_0=0$.
More generally, for any multi-indices, 
using \eqref{eq_X_tildeX},
we have:
\begin{eqnarray}
&&\left| \tilde X_y ^{\beta} \tilde X_x^{\beta_0} 
\left\{ q_{\alpha_0}(y)R^{\kappa_\cdot^* (y)}_{x,N}(y^{-1})
\right\} \right|
=
\left| \tilde X_y ^{\beta} 
\left\{ R^{\tilde X_{x_1}^{\beta_0} (q_{\alpha_0}\kappa_{x_1}^*) (y)}_{x_1=x,N}(y^{-1})
\right\} \right| \nonumber
\\&&\quad\lesssim
\sum_{|\beta_1|+|\beta_2|=|\beta|}
\left| X_{y_2=y} ^{\beta_2} 
\left\{ R^{\tilde X_{x_1}^{\beta_0} \tilde X_{y_1=y} ^{\beta_1}(q_{\alpha_0}\kappa_{x_1}^*) (y_1)}_{x_1=x,N}(y_2)
\right\}\right| \nonumber
\\&&\quad\lesssim
\sum_{|\beta_1|+|\beta_2|=|\beta|}
|y|^{(N-|\beta_2|)_+}
\max_{\substack{x_1\in G\\|\alpha| \leq N }} 
|X^\alpha_{x_1} \tilde X_{x_1}^{\beta_0} X_y ^{\beta_1}
(q_{\alpha_0}\kappa_{x_1}^*) (y)|, \label{eq_lem_adjoint1}.
\end{eqnarray}
We apply Proposition \ref{prop_estimate_kernel}
(see also Section \ref{subsec_symbol_algebra}\footnote{change})
to estimate the maximum:
\begin{eqnarray*}
\max_{\substack{x_1\in G\\|\alpha| \leq N}} 
|X^\alpha_{x_1} \tilde X_{x_1}^{\beta_0} X_y ^{\beta_1}
(q_{\alpha_0}\kappa_{x_1}^*) (y)|
\lesssim \|\sigma\|_{S^m_{\rho,\delta},a,b}
\left\{\begin{array}{ll}
|y|^{-\frac e \rho}
& \mbox{if}\ e>0,
\\
|\ln|y|| 
& \mbox{if}\ e=0,
\\
1 & \mbox{if}\ e<0,
\end{array}
\right.
\end{eqnarray*}
with $e=n+m+\delta(|\beta_0|+N)+|\beta_1|-\rho|\alpha_0|$.
We assume $N\geq |\beta|$.
For any $\epsilon_o>0$ as small as one wants,   
the  sum in \eqref{eq_lem_adjoint1}
is 
$$
\lesssim \|\sigma\|_{S^m_{\rho,\delta},a,b}
\left\{\begin{array}{l}
|y|^{N-\epsilon_o-\frac{|\beta|}\rho} 
\qquad \mbox{if}\ n+m+\delta(|\beta_0|+N)+|\beta|-\rho|\alpha_0|  \leq 0,\\
|y|^{-\frac{n+m+\delta|\beta_0|-\rho|\alpha_0| +|\beta|+(\delta-\rho)N }\rho}
\qquad\mbox{otherwise.}
\\
\end{array}
\right.
$$
This is integrable against $dy$ when $N> n+|\beta|/\rho$ (with a suitable $\epsilon_o$) and the following implication holds
$$
n+m+\delta(|\beta_0|+N)+|\beta| -\rho|\alpha_0|>0
\Longrightarrow
n+m+\delta|\beta_0|-\rho|\alpha_0| +|\beta|+(\delta-\rho)N
<\rho n.
$$
As $\rho>\delta$, we can choose $N_0\in \bN$ such that $N_0>n+ |\beta|/\rho$ is the smallest integer satisfying the implication just above. 
This shows Lemma \ref{lem_adjoint}.
\end{proof}

\begin{proof}
[Proposition \ref{prop_adjoint} when $\rho>\delta$]
Let $\sigma\in S^m_{\rho,\delta}$.
First we assume that its associated kernel $(x,y)\mapsto \kappa_x(y)$ is smooth on $G\times G$.
We set 
$$
\tau_N(x,\pi):= \sigma^{(*)}(x,\pi) - \sum_{|\alpha <N}\Delta_Q^\alpha X_x^\alpha \sigma(x,\pi)^*,
$$
Using the properties of the left or right invariant vector fields, especially \eqref{eq_piXbeta_sup}, 
it is not difficult to obtain the following very crude estimate:
$$
\|\tau_N\|_{S^m_{\rho,\delta},a,b}
\leq C \!\!\!
\sum_{\substack{|\alpha_0|\leq a, |\beta_0|\leq b\\ 
|\beta|\leq 2\lceil \rho a +\max (m,0)}}
\sup_{\substack{\pi\in \Gh\\ x\in G}}
\| \tilde X_x^{\beta_0} \Delta_Q^{\alpha_0}\tau_N(x,\pi)
\pi(X)^\beta \|_{\sL(\cH_\pi)}.
$$
We see that $\tau_N(x,\cdot)$ is the group Fourier transform of 
$y\mapsto R^{\kappa_\cdot^* (y)}_{x,N}(y^{-1})$ given in \eqref{eq_remainder_lem_adjoint}.
Using \eqref{eq_Tkappa_sup} and \eqref{eq_cF_L1}, 
we see that each maximum above is bounded by 
the integral given in Lemma \ref{lem_adjoint}.
Thus for $N\geq N_0$ with $N_0,a',b'$ depending on $m,\rho,\delta,a,b$,
we have
$$
\|\tau_N\|_{S^m_{\rho,\delta},a,b}
\lesssim \|\sigma\|_{S^m_{\rho,\delta}, a',b'}.
$$
From the properties of the symbol classes (see Section \ref{subsec_symbol_algebra}),
 the sum $\sum_{|\alpha <N}\Delta_Q^\alpha X_x^\alpha \sigma(x,\pi)^*$
 is a symbol in $S^m_{\rho,\delta}$.
This implies that $\sigma^{(*)}$ is also in $S^m_{\rho,\delta}$
and depend continuously on $\sigma$.
By Lemma \ref{lem_approximation_kernel2}, 
this extends to any symbol $\sigma$.
\end{proof}

The proofs above provide a more precise version of Proposition \ref{prop_adjoint}:
\begin{corollary}
\label{cor_prop_adjoint}
Let $1\geq \rho\geq \delta\geq 0$.
If $\sigma\in S^m_{\rho,\delta}$ then there exists a unique symbol 
$\sigma^{(*)}$ in $S^m_{\rho,\delta}$ such that $(\Op(\sigma))^*=\Op(\sigma^{(*)})$.
Furthermore, 
choosing $\Delta$ and $\cX_\Delta$ as in Lemma \ref{lem_choice_q_X} with  $X_j:=X_{\Delta,j}$,
we have for any $N\in \bN_0$, 
$$
\{\sigma^{(*)}(x,\pi) - \sum_{|\alpha| \leq N}\Delta_Q^\alpha X_x^\alpha \sigma(x,\pi)^*\} \in S^{m-(\rho-\delta)N}_{\rho,\delta},
$$
and the following mapping is continuous
$$
\left\{\begin{array}{rcl}
S^m_{\rho,\delta} 
&\longrightarrow&
 S^{m-(\rho-\delta)N}_{\rho,\delta}\\
\sigma
&\longmapsto&
 \{\sigma^{(*)}(x,\pi) - \sum_{|\alpha| \leq N}\Delta_Q^\alpha X_x^\alpha \sigma(x,\pi)^*\} 
\end{array}\right. .
$$

If $\sigma\in S^m_{\rho,\delta}$ with $\rho > \delta$, 
then
$\sigma^{(*)} \sim \sum_j \sum_{|\alpha|=j} 
Delta_Q^\alpha X_x^\alpha \sigma^*.$
\end{corollary}

\begin{remark}
\label{rem_RT_pb_adjoint}
The proof that the adjoint of an operator remains in the calculus 
 given in
 \cite[Theorem 10.7.10]{ruzhansky+turunen_bk} is very formal 
  since it is impossible with their analysis to justify the claims 
 in the last paragraph of their proof.
\end{remark}

\subsection{Composition property for $\rho\not=\delta$}
\label{subsec_comp1}

This section is devoted to showing Proposition \ref{prop_composition}
with a more classical proof for $\rho>\delta$
which yields asymptotic expansions.
In the rest of this section, we assume that $\Delta$ and $\cX_\Delta$ are fixed and chosen as in Lemma \ref{lem_choice_q_X}. We also simplify slightly the notation by setting $X_{\Delta,j}=X_j$.

\begin{lemma}
\label{lem_composition}
We assume that $\rho>\delta$.
Let $\sigma_1\in S^{m_1}_{\rho,\delta}$ and 
$\sigma_2\in S^{m_2}_{\rho,\delta}$ 
with smooth associated kernels $\kappa_{2x}$, $\kappa_{1x}$.
Let also $\kappa_x$ given by \eqref{eq_kernel_composition}.
Then for any multi-indices $\beta_0,\alpha_0\in \bN_0^n$
and $b>0$,
there exists $N_0\in \bN$ such that for any integer $N\geq N_0$, 
we have for any $(x,\pi)\in G\times\Gh$
\begin{eqnarray*}
\|\tilde X_x^{\beta_0} \Delta_Q^{\alpha_0}
\big(\sigma(x,\pi) -\sum_{|\alpha|<N}
\Delta_Q^{\alpha}\sigma_1(x,\pi)
X_x^{\alpha}\sigma_2(x,\pi)\big)
\|_{\sL(\cH_\pi)}
\\
\leq C
\|\sigma_1\|_{S^{m_1}_{\rho,\delta}, a_1,b_1}
\|\sigma_2\|_{S^{m_2}_{\rho,\delta}, a_2,b_2}
(1+\lambda_\pi)^{-b},
\end{eqnarray*}
where the constant $C>0$ and the semi-norms
$\|\cdot\|_{S^{m_1}_{\rho,\delta}, a_1,b_1}$,
$\|\cdot\|_{S^{m_2}_{\rho,\delta}, a_2,b_2}$,
 are independent of $x,\pi$ and $\sigma_1,\sigma_2$
(but may depend on $b,N,m_1,m_2,\rho,\delta,\Delta,\alpha_0,\beta_0$).
\end{lemma}

\begin{proof}
We notice that 
$$
\kappa_{x}(y) - 
\sum_{|\alpha|< N} 
(X_x^\alpha \kappa_{2,x})*
(q^\alpha \kappa_{1,x})(y)
=
\int_G  \kappa_{1,x}(z)
R_{x,N}^{\kappa_{2,\cdot}(yz^{-1}) }(z^{-1})
dz
$$
thus taking the group Fourier transform
$$
\sigma(x,\pi) -\sum_{|\alpha|<N}
\Delta_Q^{\alpha}\sigma_1(x,\pi)
X_x^{\alpha}\sigma_2(x,\pi)
=
\int_G  \kappa_{1,x}(z)\pi(z)^* 
R_{x,N}^{\sigma_2(\cdot, \pi)}(z^{-1})
dz
$$
having used the notation for the Taylor estimate for a matrix valued function
- which is possible.
We may assume, and we do, that $\Delta=\Delta_Q$ satisfies the Leibniz like property
(see Theorem \ref{thm_Deltaeq+coincide} and Corollary \ref{cor_choice_Delta}).
Using this and the Leibniz property for vector fields, 
one checks easily that  
\begin{eqnarray}
&&\|\tilde X_x^{\beta_0}
\Delta_Q^{\alpha_0}
\big(\sigma(x,\pi) -\sum_{|\alpha|<N}
\Delta_Q^{\alpha}\sigma_1(x,\pi)
X_x^{\alpha}\sigma_2(x,\pi)\big) \|_{\sL(\cH_\pi)}
\nonumber
\\&&\ \lesssim 
\!\!\!
\sum_{\substack{|\alpha_0|\leq |\alpha_1|+|\alpha_2|\leq 2|\alpha_0|\\ 
|\beta_{0,1}|+|\beta_{0,2}|=|\beta_0|}}
\!\!\!
\|\int_G  (\tilde X_x^{\beta_{0,1}}q_{\alpha_1}\kappa_{1,x})(z)\pi(z)^* 
R_{x,N}^{
\Delta_Q^{\alpha_2}\{\tilde X^{\beta_{0,2}}\sigma_2(\cdot, \pi)\} }(z^{-1})
dz \|_{\sL(\cH_\pi)}
\label{eq_pf_lem_composition0}
\\&&\ \lesssim \!\!\!
\sum_{\substack{|\beta_{1}|+|\beta_{2}|=2b_1\\
|\alpha_0|\leq |\alpha_1|+|\alpha_2|\leq 2|\alpha_0|\\ 
|\beta_{0,1}|+|\beta_{0,2}|=|\beta_0|}}
(1+\lambda_\pi)^{-b_1}
\int_G  
|(X_z^{\beta_1}\tilde X_x^{\beta_{0,1}}q_{\alpha_1}\kappa_{1,x})(z)| 
\nonumber
\end{eqnarray}
\vspace{-2em}
\begin{equation}
\qquad\qquad\qquad
\| \tilde X_{z_1=z^{-1}}^{\beta_2} \{
 R_{x,N}^{
\Delta_Q^{\alpha_2}\{\tilde X^{\beta_{0,2}}\sigma_2(\cdot, \pi)\} }(z_1)\}\|_{\sL(\cH_\pi)}
dz,\label{eq_lem_composition1}
\end{equation}
for any $b_1\in \bN_0$,
having interpreted 
$\pi(z)^* =  (1+\lambda_\pi)^{-b_1} (\id +\cL)_z^{b_1}  \pi(z)^*  $
and using integration by parts. 
Using the Taylor estimates, see Lemma \ref{lem_choice_q_X}, 
we have
\begin{eqnarray*}
\| \tilde X_{z_1=z^{-1}}^{\beta_2} \{
 R_{x,N}^{
\Delta_Q^{\alpha_2}\{\tilde X^{\beta_{0,2}}\sigma_2(\cdot, \pi)\} }(z_1)\}\|_{\sL(\cH_\pi)}
\lesssim
|z|^{(N-|\beta_2|)_+} 
\sup_{\substack{x_1\in G\\
|\beta'_2|\leq N} }
\|X_{x_1}^{\beta'_2}
\Delta_Q^{\alpha_2}\{\tilde X_{x_1}^{\beta_{0,2}}\sigma_2(x_1, \pi)\}
\|_{\sL(\cH_\pi)}
\\
\lesssim
\|\sigma_2\|_{S^{m_2}_{\rho,\delta}, N+|\beta_{0,2}|,|\alpha|}
|z|^{(N-|\beta_2|)_+} 
(1+\lambda_\pi)^{\frac{m_2+\delta(N+|\beta_{0,2}|)
-\rho |\alpha_2|}2}.
\end{eqnarray*}
By  Proposition \ref{prop_estimate_kernel}
(see also Section \ref{subsec_symbol_algebra}),
we have
$$
|(X_z^{\beta_1}\tilde X_x^{\beta_{0,1}}q_{\alpha_1}\kappa_{1,x})(z)| 
\lesssim \|\sigma_1\|_{S^{m_1}_{\rho,\delta},a_1,b_1}
\left\{
\begin{array}{ll}
|z|^{-\frac e \rho}
&\mbox{if} \ e >0,
\\
|\ln |z||
&\mbox{if} \ e =0,
\\
1&\mbox{if} \ e <0,
\\
\end{array}
\right.
$$
where 
$$
e:=e(|\beta_{0,1}|,|\beta_1|, |\alpha_1|):=n+m_1+\delta|\beta_{0,1}|+|\beta_1|-\rho|\alpha_1|.
$$
Thus each term in the sum \eqref{eq_lem_composition1}
is 
$$
\lesssim 
(1+\lambda_\pi)^{-b_1+\frac{m_2+\delta(N+|\beta_{0,2}|)
-\rho |\alpha_2|}2}
\|\sigma_1\|_{S^{m_1}_{\rho,\delta},a_1,b_1}
\|\sigma_2\|_{S^{m_2}_{\rho,\delta}, a_2,b_2}
I(|\beta_{0,1}|,|\beta_1|, |\alpha_1|)
$$
where $I(|\beta_{0,1}|,|\beta_1|, |\alpha_1|)$ is the integral
$$
I(|\beta_{0,1}|,|\beta_1|, |\alpha_1|)=\left\{
\begin{array}{ll}
\int_G |z|^{(N-|\beta_2|)_+ -\frac e\rho}dz
|z|^{-\frac e \rho}
&\mbox{if} \ e >0,
\\
\int_G |z|^{(N-|\beta_2|)_+}
|\ln |z||dz
&\mbox{if} \ e =0,
\\
\int_G |z|^{(N-|\beta_2|)_+}
dz&\mbox{if} \ e <0.
\\
\end{array}\right.
$$
The integrals $I(|\beta_{0,1}|,|\beta_1|, |\alpha_1|)$ are  finite 
when $(N-|\beta_2|)_+ -\frac {e_+}\rho >-n$.
To ensure this, 
we choose $N_0\in \bN$ satisfying
$$
N_0>-n +\frac 1\rho
\max_{\substack{ |\alpha_1|\leq |2\alpha_0|
\\
|\beta_{0,1}|\leq |\beta_0|}} e(|\beta_{0,1}|,0, |\alpha_1|)_+,
$$
and, noticing that 
$$
\max_{ |\alpha_1|\leq 2 |\alpha_0|,
|\beta_{0,1}|\leq |\beta_0|,
|\beta_2|\leq 2b_1} 
\left(|\beta_2|+
\frac {e(|\beta_{0,1}|,|\beta_1|, |\alpha_1|)_+}\rho\right)
\leq \frac {2b_1}\rho + (n+m_1+\delta|\beta_0|)_+,
$$
we define $b_1\in \bN_0$ as the largest  integer such that $b_1\leq N/2$ and 
$$
b_1<\frac \rho 2(N+n)  - \frac {(n+m_1+\delta|\beta_0|)_+} 2.
$$
Under these conditions, we have obtained:
\begin{eqnarray*}
&&\|\tilde X_x^{\beta_0}
\Delta_Q^{\alpha_0}
\big(\sigma(x,\pi) -\sum_{|\alpha|<N}
\Delta_Q^{\alpha}\sigma_1(x,\pi)
X_x^{\alpha}\sigma_2(x,\pi)\big) \|_{\sL(\cH_\pi)}
\\&&\ \lesssim \!\!\!
\sum_{\substack{|\beta_{1}|+|\beta_{2}|=2b_1\\
|\alpha_0|\leq |\alpha_1|+|\alpha_2|\leq 2|\alpha_0|\\ 
|\beta_{0,1}|+|\beta_{0,2}|=|\beta_0|}}
(1+\lambda_\pi)^{-b_1+\frac{m_2+\delta(N+|\beta_{0,2}|)
-\rho |\alpha_2|}2}
\|\sigma_1\|_{S^{m_1}_{\rho,\delta},a_1,b_1}
\|\sigma_2\|_{S^{m_2}_{\rho,\delta}, a_2,b_2}
\end{eqnarray*}
\vspace{-2em}
$$
\qquad\qquad\qquad\qquad\qquad\lesssim
(1+\lambda_\pi)^{-\tilde b/2}
\|\sigma_1\|_{S^{m_1}_{\rho,\delta},a_1,b_1}
\|\sigma_2\|_{S^{m_2}_{\rho,\delta}, a_2,b_2},
$$
where
\begin{eqnarray*}
\tilde b
:=
2b_1- m_2-\delta(N+|\beta_{0}|)
\ \geq (\rho-\delta) N 
+\rho n -(n+m_1+\delta|\beta_0|)_+-m_2- \delta |\beta_0|  -2.
\end{eqnarray*}
Hence if $\rho>\delta$ with $N_0$  chosen large enough, $\tilde b$ may be as large as one wants.
This shows Lemma \ref{lem_composition} in this case.
\end{proof}

Proceeding in a similar way as for the case of the adjoint, 
Lemma \ref{lem_composition} implies Proposition \ref{prop_composition}.
The proof also yields:
\begin{corollary}
\label{cor_prop_composition}
Let $1\geq \rho\geq \delta\geq 0$.
If $\sigma_1\in S^{m_1}_{\rho,\delta}$
and 
$\sigma_2\in S^{m_2}_{\rho,\delta}$
 then there exists a unique symbol 
$\sigma = \sigma_1 \circ \sigma_2$ in $S^{m_1+m_2}_{\rho,\delta}$ such that $(\Op(\sigma))=\Op(\sigma_1)\Op(\sigma_2)$.
Furthermore, 
choosing $\Delta$ and $\cX_\Delta$ as in Lemma \ref{lem_choice_q_X} with  $X_j:=X_{\Delta,j}$,
we have for any $N\in \bN_0$, 
$$
\{\sigma(x,\pi) - \sum_{|\alpha| \leq N} 
\Delta_Q^\alpha \sigma_1(x,\pi)
X_x^\alpha \sigma_2(x,\pi)\} \in S^{m_1+m_2-(\rho-\delta)N}_{\rho,\delta},
$$
and the following mapping is continuous
$$
\left\{\begin{array}{rcl}
S^{m_1}_{\rho,\delta}
\times S^{m_2}_{\rho,\delta}
&\longrightarrow&
 S^{m_1+m_2-(\rho-\delta)N}_{\rho,\delta}\\
\sigma
&\longmapsto&
 \{\sigma(x,\pi) - \sum_{|\alpha| \leq N} 
 \Delta_Q^\alpha \sigma_1(x,\pi)
X_x^\alpha \sigma_2(x,\pi)\} 
\end{array}\right. .
$$

If $\sigma_1\in S^{m_1}_{\rho,\delta}$
and 
$\sigma_2\in S^{m_2}_{\rho,\delta}$ with $\rho>\delta$, 
then
$\sigma\sim \sum_j \sum_{|\alpha|=j} 
 \Delta_Q^\alpha \sigma_1(x,\pi)
X_x^\alpha \sigma_2(x,\pi)$.
\end{corollary}

\begin{remark}
\label{rem_RT_pb_comp}
The proof that the composition of two operators remains in the calculus 
 given in
 \cite[Theorem 10.7.8]{ruzhansky+turunen_bk} is very formal 
  since it is impossible with their analysis to justify the claims 
 in the last paragraph of their proof.
\end{remark}

\section{Boundedness on Sobolev spaces and commutators}
\label{sec_L2bdd+commutator}

In this section, we show  that pseudo-differential operators are bounded
on Sobolev spaces and we give a commutator characterisation of the operators in the calculus.
This will prove the last property \eqref{item_def_pseudo-diff_calculus_sobolev} in Definition \ref{def_pseudo-diff_calculus}
and the fact that our calculus coincide with the H\"ormander calculus when the latter is defined. This will conclude the proof of Theorem \ref{thm_main}.

\subsection{Boundedness on $L^2(G)$}

This section is devoted to showing that operators of order 0 are bounded on $L^2(G)$ in the following sense:

\begin{proposition}
\label{prop_L2bdd}
Let $1\geq \rho\geq \delta\geq 0$ with $\delta\not=1$.
If $\sigma\in S^0_{\rho,\delta}$
then $\Op(\sigma)$ is bounded on $L^2(G)$:
$$
\exists C>0\qquad
\forall \phi\in \cD(G)\qquad
\|\Op(\sigma)\phi\|_{L^2(G)}\leq C \|\phi\|_{L^2(G)}.
$$
Moreover the constant $C$ may be chosen 
of the form $C=C' \|\sigma\|_{S^0_{\rho,\delta},a,b}$
with $C'>0$ and $\|\cdot\|_{S^0_{\rho,\delta},a,b}$
independent of $\sigma$ (but maybe depending on $G$ and $\rho,\delta$).
\end{proposition}

Given the continuous inclusions of the spaces $S^0_{\rho,\delta}$, 
it suffices to prove the case $\rho=\delta$.
We first show the case $\rho=\delta=0$ and then 
the case $\rho=\delta$ (strictly) positive.
The case $(\rho,\delta)=(0,0)$ follows from the following  lemma
since, using the notation of the lemma, 
$C_0=\|\sigma\|_{S^0_{0,0}, 0,\lceil \frac n2\rceil}$.
This lemma was already given in in
 \cite[Theorem 10.5.5]{ruzhansky+turunen_bk}.

\begin{lemma}
\label{lem_L2bdd_00}
If $\sigma$ is a smooth symbol such that 
$$
C_0:=\max_{\substack{ x\in G \\|\alpha|\leq \lceil \frac n2\rceil} }
\sup_{\pi\in \Gh} \|X_x^\alpha \sigma(x,\pi)\|_{\sL(\cH_\pi)}<\infty
$$
then $\Op(\sigma)$ is bounded on $L^2(G)$:
$$
\exists C>0\qquad
\forall \phi\in \cD(G)\qquad
\|\Op(\sigma)\phi\|_{L^2(G)}
\leq C \|\phi\|_{L^2(G)},
$$
Moreover the constant $C$ may be chosen of the form 
$C=C'C_0$
with $\sigma'$ independent of $\sigma$.
\end{lemma} 

\begin{proof}
Let $T=\Op(\sigma)$, $\sigma\in S^0_{0,0}$ and  $f\in \cD(G)$.
 Sobolev's inequalities yield
$$
|Tf(x)|^2= | f*\kappa_x(x)|^2
\leq
\sup_{x_1\in G} |f*\kappa_{x_1}(x)|^2
\lesssim
\sum_{|\alpha|\leq \lceil \frac n2\rceil} 
\int_G|  X_{x_1}^\alpha f*\kappa_{x_1}(x)|^2 d{x_1}.
$$
As $X_{x_1}^\alpha f*\kappa_{x_1}(x) = T_{X_{x_1}\kappa_{x_1}}(f)$, 
after integration over $G$, we obtain:
\begin{eqnarray*}
\int_G|Tf(x)|^2 dx
&\lesssim&
\sum_{|\alpha|\leq \lceil \frac n2\rceil} 
\int_G\int_G | T_{X_{x_1}^\alpha\kappa_{x_1}}(f)(x) |^2 dx dz
\\&\lesssim&
\sum_{|\alpha|\leq \lceil \frac n2\rceil} 
\int_G
\|T_{X_{x_1}^\alpha\kappa_{x_1}}\|_{\sL(L^2(G))}^2
\|f\|_{L^2(G)}^2
 dz\\
& \lesssim&
\max_{\substack{ z\in G \\|\alpha|\leq \lceil \frac n2\rceil} }
\|T_{X_{x_1}^\alpha\kappa_{x_1}}\|_{\sL(L^2(G))}^2
\|f\|_{L^2(G)}^2.
\end{eqnarray*}
We conclude with
$C_0=\max \big\{\|T_{X_{x_1}^\alpha\kappa_{x_1}}\|_{\sL(L^2(G))},
x_1\in G,  |\alpha|\leq \lceil \frac n2\rceil \big\}$.
\end{proof}

The case of $\rho=\delta\in (0,1)$, 
 is more delicate and, in its proof,  
 we will need the following property 
which uses the arguments above (amongst others).

\begin{lemma}
\label{lem_L2bdd_rho0}
Let $\eta\in \cD(0,\infty)$ and $\rho\in (0,1)$.
There exists $C=C_{\eta,\rho,G}$
such that for any  $T\in \Psi^0_{\rho,\rho}$, 
 $\ell\in \bN_0$, 
we have:
$$
\|T\eta(2^{-\ell} \cL)\|_{\sL(L^2(G))}
\leq C  \|T\|_{\Psi^0_{\rho,\rho}, 0,\lceil \frac n2\rceil}.
$$
\end{lemma} 

\begin{proof}
As the exponential mapping is a diffeomorphism from a neighbourhood 
$\cV$ of $0\in \bR^n$ to $B(e_G,\epsilon_0)$, 
there exists a finite number of points $x_0=e_G, x_1, \ldots, x_{N_0}$ 
such that $G=\cup_{j=0}^{N_0} B(x_j, \epsilon_0/4)$
and some functions $\chi_j\in C^\infty(G)$ valued in $[0,1]$ 
and supported in $B(e_G, \epsilon_0/2)$  such that 
$\sum_{j=0}^{N_0} \chi_j(x_j^{-1}x  )=1$ for all $x\in G$.

Note that if  $x\in B(e_G,\epsilon)$ and $r\leq 1$, 
then we can define a local dilation via: 
$r\cdot x =\exp (rv)$ where $x=\exp v$, $v\in \cV$.

Let $\sigma\in S^0_{\rho,\rho}$.
For each $j=0, \ldots, N_0$, we define $\sigma_j\in S^0_{\rho,\rho}$ via 
$$
\sigma_j(x,\pi):= \sigma(x_j x ,\pi) \chi_j(x), 
\quad (x,\pi)\in G\times \Gh.
$$
For each $\ell\in \bN_0$ and $j=0, \ldots, N_0$,
we set 
$$
\sigma_\ell(x,\pi)=\sigma(x,\pi) \eta (2^{-\ell}\lambda_\pi), 
\quad\mbox{and}\quad
\sigma_{j,\ell}(x,\pi)=\sigma_j(x,\pi) \eta (2^{-\ell}\lambda_\pi).
$$
We have
$\sigma_\ell (x,\pi)=\sum_{j=0}^{N_0} \sigma_{j,\ell}(x x_j,\pi)$.
Recall that $\Op$ and using the argument in Lemma \ref {lem_Psi0_inv_left_translation}, one shows easily that 
if $\tau=\{\tau(x,\pi), (x,\pi)\in G\times \Gh\}$ is a symbol 
such that $\Op(\tau)$ is bounded on $L^2(\bR^n)$ then for any $x_0\in G$ we have
$$
\| \Op(\tau)\|_{\sL(L^2(G))}
=
\| \Op(\tau_{L,x_0})\|_{\sL(L^2(G))},
\quad\mbox{where}\quad \tau_{L,x_0}(x,\pi)=\tau( x_0 x,\pi).
$$
Therefore we have
\begin{equation}
\label{eq_pf_lem_L2bdd_rho0}
\|\Op(\sigma_\ell)\|_{\sL(L^2(G))} 
\leq 
\sum_{j=0}^{N_0}
\|\Op(\sigma_{j,\ell})\|_{\sL(L^2(G))},
\end{equation}
and we are left with proving the $L^2$-boundedness for each $\Op(\sigma_{j,\ell})$.
We notice that the $x$-support of its symbol $\sigma_{j,\ell}(x,\pi)$ is included in $B(e_G,\epsilon_0)$
and we can dilate its argument to define:
$$
\tilde\sigma_{j,\ell}(x,\pi)=
\left\{\begin{array}{ll}
\sigma_{j,\ell}(2^{-\rho \ell}\cdot x,\pi)
&\mbox{if}\ x\in B(e_G,\epsilon_0),\\
0 &\mbox{otherwise}.
\end{array}\right.
$$ 
Then one checks easily that the symbols $\sigma_\ell$,$\sigma_{j,\ell}$, and $\tilde \sigma_{j,\ell}$ are in $S^0_{\rho,\rho}$.

The symbol $\sigma_{j,\ell}$
and its convolution kernel $\kappa_{j,\ell}=\{\kappa_{j,\ell,x}(y)\}$
are supported in $x$ in $B(e_G,\epsilon_0)$,
thus for any $f\in \cD(G)$, 
$\Op(\sigma_{j,\ell})(f)$ is also supported in $B(e_G,\epsilon_0)$
and we can dilate its argument, that is, for any $x\in B(e_G, 2^{\ell \rho}\epsilon_0)$
$$
\Op(\sigma_{j,\ell})(f) (2^{-\ell \rho} \cdot x)
=
f* \kappa_{j,\ell,2^{-\ell \rho} \cdot x}(2^{-\ell \rho} \cdot x)
=
f* \tilde \kappa_{j,\ell,x}(2^{-\ell \rho} \cdot x),
$$
where $\tilde \kappa_{j,\ell}=\{\tilde \kappa_{j,\ell,x}(y)\}$ is the convolution kernel associated with  $\tilde \sigma_{j,\ell}$. 
Proceeding as in the proof of Lemma \ref{lem_L2bdd_00}, 
we have
$$
|\Op(\sigma_{j,\ell})(f) (2^{-\ell \rho} \cdot x)|
\lesssim
\sum_{|\beta|\leq \lceil \frac n2\rceil}
\|X_{x_1}^\beta f* \tilde \kappa_{j,\ell,x_1}(2^{-\ell \rho} \cdot x)\|_{L^2(dx_1)}.
$$
On both sides, we now integrate over $x\in B(e_G, 2^{\ell \rho}\epsilon_0)$ and make the change of variables $x'= 2^{-\ell \rho} \cdot x$
(with constant Jacobian $2^{-\ell \rho n}$):
$$
\|\Op(\sigma_{j,\ell})(f) (x')\|_{L^2(dx')}
\lesssim 
\sum_{|\beta|\leq \lceil \frac n2\rceil}
\|X_{x_1}^\beta f* \tilde \kappa_{j,\ell,x_1}(x')\|_{L^2(dx_1dx')}.
$$
Therefore 
$$
\|\Op(\sigma_{j,\ell})\|_{\sL(L^2(G))}
\lesssim 
\sum_{|\beta|\leq \lceil \frac n2\rceil}
\sup_{(x_1,\pi)\in G\times \Gh}
\|X_{x_1}^\beta\tilde \sigma_{j,\ell}(x_1,\pi)\|_{\sL(\cH_\pi)}
\lesssim 
\|\sigma\|_{S^0_{\rho,\rho},0,\lceil \frac n2\rceil}.
$$
Because of \eqref{eq_pf_lem_L2bdd_rho0}, 
the proof of Lemma \ref{lem_L2bdd_rho0} is now complete.
\end{proof}

The case $\rho=\delta\in (0,1)$ is proved as follows:

\begin{lemma}
\label{lem_L2bdd_rho}
Let $\rho\in (0,1)$.
If $\sigma\in S^0_{\rho,\rho}$
then $\Op(\sigma)$ is bounded on $L^2(G)$:
$$
\exists C>0\qquad
\forall \phi\in \cD(G)\qquad
\|\Op(\sigma)\phi\|_{L^2(G)}\leq C \|\phi\|_{L^2(G)}.
$$
Moreover the constant $C$ may be chosen 
of the form $C=C' \|\sigma\|_{S^0_{\rho,\rho},0,b}$
with $C'>0$ and $b$
independent of $\sigma$ (but depending on $n$ and $\rho$).
\end{lemma} 

\begin{proof}
We consider the same type of dyadic decomposition of $\spec(\cL)$
as in the first proofs of Propositions \ref{prop_adjoint}
and \ref{prop_composition}: 
we choose two functions $\eta_0,\eta_1\in \cD(\bR)$ supported in 
$[-1,1]$ and $[1/2,2]$ respectively, both valued in $[0,1]$
and satisfying \eqref{eq_dyadic_dec}. 
We set $\sigma_{1,\ell}(x,\pi) =\sigma_1 (x,\pi) \eta_\ell(\lambda_\pi) $
for each $\ell\in \bN_0$, 
as well as
$$
T:=\Op(\sigma),
\quad\mbox{and}\quad
T_\ell:=\Op(\sigma_\ell) = \Op(\sigma) \eta_\ell(\cL)
= T \eta_\ell(\cL).
$$

The properties of such a dyadic decomposition implies classically
\begin{equation}
\label{eq_pf_lem_L2bdd_rho_dec}
\|T\|_{\sL(L^2(G))}^2
\lesssim 
\sup_{\ell\in \bN_0} \|T_\ell\|^2_{\sL(L^2(G))}
+
\sum_{\substack{\ell'\not=\ell \\ \ell,\ell' \in 2\bN_0}}
\|T_\ell^* T_{\ell'}\|_{\sL(L^2(G))}
+
\sum_{\substack{\ell'\not=\ell \\ \ell,\ell' \in 2\bN_0+1}}
\|T_\ell^* T_{\ell'}\|_{\sL(L^2(G))}.
\end{equation}
The uniform boundedness of $T_\ell$'s operator norms follow from 
Lemmata \ref{lem_L2bdd_00} and \ref{lem_L2bdd_rho0}
but the boundedness of the sums remain to be shown. For this, we proceed as follows.

Let   $\kappa_\ell=\{\kappa_{\ell,x}(y)\}$ denote the convolution kernel of $T_\ell$ and let  $K_{\ell,\ell'}$ denote
the integral kernel of $T_{\ell}^* T_{\ell'}$:
$$
T_{\ell}^* T_{\ell'}f(x)=\int_G K_{\ell,\ell'}(x,y) f(y) dy,
\qquad K_{\ell,\ell'}(x,y) = \int_G \bar \kappa_{\ell,z} (x^{-1}z) \kappa_{\ell',z}(y^{-1}z) dz.
$$
As $G$ is compact, we have
\begin{equation}
\label{eq_cq_stein_VII241}
\|T_\ell^* T_{\ell'}\|_{\sL(L^2(G))}
\lesssim 
\sup_{x,y\in G} |K_{\ell,\ell'}(x,y) |.
\end{equation}

Let us assume $\ell\not=\ell'$.
Let $N\in \bN_0$. 
Introducing powers of $\id+\cL$ and using the Sobolev embedding
(cf. Lemma \ref{lem_sob_embedding} with $s'_0:=\lceil \frac n4\rceil$), we have
\begin{eqnarray*}
&&|K_{\ell,\ell'}(x,y)| 
=
\left|\int_G 
(\id+\cL)_{z_1=z}^N
(\id+\cL)_{z_2=z}^{-N}
\left\{
 \bar \kappa_{\ell,z_1} (x^{-1}z_2) \kappa_{\ell',z_1}(y^{-1}z_2)\right\} dz
\right|
\\
&&\qquad\lesssim
\int_G 
\sup_{z_1\in G} \left|(\id+\cL)_{z_1}^N
(\id+\cL)_{z_2}^{-N}
\left\{
 \bar \kappa_{\ell,z_1} (x^{-1}z_2) \kappa_{\ell',z_1}(y^{-1}z_2)\right\} 
\right| dz_2
\\
&&\qquad\lesssim
\int_G 
\left\|(\id+\cL)_{z_1=z}^{N+s'_0}
(\id+\cL)_{z_2}^{-N}
\left\{
 \bar \kappa_{\ell,z_1} (x^{-1}z_2) \kappa_{\ell',z_1}(y^{-1}z_2)\right\} 
\right\|_{L^2(dz_1)} dz_2
\\
&&\qquad\lesssim
\left\|(\id+\cL)_{z_1=z}^{N+s'_0}
(\id+\cL)_{z_2}^{-N}
\left\{
 \bar \kappa_{\ell,z_1} (x^{-1}z_2) \kappa_{\ell',z_1}(y^{-1}z_2)\right\} 
\right\|_{L^2(dz_2dz_1)}
\\
&&\qquad\lesssim
\!\!\!\!\!\!
\sum_{|\alpha_1|+|\alpha_2|\leq 2(N+s'_0)}
\!\!\!\!\!\!
\left\|
(\id+\cL)_{z_2}^{-N}
\left\{
 X_{z_1}^{\alpha_1}\bar \kappa_{\ell,z_1} (x^{-1}z_2) 
 X_{z_1}^{\alpha_2}\kappa_{\ell',z_1}(y^{-1}z_2)\right\} 
\right\|_{L^2(dz_2dz_1)},
\end{eqnarray*}
by the Leibniz rule.
Applying Lemma \ref{lem_bilinear} for $N\geq s'_0$,
we obtain easily
\begin{eqnarray*}
&&\left\|
(\id+\cL)_{z_2}^{-N}
\left\{
 X_{z_1}^{\alpha_1}\bar \kappa_{\ell,z_1} (x^{-1}z_2) 
 X_{z_1}^{\alpha_2}\kappa_{\ell',z_1}(y^{-1}z_2)\right\} 
\right\|_{L^2(dz_2)}\\
&&\qquad\lesssim 2^{-\max(\ell,\ell') (N-s'_0)}
 \|X_{z_1}^{\alpha_1}\kappa_{\ell,z_1} (z'_2) \|_{L^2(dz'_2)}
\| X_{z_1}^{\alpha_2}\kappa_{\ell',z_1}(z'_2)\|_{L^2(dz'_2)}.
\end{eqnarray*}
Lemma \ref{lem_kernel_inL2} and $\sigma\in S^0_{\rho,\rho}$ yield:
$$
\|X_{z_1}^{\alpha_1}\kappa_{\ell,z_1} (z'_2) \|_{L^2(dz'_2)}
\lesssim 
\sup_{\pi\in \Gh} (1+\lambda_\pi)^{s}
\|X_{z_1}^{\alpha_1} \sigma_\ell(z_1,\pi)\|_{\sL(\cH_\pi)}
\lesssim 
\|\sigma\|_{S^0_{\rho,\rho}, 0, |\alpha_1|}
2^{\ell (s'_0 +\rho\frac{|\alpha_1|}2)} ,
$$
thus
\begin{eqnarray*}
\sum_{|\alpha_1|+|\alpha_2|\leq 2(N+s'_0)}
\!\!\!\!\!\!\!\!\!\!\!\!\!\!\!
 \|X_{z_1}^{\alpha_1}\kappa_{\ell,z_1} (z'_2) \|_{L^2(dz'_2)}
\| X_{z_1}^{\alpha_2}\kappa_{\ell',z_1}(z'_2)\|_{L^2(dz'_2)}
\lesssim 
\|\sigma\|_{S^0_{\rho,\rho}, 0, 2(N+s'_0)}^2
2^{\max(\ell,\ell') (2s'_0+\rho (N + s'_0))}. 
\end{eqnarray*}
We have obtained for any $N\geq s'_0$:
$$
\sup_{x,y\in G} |K_{\ell,\ell'}(x,y)|  \lesssim \|\sigma\|_{S^0_{\rho,\rho}, 0, 2(N+s'_0)}^2
2^{\max(\ell,\ell') ((\rho-1)N+s_1)},
$$
with $s_1:=(2+\rho)s'_0$.
As $\rho\in (0,1)$, 
we can choose $N\in \bN$ such that $N\geq s_0$ and $(\rho-1)N+s_1<0$. 
This choice together with the estimates in \eqref{eq_cq_stein_VII241}
shows that the two sums in  \eqref{eq_pf_lem_L2bdd_rho_dec} are bounded, up to a constant by 
$\|\sigma\|_{S^0_{\rho,\rho}, 0, 2(N+s'_0)}^2$.
This concludes the proof of Lemma \ref{lem_L2bdd_rho}.
 \end{proof}
 
 \begin{remark}
\label{rem_RT_pb_L2}
 Lemma \ref{lem_L2bdd_rho} in the case of the torus 
 was announced in \cite[Section 4.8]{ruzhansky+turunen_bk} 
and proved in \cite[Theorem 9.5]{ruzhansky+turunen_10}.
However, the arguments there can not be extended to the case of a non-abelian group since the dimension of any $\pi\in \Gh$ is usually strictly greater than one.
 \end{remark}

Proposition \ref{prop_L2bdd} is thus proved.
We obtain the continuity on 
($L^2$-)Sobolev spaces with loss of derivatives controlled by the order:
\begin{corollary}
\label{cor_bdd_sob_spaces}
Let $1\geq \rho\geq \delta\geq 0$ with $\delta\not=1$
and $m\in \bR$.
If $\sigma\in S^m_{\rho,\delta}$, 
then $\Op(\sigma)$ maps boundedly the Sobolev spaces $H^s\to H^{s-m}$  for any $s\in \bR$ and we have
$$
\|\Op(\sigma)\|_{\sL(H^s, H^{s-m})}
\leq C \|\sigma\|_{S^m_{\rho,\delta},a,b},
$$
where the constant $C>0$ and the semi-norm 
$\|\cdot\|_{S^m_{\rho,\delta},a,b}$
are independent of $\sigma$ (but may depend on $s,m,\rho,\delta,G$).
\end{corollary}

This corollary of Proposition \ref{prop_L2bdd} 
follows readily from $(\id+\cL)^{m'/2}\in \Psi^{m'}_{1,0}$ for any $m'\in \bR$, see Proposition \ref{prop_mult_t}.

Note that, from the estimates of the kernel given in Proposition \ref{prop_estimate_kernel}, 
one checks easily that the operators $\Psi^0_{1,0}$ are of Calderon-Zygmund type
and hence are bounded on $L^p(G)$, 
$1<p<\infty$, 
see
\cite{coifman+weiss}.
So in the case $(\rho,\delta)=(1,0)$, 
also Corollary \ref{cor_bdd_sob_spaces} also holds any $L^p$-Sobolev spaces, $p\in(1,\infty)$.

Another consequence is the continuity for commutators, see the next section.
We will need the following property:
\begin{lemma}
\label{lem_Deltaqsigma_HL2}
Let $1\geq \rho\geq \delta\geq 0$ with $\delta\not=1$, $\rho\not=0$,
and $m\in \bR$.
If $q$ is a smooth function on $G$ vanishing at $e_G$  up to order $a_0-1$
(see Definition \ref{def_q_vanish_order})
and if $\sigma\in S^m_{\rho,\delta}$, 
then 
$\Op(\Delta_q \sigma)$ maps $H^{m-\rho a_0}$ boundedly to $L^2(G)$
and 
$$
\|\Op(\Delta_q\sigma)\|_{\sL(H^{m-\rho a_0}, L^2(G))}
\leq C \|\sigma\|_{S^m_{\rho,\delta},a,b},
$$
where the constant $C>0$ and the semi-norm $\|\sigma\|_{S^m_{\rho,\delta},a,b}$ are independent of $\sigma$
(but may depend on $q,a_0,m,\rho,\delta,\Delta,G$).
\end{lemma}

\begin{proof}[Lemma \ref{lem_Deltaqsigma_HL2}]
Let $\chi\in \cD(G)$ be valued in $[0,1]$ and  such that 
$\chi|_{B(\epsilon_0/2)} \equiv 1$ and 
$\chi|_{B(\epsilon_0)^c} \equiv 0$.
We write $\Delta_q \sigma = \Delta_{q\chi} \sigma +\Delta_{q(1-\chi)}\sigma$.
As the kernel associated with $\Delta_{q(1-\chi)}\sigma$ is smooth, 
this symbol is smoothing.
Let $\Delta=\Delta_Q$ be a strongly admissible collection of  RT-difference operators, 
for instance the ones constructed in Lemma \ref{lem_q0}.
It is not difficult to construct a smooth function $q'$ as a linear combination of $q^\alpha=q_1^{\alpha_1}\ldots q_n^{\alpha_n}$, $|\alpha|=a$, 
such that $\chi q/q'$ is smooth on $G$.
We check easily that
$$
 \Op(\Delta_{q\chi } \sigma)\phi (x)= 
 \Op(\Delta_{q'}\sigma) (\psi_x \phi)(x)
$$
where $\psi_x(y) = \chi q/q' (y^{-1}x)$,
thus by the Sobolev embedding
(cf. Lemma \ref{lem_sob_embedding}), 
\begin{eqnarray*}
&&\|\Op(\Delta_{q\chi } \sigma)\phi \|_{L^2(G)}^2
\leq
\int_G \sup_{x_1\in G} |\Op(\Delta_{q'}\sigma) (\psi_{x_1} \phi)(x)|^2 dx
\\&&\qquad\lesssim
\int_G
\int_{x_1\in G} |\Op(\Delta_{q'}\sigma) (X_{x_1}^\beta\psi_{x_1} \phi)(x)|^2 dx_1 dx
\\&&\qquad\lesssim
  \|\Op(\Delta_{q'}\sigma)\|_{\sL(H^{m-\rho a_0}, L^2(G))}^2
 \sum_{|\beta|\leq \lceil \frac n2\rceil} 
\int_{x_1\in G} \|X_{x_1}^\beta\psi_{x_1} \phi\|^2_{H^s} dx_1.
\end{eqnarray*}
We argue in a similar way as at the end of the proof of 
Lemma \ref{lem_prop_indep_Delta} to obtain
$$
 \sum_{|\beta|\leq \lceil \frac n2\rceil} 
\int_{x_1\in G} \|X_{x_1}^\beta\psi_{x_1} \phi\|^2_{H^s} dx_1
\lesssim_{s,\psi} 
  \|\phi\|_{H^s}^2,
  $$
and we conclude with 
$$
\|\Op(\Delta_{q'}\sigma)\|_{\sL(H^{m-\rho a_0}, L^2(G))}
\lesssim
\|\Delta_{q'}\sigma\|_{S^{m-\rho a_0}_{\rho,\delta},a_1,b_1}
\lesssim
\|\sigma\|_{S^{m-\rho a_0}_{\rho,\delta},a_1+a_0,b_1},
$$
by Corollary \ref{cor_bdd_sob_spaces}.
\end{proof}

\subsection{Commutators}

We adopt the following notation:
if $q\in \cD(G)$ and $D\in \Diff$, 
we denote by $L_q$ and $M_D$ the commutators defined via
$$
L_q T = q T -Tq
\quad\mbox{and}\quad 
M_D T= DT-TD,
$$
for any linear operator $T:\cD(G)\to\cD'(G)$.

Let us collect some easy properties for these commutators:
\begin{lemma}
\label{lem_obs_commutators}
\begin{itemize}
\item 
If $q$ is a smooth function, $T$ is an operator
$\cD(G)\to\cD'(G)$ and $D$ is a vector field then 
$$
M_{D} (qT) = (D q)  T + q M_{D}T
\quad\mbox{and}\quad
M_D L_q-L_qM_D = L_{Dq}.
$$
\item 
If $q$ is a smooth function 
and 
 if  $T:\cD(G)\to\cD'(G)$ is a linear continuous operator,
 then 
 $ \|L_q T \|_{\sL(L^2(G))}
\leq 2\|q\|_\infty 
\|T\|_{\sL(L^2(G))}$
since
$$
\|q T\|_{\sL(L^2(G))}
\leq \|q\|_\infty 
\|T\|_{\sL(L^2(G))}
\quad\mbox{and}\quad
\|T q\|_{\sL(L^2(G))}
\leq \|q\|_\infty 
\|T\|_{\sL(L^2(G))}.
$$
More generally, for any $s_1,s_2\in \bR$,  we have
 $ \|L_q T \|_{\sL(H^{s_1},H^{s_2})}
\leq 2C_{q,s_1,s_2}
\|T\|_{\sL(H^{s_1},H^{s_2})}$
since
$$
\max (\|q T\|_{\sL(H^{s_1},H^{s_2})},\|T q\|_{\sL(H^{s_1},H^{s_2})})
\leq C_{q,s_1,s_2}
\|T\|_{\sL(H^{s_1},H^{s_2})}.
$$
\end{itemize}
\end{lemma}
\begin{proof}[Lemma \ref{lem_obs_commutators}]
The first part is easily checked by direct computations.
The second part follows from the continuity of 
$\phi \to q\phi$ on any $H^s$ for any $q\in \cD(G)$.
\end{proof}

The Leibniz properties yield:
\begin{lemma}
\label{lem_commutator}
\begin{enumerate}
\item 
Let $\Delta=\Delta_Q$ be a collection of difference operators 
satisfying the Leibniz-like property
as in Definition \ref{def_leibniz}.
Then, for any continuous symbol $\sigma$,
we have:
$$
L_{q_j}\Op(\sigma) 
=
\Op(\Delta_{q_j} \sigma)+
\sum_{1\leq l,k \leq n_\Delta} c_{l,k}^{(j)}
\Op(\Delta_{q_k} \sigma)  q_l
$$
and
$$
\Op(\Delta_{q_j} \sigma)
=
-L_{\tilde q_j}\Op(\sigma) 
-\sum_{1\leq l,k \leq n_\Delta} c_{l,k}^{(j)} q_k
 L_{\tilde q_l}\Op(\sigma), 
$$
with the same  coefficients $c_{l,k}^{(j)}\in \bC$
as in Definition \ref{def_leibniz}, 
and $\tilde q_j(x)=q_j(x^{-1})$.
\item For any $X\in \fg$ and any smooth symbol $\sigma$,
we have
$$
M_{\tilde X} \Op(\sigma) =\Op(\tilde X \sigma).
$$
\end{enumerate}
\end{lemma}
\begin{proof}
For the first formula, 
we apply 
\eqref{eq_def_leibniz_q} to 
$q_j(x) = q_j(y \ y^{-1}x)$ in
$$
L_{q_j}\Op(\sigma) \phi(x)
=
\int_G \left( 
q_j(x)\phi(y) \kappa_x(y^{-1} x) 
-q_j(y)\phi(y) \kappa_x(y^{-1} x) 
\right)dy.
$$
For the second formula, 
we apply 
\eqref{eq_def_leibniz_q} to 
$q_j(y^{-1}x)$ in 
$$
\Op(\Delta_{q_j}\sigma) \phi(x)
=
\int_G \phi(y) q_j (y^{-1}x) \kappa_x(y^{-1}x) dy
$$
and we have
$$
\Op(\Delta_{q_j}\sigma) 
=
\Op(\sigma) \tilde q_j 
+q_j \Op(\sigma)
+\sum_{1\leq l,k \leq n_\Delta} c_{l,k}^{(j)} q_k
\Op(\sigma) \tilde q_l.
$$
We write 
$\Op(\sigma) \tilde q_l = (\tilde q_l-L_{\tilde q_l})\Op(\sigma)$
and observe that
$$
\sum_{1\leq l,k \leq n_\Delta} c_{l,k}^{(j)} q_k\tilde q_l = -(q_j +\tilde q_j),
$$
having applied \eqref{eq_def_leibniz_q} to $x, y=x^{-1}$.
Thus we obtain:
\begin{eqnarray*}
\Op(\Delta_{q_j}\sigma) 
&=&
\Op(\sigma) \tilde q_j 
+q_j \Op(\sigma)
+\sum_{1\leq l,k \leq n_\Delta} c_{l,k}^{(j)} q_k
(-L_{\tilde q_l}+\tilde q_l)\Op(\sigma)
\\
&=&
\Op(\sigma) \tilde q_j 
+q_j \Op(\sigma)
-(q_j +\tilde q_j)  \Op(\sigma)
-\sum_{1\leq l,k \leq n_\Delta} c_{l,k}^{(j)} q_kL_{\tilde q_l}\Op(\sigma)
\\
&=&
-L_{\tilde q_j}\Op(\sigma) 
-\sum_{1\leq l,k \leq n_\Delta} c_{l,k}^{(j)} q_kL_{\tilde q_l}\Op(\sigma).
\end{eqnarray*}

For the second part, 
we see 
\begin{eqnarray*}
\tilde X_x\Op(\sigma) \phi(x)
&=&
\tilde X_{x}\{ \phi*\kappa_x(x)\}
=
\tilde X_{x_1=x} \phi*\kappa_{x_1}(x)
+
\tilde X_{x_2=x} \phi*\kappa_{x}(x_2)
\\
&=&
\phi*\tilde X_{x_1=x}  \kappa_{x_1}(x)
+
(\tilde X \phi)*\kappa_{x}(x)
=
\Op(\tilde X\sigma)\phi + \Op(\sigma)(\tilde X\phi).
\end{eqnarray*}
\end{proof}

If $\Delta=\Delta_Q$ is a collection of RT-difference operators
and if $X_1,\ldots,X_n$ form a basis of $\fg$, 
then we set
\begin{equation}
\label{eq_not_Lalpha_Mbeta}
L^\alpha_\Delta:=L_q^\alpha := L_{q_1}^{\alpha_1}\ldots L_{q_{n_\Delta}}^{\alpha_{n_\Delta}},
\quad \alpha\in \bN_0^{n_\Delta},
\quad\mbox{and}\quad
M^\beta_{\tilde X}:=M_{\tilde X_1}^{\beta_1}\ldots M_{\tilde X_n}^{\beta_n}
\quad \beta\in \bN_0^{n}.
\end{equation}

\begin{proposition}
\label{prop_bdd_commutators}
Let $1\geq \rho\geq \delta\geq 0$ with $\delta\not=1$
and $m\in \bR$.
If $T\in \Psi^m_{\rho,\delta}$, 
then $L^\alpha M^\beta T$  extends boundedly in an operator
from $H^{m - \rho |\alpha|+\delta|\beta|}$ to $L^2(G)$
for each $\alpha\in \bN_0^{n_\Delta},\beta\in \bN_0^{n}$
and for $L^\alpha_\Delta$, $M^\beta_{\tilde X}$ as defined in \eqref{eq_not_Lalpha_Mbeta}
where $\Delta$ is any  collection of RT-difference operators. 
Moreover 
$$
\|L_\Delta^\alpha M_{\tilde X}^\beta T\|_{\sL(H^{m - \rho |\alpha|+\delta|\beta|},L^2(G))}
\leq C \|T\|_{\Psi^m_{\rho,\delta},a,b}
$$
where the constant $C>0$ and the semi-norm $\|\cdot\|_{\Psi^m_{\rho,\delta},a,b}$ are independent on $T$ (but may depend on $\alpha,\beta,\Delta$ and the choice of basis for $\fg$).
\end{proposition}

If $\Delta=\Delta_Q$ satisfies a Leibniz-like property, 
then  
Corollary \ref{cor_bdd_sob_spaces} and Lemma \ref{lem_commutator}  imply Proposition \ref{prop_bdd_commutators}.
In the general case, we have to use
Lemma \ref{lem_Deltaqsigma_HL2} and the ideas of its proof.

\begin{proof}[Proposition \ref{prop_bdd_commutators}
when $\rho\not=0$]
As we can always enlarge the collection $\Delta$, 
we may assume $\Delta$ to be strongly admissible.
Let $\chi\in \cD(G)$ be valued in $[0,1]$ and  such that 
$\chi|_{B(\epsilon_0/2)} \equiv 1$ and $\chi|_{B(\epsilon_0)^c} \equiv 0$.
We can always write $\sigma=\Delta_\chi \sigma+\Delta_{1-\chi}\sigma$.
As the kernel associated with $\Delta_{1-\chi}\sigma$ is smooth 
(see Proposition \ref{prop_kernel_regularity}), this symbol is smoothing 
and the operator   $L_\Delta^\alpha M_{\tilde X}^\beta \Op(\Delta_{1-\chi}\sigma)$ is also smoothing. In particular it maps 
any Sobolev space to any Sobolev space continuously by 
Corollary \ref{cor_bdd_sob_spaces}.
For $\Delta_\chi \sigma$, 
we  define the function $(x,y)\mapsto \psi_x(y)$ via
$$
\psi_x(y)= (q_1(x)-q_1(y))^{\alpha_1}\ldots (q_{n_\Delta}(x)-q_{n_\Delta}(y))^{\alpha_{n_\Delta}} (q^\alpha (y^{-1} x))^{-1} \chi(y^{-1}x)
\quad x\not = y,
$$
and extend it smoothly to $G\times G$.
We check easily:
$$
M_{\tilde X}^\beta L^\alpha \Op(\Delta_\chi \sigma)\phi (x)= \Op(\Delta_{q'} \tilde X^\beta\sigma) (\psi_x \phi)(x),
$$
thus by the Sobolev embedding (cf. Lemma \ref{lem_sob_embedding}), 
\begin{eqnarray*}
&&\|L^\alpha \Op(\Delta_\chi  \tilde X^\beta\sigma)\phi\|_{L^2(G)}^2
\leq
\int_G \sup_{x_1\in G} |\Op(\Delta_{q'}\tilde X^\beta\sigma) (\psi_{x_1} \phi)(x)|^2 dx
\\ &&\qquad\lesssim
\int_G
\int_{x_1\in G} |\Op(\Delta_{q'}\tilde X^\beta\sigma) (X_{x_1}^\beta\psi_{x_1} \phi)(x)|^2 dx_1 dx
\\ &&\qquad\lesssim
  \|\Op(\Delta_{q'}\tilde X^\beta\sigma)\|_{\sL(H^s, L^2(G))}^2
 \sum_{|\beta|\leq \lceil \frac n2\rceil} 
\int_{x_1\in G} \|X_{x_1}^\beta\psi_{x_1} \phi\|^2_{H^s} dx_1
\\ &&\qquad\lesssim
  \|\Op(\Delta_{q'}\tilde X^\beta\sigma)\|_{\sL(H^s, L^2(G))}^2
  \|\phi\|_{H^s}^2,
\end{eqnarray*}
by Lemma \ref{lem_Deltaqsigma_HL2} for the operator norm
and 
by arguing as at the end of the proof of 
Lemma \ref{lem_prop_indep_Delta} for the $H^s$-norm.
We then conclude using Lemma \ref{lem_Deltaqsigma_HL2}.
\end{proof}

\begin{proof}[Proposition \ref{prop_bdd_commutators}
when $\rho=0$]
The case $\rho=\delta=m=0$ follows from 
Proposition \ref{prop_L2bdd} and Lemma \ref{lem_obs_commutators}.
For $m\not=0$, we  observe 
\begin{equation}
\label{eq_LqT1T2}
L_q(T_1T_2) = (L_q T_1) T_2 + T_1 (L_q T_2),
\end{equation}
for any $q\in \cD(G)$ and any operator
$T_1,T_2$ (for instance both $\cD'(G)\to \cD'(G)$ or $\cD'(G)\to \cD'(G)$, or alternatively
$T_1:\cD(G)\to \cD'(G)$ and $T_2:\cD(G)\to\cD(G)$).
Setting $T_{m,\beta} = M^\beta_{\tilde X}T (\id+\cL)^{- m/2}$,
this implies that 
$M^\beta_{\tilde X} L_\Delta^\alpha T$
is a linear combination of 
$$
 (L_\Delta^{\alpha _1} T_{m,\beta})
( L_\Delta^{\alpha _2}
(\id+\cL)^{m/2})
=
(L_\Delta^{\alpha _1} T_{m,\beta})
( L_\Delta^{\alpha _2}
(\id+\cL)^{m/2}), 
\quad|\alpha_1|+|\alpha_2|=|\alpha|
$$
We may apply 
Proposition \ref{prop_bdd_commutators} to the operator 
$(\id+\cL)^{-m/2} \in \Psi^{-m}_{1,0}$
and $T_{m,\beta}=\Op(\tilde X^\beta \sigma )(\id+\cL)^{- m/2} \in \Psi^0_{0,0}$, as the cases of operators in $\Psi^{-m}_{1,0}$ and 
$\Psi^0_{0,0}$ have already been proved.
This shows that $M^\beta_{\tilde X} L_\Delta^\alpha T \in \sL(H^m)$
and concludes the proof of Proposition \ref{prop_bdd_commutators}.
\end{proof}

\subsection{Commutator characterisation}

Importantly, the converse to Proposition \ref{prop_bdd_commutators} holds:

\begin{proposition}
\label{prop_converse_commutator}
Let $1\geq \rho\geq \delta\geq 0$ with $\delta\not=1$
and $m\in \bR$.
Let $\Delta=\Delta_Q$ be a strongly admissible collection of RT-difference operators. 
If $T:\cD(G)\to \cD'(G)$ is a continuous operator 
satisfying 
$L^\alpha M_{\tilde X}^\beta T\in \sL(H^{m - \rho |\alpha|+\delta|\beta|},L^2(G))$ for any $\alpha\in \bN_0^{n_\Delta},\beta\in \bN_0^{n}$,
then $T\in \Psi^m_{\rho,\delta}$.
Moreover for any semi-norm $\|\cdot\|_{\Psi^m_{\rho,\delta},a,b}$, 
we have
$$
\|T\|_{\Psi^m_{\rho,\delta},a,b}
\leq C 
\max_{|\beta|\leq b +\lceil \frac n2\rceil, |\alpha|\leq a}
\|L^\alpha M_{\tilde X}^\beta T\|_{\sL(H^{m - \rho |\alpha|+\delta|\beta|},L^2(G))},
$$
where the constant $C>0$ is independent of $T$ (but may depend on  $\|\cdot\|_{\Psi^m_{\rho,\delta},a,b},\Delta$, $G$).
\end{proposition}

\begin{proof}[Proposition \ref{prop_converse_commutator}
when $m=\rho=\delta=0$]
Let $T$ be a linear operator which is $\cD(G)\to \cD'(G)$-continuous 
and such that 
$L^\alpha M_{\tilde X}^\beta T \in \sL(L^2(G))$
 for any $\alpha\in \bN_0^{n_\Delta},\beta\in \bN_0^{n}$.
 
We can associate a symbol $\sigma$ via \eqref{eq_sigma_T}
in a distributional sense, see Remark \ref{rem_symbol_distrib}:
$$
(\sigma(x,\pi) u,v)_{\cH_\pi}
=
 \big( (T\pi)(x) u,\pi(x) v\big)_{\cH_\pi}, 
\quad
u,v\in \cH_\pi.
$$
Given our hypotheses on $T$, for each $\pi\in \Gh$
this defines 
$\sigma(\cdot,\pi) \in L^2(G,\cH_\pi)$, 
that is, a square-integrable function defined on $G$
with values in $\sL(\cH_\pi)$ (or after a choice of basis, 
in the space of complex $d_\pi\times d_\pi$-matrices).
%
The Cauchy-Schwartz inequality and easy manipulations yield:
$$
\|(\sigma(\cdot,\pi) u,v)_{\cH_\pi}\|_{L^2(G)}
\leq
  \|T\|_{\sL(L^2(G))}|u |_{\cH_\pi} |v|_{\cH_\pi}.
$$
More generally
we may adapt the proof of Lemma \ref{lem_commutator} so that it holds for  distributional kernels and
we obtain  for any $\beta\in \bN_0^n$
\begin{equation}
\label{eq_pf_prop_converse_commutator_m=rho=delta=0_1}
\|\tilde X^\beta(\sigma(\cdot,\pi) u,v)_{\cH_\pi}\|_{L^2(G)}
\leq
  \|M_{\tilde X}^\beta T\|_{\sL(L^2(G))}|u |_{\cH_\pi} |v|_{\cH_\pi}.
\end{equation}
Denoting $\bS_{\cH_\pi}=\{u\in \cH_\pi, |u|_{\cH_\pi}=1\}$ the unit sphere on $\cH_\pi$,
the Sobolev embedding (cf. Lemma \ref{lem_sob_embedding}) yields:
\begin{eqnarray*}
&&\sup_{(x,\pi)\in G\times \Gh}
\|\sigma(x,\pi)\|_{\sL(\cH_\pi)} 
=
\sup_{\pi\in \Gh }
\sup_{u,v\in \bS_{\cH_\pi} }
\sup_{x\in G}
| (\sigma(x,\pi) u, v)_{\cH_\pi}|
\\&&\qquad\lesssim 
\sup_{\pi\in \Gh }
\sup_{u,v\in \bS_{\cH_\pi} }
\max_{|\beta|\leq \lceil \frac n2\rceil}
\|\tilde X^\beta (\sigma(\cdot,\pi) u, v)_{\cH_\pi}\|_{L^2(G)}
\lesssim 
\max_{|\beta|\leq \lceil \frac n2\rceil}
  \|M_{\tilde X}^\beta T\|_{\sL(L^2(G))},
\end{eqnarray*}
having used \eqref{eq_pf_prop_converse_commutator_m=rho=delta=0_1}.
This also implies that, 
for each $\pi\in \Gh$,
the mapping $G\ni x\mapsto \sigma(x,\pi) \in \sL(\cH_\pi)$ is continuous.
Moreover, applying this to  $M_{\tilde X}^{\beta_0} T$ for any $\beta_0\in \bN_0^n$, we obtain that $G\ni x\mapsto \tilde X^{\beta_0} \sigma(x,\pi) \in \sL(\cH_\pi)$ is continuous and that
\begin{equation}
\label{eq_pf_prop_converse_commutator_m=rho=delta=0_2}
\sup_{(x,\pi)\in G\times \Gh}
\|\tilde X^{\beta_0}\sigma(x,\pi)\|_{\sL(\cH_\pi)} 
\lesssim 
\max_{|\beta|\leq \lceil \frac n2\rceil +|\beta_0|}
  \|M_{\tilde X}^\beta T\|_{\sL(L^2(G))}.
\end{equation}
Hence the mapping $G\ni x\mapsto \sigma(x,\pi) \in \sL(\cH_\pi)$ is smooth. 

Combining Lemma \ref{lem_commutator}
 with \eqref{eq_pf_prop_converse_commutator_m=rho=delta=0_2} (with $\beta_0=0$),
we obtain:
\begin{eqnarray*}
\sup_{\substack{x \in G\\ \pi\in  \Gh}}
\|\Delta_{\tilde q_j} \sigma(x,\pi)\|_{\sL(\cH_\pi)} 
&\lesssim &
\max_{|\beta|\leq \lceil \frac n2\rceil }
  \|M_{\tilde X}^\beta \big(-L_{q_j}\Op(\sigma) 
-\sum_{1\leq l,k \leq n_\Delta} c_{k,l}^{(j)} q_k
 L_{q_l}\Op(\sigma)\big)\|_{\sL(L^2(G))}\\
&\lesssim &
\max_{|\beta|\leq \lceil \frac n2\rceil, |\alpha|= 1 }
  \|M_{\tilde X}^\beta L^\alpha T \|_{\sL(L^2(G))},
\end{eqnarray*}
by Lemma \ref{lem_obs_commutators}.
More generally, using  the same methods as above, we obtain recursively
$$
\sup_{(x,\pi)\in G\times\Gh}
\|\tilde \Delta_Q^{\alpha_0} \tilde X^{\beta_0}_x \sigma(x,\pi)\|_{\sL(\cH_\pi)} 
\lesssim 
\max_{|\beta|\leq \lceil \frac n2\rceil +|\beta_0|, |\alpha|\leq |\alpha_0|}
  \|M_{\tilde X}^\beta L^\alpha T \|_{\sL(L^2(G))},
$$
for any $\alpha_0\in \bN_0^{n_\Delta}$, $\beta_0 \in \bN_0^n$
where $\tilde \Delta := \{\Delta_{\tilde q_j}\}_{j=1}^{n_\Delta}$ is also a strongly admissible collection of RT-difference operators.
This shows that $\sigma\in S^0_{0,0}$.
\end{proof}

\begin{proof}[Rest of the proof of Proposition \ref{prop_converse_commutator}]
Let $T$ be as in the statement.
Then $M^\beta_{\tilde X}T(\id+\cL)^{-m/2} \in \sL(L^2(G))$
for any $\beta$ and  by the first part of the proof of the case  $\rho=\delta=m=0$,
the symbol of the operator
 $T(\id+\cL)^{-m/2}$ 
 satisfies
\eqref{eq_pf_prop_converse_commutator_m=rho=delta=0_2}
  with 
$x\mapsto \Op^{-1}(T(\id+\cL)^{-m/2})(x,\pi)$ smooth.
We may define the symbol of the operator $T$ to be 
$\sigma:=\Op^{-1}(T(\id+\cL)^{-m/2}) (1+\lambda_\pi)^{m/2}$.

We set  $T_{\alpha_0,\beta_0,m}:=
\Op(\Delta_Q^{\alpha_0} \tilde X^{\beta_0}\sigma)  (\id+\cL)^{(-m +\rho|\alpha_0| -\delta|\beta_0|)/2}$.
Lemmata \ref{lem_obs_commutators} and \ref{lem_commutator}
imply that  $T_{\alpha_0,\beta_0,m}$  is a linear combination of $p_{\alpha'} (L^{\alpha'} M^{\beta_0} T)  (\id+\cL)^{-m/2}$ for some $p_{\alpha'}\in \cD(G)$, $|\alpha'|=|\alpha_0|$.
One shows recursively that $L^\alpha_\Delta M^\beta_{\tilde X} T_{\alpha_0,\beta_0,m} \in \sL(L^2(G))$ for any multi-indices $\alpha$ and $\beta$ using the `almost commutation' of $L^\alpha$ and $M^\beta$ (see Lemma \ref{lem_commutator}), \eqref{eq_LqT1T2} and $L^q (\id +\cL)^{m'/2}\in \sL(H^s, H^{s-m'})$.
This is routine but lengthy and left to the reader.
Hence  we can apply Proposition \ref{prop_converse_commutator} for the case $\rho=\delta=m=0$ which is already proven: 
we have $T_{\alpha_0,\beta_0,m} \in \Psi^0_{0,0}$ and 
\begin{eqnarray*}
&&\sup_{(x,\pi)\in G\times\Gh}
\|\Delta^{\alpha_0} \tilde X^{\beta_0}\sigma (x,\pi) \|_{\sL(\cH_\pi)} 
(1+\lambda_\pi)^{\frac{-m +\rho|\alpha_0| -\delta|\beta_0|}2}
\\&&\qquad\lesssim 
\max_{|\beta|\leq \lceil \frac n2\rceil}
  \|M_{\tilde X}^\beta  T_{\alpha_0,\beta_0,m} \|_{\sL(L^2(G))}.
\\&&\qquad
\lesssim 
\max_{|\beta|\leq \lceil \frac n2\rceil +|\beta_0|, |\alpha|\leq |\alpha_0|}
  \| L^{\alpha}_\Delta M_{\tilde X}^\beta   T \|_{\sL(H^{-m +\rho|\alpha|- \delta |\beta|},L^2(G))},
\end{eqnarray*}
by Lemmata \ref{lem_obs_commutators} and \ref{lem_commutator}, 
together with \eqref{eq_LqT1T2} and $L^q (\id +\cL)^{m'/2}\in \sL(H^s, H^{s-m'})$. Thus $\sigma\in S^m_{\rho,\delta}$ and 
this  concludes the proof of Proposition \ref{prop_converse_commutator}.
\end{proof}

Because of Lemma \ref{lem_obs_commutators} and of the inclusions $H^{s_1} \subset H^{s_2}$, $s_1\geq s_2$,
the order for the commutators $L_{q_j}$ and $M_{\tilde X_j}$
for $L^\alpha M^\beta T$ in 
Propositions \ref{prop_bdd_commutators} and  \ref{prop_converse_commutator} 
could be arbitrarily changed.
Furthermore, we could replace the basis of right-invariant vector fields $\tilde X_1, \ldots, \tilde X_n$, with any other collection vector fields $D_1,\ldots , D_d$ generating the $\cD(G)$-module of 
$\Diff^1(G)$.  
Then we would adopt the notation 
$M^\beta_D:=M_{D_1}^{\beta_1}\ldots M_{D_d}^{\beta_d}
\quad \beta\in \bN_0^{d}$.
Hence we have obtained the following  characterisation of the operators in $\Psi^m_{\rho,\delta}$:

\begin{corollary}
\label{cor_commutator_characterisation}
Let $1\geq \rho\geq \delta\geq 0$ with $\delta\not=1$,
and $m\in \bR$.
Let $T:\cD(G)\to\cD'(G)$ be a continuous operator.
The operator $T$ is in $\Psi^m_{\rho,\delta}$ if and only if 
there exists a strongly admissible collection $\Delta$ of RT-difference operators and $\{D_1,\ldots, D_d\}$ a family of smooth vector fields 
generating the $\cD(G)$-module $\Diff^1(G)$.
such that 
$L^\alpha_\Delta M_D^\beta T \in \sL( H^{m - \rho |\alpha|+\delta|\beta|}, L^2(G))$
for any $\alpha\in \bN_0^{n_\Delta},\beta\in \bN_0^{n}$.
In this case 
$L^\alpha_\Delta M_D^\beta T \in \sL( H^{m - \rho |\alpha|+\delta|\beta|}, L^2(G))$ for any collection $\Delta=\Delta_Q$ of RT-difference operators
and any family  $\{D_1,\ldots, D_d\}$  of smooth vector fields on $G$
and any multi-indices $\alpha$ and $\beta$.
\end{corollary}

This commutator characterisation is almost the same as the characterisation of the H\"ormander classes of operators on a manifold.
This was already  explained in \cite[Section 10.7.2]{ruzhansky+turunen_bk}
(but see Remarks \ref{rem_RT_pb_comp} and \ref{rem_RT_pb_L2}). 
In this paper,  we  obtain:

\begin{corollary}
\label{cor_hormander}
If  $\rho$ and $\delta$ satisfy  
$1\geq \rho>\delta\geq 0$ and $\rho\geq 1-\delta$,
then $\Psi^m_{\rho,\delta}$ coincides with the H\"ormander class 
$\Psi^m_{\rho,\delta}(G,loc)$.
\end{corollary}

Recall that a linear operator $T:\cD(G)\to \cD'(G)$ is in 
$\Psi^m_{\rho,\delta}(G,loc)$ when 
for any $\phi,\psi\in \cD(G)$ supported in charts of $G$,
the  operator $\phi T \psi : f\mapsto \phi T ( \psi f) $ 
viewed as an operator $\tilde T_{\phi,\psi}$ on $\bR^n$
 is in $\Psi^m_{\rho,\delta}(\bR^n)$.
The hypotheses  on $\rho$ and $\delta$, 
that is, $1\geq \rho >\delta\geq 0$, $\rho\geq 1-\delta$, 
 ensure that the operators in $\Psi^m_{\rho,\delta}(G,loc)$ are  well defined using changes of charts.

\begin{proof}[Corollary \ref{cor_hormander}]
Let $T \in \Psi^m_{\rho,\delta}$. Let also $\phi,\psi\in \cD(G)$ supported in charts of $G$.
By Lemma \ref{lem_Psi0_inv_left_translation} and the linearity of $T$, we may assume that $\psi$ is supported in the `small' neighbourhood $B(\epsilon_0/2)$ of $e_G$ and use the exponential mapping there as chart.
We apply Corollary \ref{cor_commutator_characterisation} with  a basis of right invariant vector fields and 
 the collection $\Delta=\Delta_Q$ constructed in Lemma \ref{lem_q0}.
 This implies that $\tilde T_{\phi,\psi}$ satisfies the hypotheses of Beal's characterisation of pseudo-differential operators 
 (for the commutators of $\partial_{x_i}$ and $x_j$) \cite{beals_77}.
 Thus $\tilde T_{\phi,\psi}\in \Psi^m_{\rho,\delta}(\bR^n)$
 and $T \in \Psi^m_{\rho,\delta}(G,loc)$. 
  The converse holds for the same reasons.
\end{proof}

\appendix

\section{Multipliers of the Laplace operator}
\label{sec_multipliers}

This appendix is devoted to the proof of 
Propositions  \ref{prop_mult} and \ref{prop_mult_t}. 
We will use  `classical' methods to estimates weighted norms of kernels of spectral $\cL$-multipliers using the heat kernels.

First we will reformulate Propositions \ref{prop_mult_t}
and \ref{prop_mult} into Proposition \ref{prop_app_mult_t} 
and Corollary \ref{cor_mult}
below using the notation of vanishing order of a function which we now define precisely.

\begin{lemma}
\label{lem_q_vanish_order}
Let $q\in \cD(G)$ and $a\in \bN$.
The following are equivalent:
\begin{enumerate}
\item For all $\alpha\in \bN_0^n$ with  
$|\alpha|< a$, then $X^\alpha q (e_G)=0$.
\item For any differential $D \in \Diff^k$, $k<a$, we have
 $D q (e_G)=0$.
\item 
\label{item_lem_q_vanish_order_estimate}
There exists a constant $C_q$ such that 
for all $x\in G$, we have 
$|q(x)|\leq C_q |x|^a$.
\end{enumerate}
\end{lemma}

\begin{definition}
\label{def_q_vanish_order}
If $q\in \cD(G)$ satisfies the equivalent properties of Lemma \ref{lem_q_vanish_order}, 
then we say that $q$ vanishes at $e_G$ up to order $a-1$.
We extend this to $a\leq 0$:
a smooth function $q$ vanishes at $e_G$ up to order $a-1$ if $q(e_G)\not=0$.
\end{definition}

  We reformulate Proposition \ref{prop_mult_t} into the following property:
    \begin{proposition}
\label{prop_app_mult_t}
Let  $m\in \bR$ and $a\in \bN_0$.
For any $q\in \cD(G)$ vanishing at $e_G$ up to order $a-1$, 
there exists $d\in \bN_0$
such that for all
 $f\in C^d[0,\infty)$ satisfying $\|f\|_{\cM_{m/2}, d}<\infty$, we have
$$
\forall \pi\in\Gh, \ t\in (0,1)\qquad
\| \Delta_q \{f(t\lambda_\pi)\}\|_{\sL(\cH_\pi)} 
\leq C t^{\frac {m}  2}
(1+\lambda_\pi)^{\frac{m -a}2},
$$
where the constant $C$ may be chosen as $C' \|f\|_{\cM_{m/2}, d}$
with $C'$ depending only on $m,q,a$ and the group $G$ but not on $f,t,\pi$.
\end{proposition}

In the statement above, we have used the following notation for $d\in \bN_0$ and $m'\in \bR$:
$$
\|f\|_{\cM_{m'}, d} :=\sup_{\lambda\geq 0, \ \ell =0,\ldots, d}
(1+\lambda)^{-m' +\ell} |\partial^\ell_\lambda f(\lambda)|.
$$

Proposition \ref{prop_app_mult_t} easily implies:
\begin{corollary}
\label{cor_mult}
Let $m\in \bR$ and $a\in \bN_0$.
For any $q\in \cD(G)$  vanishing at $e_G$ up to order $a-1$, 
there exists $C$ 
such that for any function  $f:\spec(\cL) \to \bC$
satisfying $\sup_{\lambda\in \spec(\cL)} 
(1+\lambda)^{- \frac {m}  2} 
| f(\lambda)|<\infty$
and $\pi\in\Gh$, we have
$$
\| \Delta_q f(\lambda_\pi)\|_{\sL(\cH_\pi)} 
\leq C (1+\lambda_\pi)^{\frac{m +a}2}
\sup_{\lambda\in \spec(\cL)} 
(1+\lambda)^{- \frac {m}  2} 
| f(\lambda)|  .
$$
\end{corollary}

\begin{proof}
We can construct the function  
$$
\tilde f(\lambda):= \sum_{\mu\in\spec(\cL)} f(\mu) \phi_\mu(\lambda),
$$
where the functions $\phi_\mu\in \cD(\bR)$ are bump functions valued in $[0,1]$ with disjoint supports and such that $\phi_\mu(\mu)=1$.
We have
$\tilde f\in \cC^\infty[0,\infty)$, 
$f(\cL)=\tilde f(\cL)$
and 
$$\sup_{\lambda\in \spec(\cL)} (1+\lambda)^{- \frac {m}  2}  |f(\lambda)|
=
\sup_{\lambda\in \spec(\cL)}(1+\lambda)^{- \frac {m}  2}  |\tilde f(\lambda)|.
$$
Hence we may assume $f=\tilde f \in \cC^\infty[0,\infty)$.

More precisely, we can choose the bump functions as
$$
\phi_\mu(\lambda) = 
\chi\left(\frac{|\lambda-\mu|}{\max (\mu, 1, \delta_e)}\right),
$$
where $\chi\in \cD(\bR)$ is a fixed function such that 
$$
0\leq \chi\leq 1, \quad \supp \chi \subset [-1,1],
\quad \chi\big|_{[-\frac12,\frac12]}\equiv1 ,
$$
and where $\delta_0:=\min
\{|\lambda_1-\lambda_2|, \lambda_1\not=\lambda_2\in \spec(\cL)\}$
is the minimum distance between two distinct eigenvalues of $\cL$.
In this case, we have 
$$
\|f\|_{\cM_{m/2}, d}  \lesssim_d \sup_{\lambda \in \spec(\cL) }
(1+\lambda)^{- \frac {m}  2} |f(\lambda)|.
$$
We then apply Proposition \ref{prop_app_mult_t} to $\tilde f$ and, for instance, $t=1$.
\end{proof}

Corollary \ref{cor_mult} easily implies the first and second part of Proposition \ref{prop_mult}.
The last part follows from the following remark:
it is possible to extend the proof presented in this appendix to symbols 
 depending on $x$ in the following way:
$\sigma(x,\pi)=f(x,\lambda_\pi)$, 
for a function $f$ very regular in $x\in G$.

Hence this section is devoted to the proof of Proposition \ref{prop_app_mult_t}, 
which will be presented in  \ref{subsec_pf_prop_app_mult_t}.
Before this, we present its main tool, the heat kernel, whose properties will be recalled in \ref{subsec_pt}.
We also state and prove  technical lemmata in \ref{subsec_technical_lemmata}
and \ref{subsec_pf_lem_prop_mult_1}.

\subsection{The heat kernel}
\label{subsec_pt}

The heat kernel, i.e. the kernel of the operator $e^{-t\cL}$:
$$
p_t:=e^{-t\cL}\delta_e, \quad t>0,
$$
is a positive smooth function on $G$ which satisfies 
$$
\forall s,t>0\qquad
\int_G p_t(x) dx=1,
\quad p_t(x^{-1})=p_t(x),
\quad\mbox{and}\quad p_t*p_s=p_{t+s}.
$$
and the following estimates \cite{varo} 
\begin{eqnarray}
|p_t(x)|&\leq& C V(\sqrt t)^{-1} e^{-\frac{|x|^2}{Ct}}, 
\quad x\in G, \quad t>0,
\label{eq_heat_kernel}\\
|X^\alpha p_t(x)|&\leq& C \sqrt t^{-n -|\alpha|}e^{-\frac{|x|^2}{Ct}},
\quad x\in G, \quad 0<t\leq 1.
\label{eq_heat_kernel_der}
\end{eqnarray}
In these estimates, $C$ is independent of $x\in G$ and $t>0$
but may depend on the multi-index $\alpha\in \bN_0^n$.
$V(r)$ denotes the volume of the ball centred at $e_G$ and of radius $r>0$.
It may be estimated via
\begin{equation}
\label{eq_V(r)}
V(r):= |B(r)|
 \sim \left\{\begin{array}{ll}
r^n 
&\mbox{if}\ r\in (0,\epsilon_0),\\
1 
&\mbox{if}\ \epsilon_0<r\leq R_0. \\
\end{array}\right.
\end{equation}
and  \cite[p.111]{varo}
\begin{equation}
\label{eq_int_e-x2t}
\int_G e^{-\frac{|x|^2}{Ct}} dx
\leq C V(\sqrt t).
\end{equation}

For the sake of completeness, let us sketch the proof of the following well known facts:
\begin{lemma}
\label{lem_sob_embedding}
If $s>n/2$, then the kernel $\cB_s$ of the operator $(\id+\cL)^{-s/2}$ is square integrable and 
the continuous inclusion $H^s \subset \cC(G)$ holds.
\end{lemma}

\begin{proof}[Sketch of the proof of Lemma \ref{lem_sob_embedding}]
If $s>0$,
the properties of the Gamma function and of the heat kernel together with  the spectral calculus of $\cL$ 
imply that the kernel $\cB_s$ of the operator $(\id+\cL)^{-s/2}$ is the integrable function given via:
$$
\cB_s
=\frac 1{\Gamma (s/2)}
\int_{t=0}^\infty t^{\frac s2 -1} e^{-t} p_t dt,
$$
and that we have
\begin{eqnarray*}
&&\|\cB_s\|_{L^2(G)}^2=
\cB_s * \cB_s^*(e)
=\frac 1{|\Gamma (s/2)|^2}
\int_{0}^\infty
\int_{0}^\infty  (t_1t_2)^{\frac s2 -1} e^{-(t_1+t_2)} p_{t_1+t_2} (e) dt_1dt_2
\\&&\qquad\leq
\frac 1{|\Gamma (s/2)|^2}
\int_{0}^\infty
\int_{0}^\infty  (t_1t_2)^{\frac s2 -1} e^{-(t_1+t_2)}  C(t_1+t_2)^{-\frac n2}
dt_1dt_2.
\end{eqnarray*}
It is not difficult to show that this last integral against $dt_1dt_2$ is finite whenever $s>n/2$.
The Sobolev embedding then follows easily from 
the fact that one can write $f=\{(\id+\cL)^{-s/2} f\}* \cB_s$ 
for any $f\in H^s$ with $s>n/2$.
\end{proof}

\subsection{Technical lemmata}
\label{subsec_technical_lemmata}

In this section, we state in Lemma \ref{lem_prop_mult_1}
the main step in the proof of Proposition \ref{prop_app_mult_t}
as well as two properties used in its proof in the next section.

Recall that  $f(\cL)\delta_e$ denotes the convolution kernel of the operator $f(\cL)$, see \eqref{eq_not_f(L)delta}.
\begin{lemma}
\label{lem_prop_mult_1}
\begin{enumerate}
\item 
Let $q \in \cD(G)$ and $m\in \bR$.
There exists $C=C_{q,m}$
such that for any continuous function  $f$ with support in $[0,2]$, 
 we have for any $t\geq \epsilon_0$
$$
\int_G | q(x)  \ f(t\cL)\delta_e(x)| dx 
\leq C 
\|f\|_{\infty}.
$$
\item 
Let $a\in \bN_0$ and $\beta\in \bN_0^n$.
For any $q\in \cD(G)$ vanishing up at $e_G$ to order $a-1$,
there exists $C=C_{q,a,\beta}$ and $d=d_{a,\beta}\in \bN$
such that for any function  $f\in C^d[0,\infty)$ with support in $[0,2]$, 
$\pi\in\Gh$
 we have for any $t\in (0,1)$
$$
\int_G | q(x)  X^\beta\{ f(t\cL)\delta_e\}(x)| dx 
\leq C t^{\frac {a-|\beta|}  2}
\max_{\ell =0,1,\ldots,d} 
\|f^{(\ell)}\|_{\infty}.
$$
\end{enumerate}
\end{lemma}

\begin{remark}
\label{rem_lem_prop_mult_1}

\begin{enumerate}

\item
\label{item_rem_lem_prop_mult_1_smooth}
It is not difficult to prove that, if $f$ is compactly supported in $[0,\infty)$, 
then the kernel of $f(\cL)$ is smooth
and thus the integrals intervening in Lemma \ref{lem_prop_mult_1} are finite. 
Indeed this follows readily from $\spec (\cL)\subset [0,\infty)$ being discrete
and the fact that the eigenspaces of $\cL$ are finite dimensional
and included in $\cD(G)$.  
However Lemma \ref{lem_prop_mult_1}  yields
 bounds for these integrals in terms of  $f$ and $t$
 which will be useful later.

\item
\label{item_rem_lem_prop_mult_1_Xq}
The second part of  
Lemma \ref{lem_prop_mult_1} implies that for any $q\in \cD(G)$ vanishing at $e_G$ up to order $a-1$, $\beta,\gamma\in \bN_0^n$, we have:
$$
\int_G | X^{\gamma} \{ q(x) X^{\beta}  \ f(t\cL)\delta_e\}(x)| dx
\leq C t^{\frac {a-|\beta| -|\gamma|}  2}
\max_{\ell =0,1,\ldots,d} 
\|f^{(\ell)}\|_{\infty},
$$
with the constant $C=C_{q,\beta,\gamma}>0$ independent of $f$.
This follows easily from
$$
X^\gamma (q\phi)(x) 
= \sum_{|\gamma_1|+|\gamma_2| =|\gamma|}
c_{\gamma_1,\gamma_2} X^{\gamma_1}q(x) \ X^{\gamma_2} \phi(x),
$$
for any reasonable function $\phi$  on $G$.
Indeed $X^{\gamma_1}q$ vanishing at $e_G$ up to order $a-1 -|\gamma|$. Here $\phi=f(t\cL)\delta_e$.
\end{enumerate}

\end{remark}

The two following  lemmata 
 will be useful in the proof of Lemma \ref{lem_prop_mult_1}
given in the next section.

\begin{lemma}
\label{lem_L2Brq}
Let $a\in \bN_0$.
For any $q\in \cD(G)$ vanishing at $e_G$ up to order $a-1$,
 there exists $C=C_{a,q}$ such that for any $r>0$ we have
$$
\|q\|_{L^2(B(r))}=\left(\int_{|x|<r} |q(x)|^2 dx\right)^{\frac 1 2}
\leq C \min (1, r^{a+\frac n2}).
$$
\end{lemma}
\begin{proof}[Lemma \ref{lem_L2Brq}]
We can estimate directly 
$\|q\|_{L^2(B(r))}\leq 
\|q\|_\infty$.
If $r$ is small, we can obtain a better estimate 
using Lemma \ref{lem_q_vanish_order} \eqref{item_lem_q_vanish_order_estimate}
and the fact the ball $B(\epsilon_0)$ yields a chart around the neutral element.
More precisely we have 
$$
\forall r\in (0,\epsilon_0) \quad
\|q\|_{L^2(B(r))}^2
\leq 
\int_{|x|<r} C_q^2 |x|^{2a} dx
\lesssim 
C_q^2  \int_{s=0}^{r} s^{2a} s^{n-1} ds 
\lesssim 
C_q^2 r^{2a +n}.
$$
\end{proof}

The second lemma is a classical construction.
\begin{lemma}
\label{lem_alexo_hg}
Let $g \in \cS (\bR)$ be an even function 
such that its (Euclidean) Fourier transform satisfies:
$$
\widehat g\in \cD(\bR), 
\qquad
\widehat g\big|_{[-\frac 12,\frac 12]} \equiv1,
\qquad\mbox{and}\qquad 
\widehat g\big|_{(-\infty,1]\cup[1,\infty)} \equiv0.
$$
Such a function exists. 

For any $d\in \bN$ and
 any $h\in \cS'(\bR)$ 
satisfying  $h\in \cC^d(\bR)$ with $\| h^{(d)}\|_\infty<\infty$,
we have
$$
\forall\delta>0\qquad
\|h - h * g_\delta\|_\infty \leq 
\frac {\delta^d}{d!} \int_\bR |y|^d |g(y)| dy  \ \| h^{(d)}\|_\infty,
$$
where $g_\delta$ is the function given by $g_\delta(x)=\delta^{-1} g(\delta^{-1}x)$. 
\end{lemma}

\begin{proof}[Lemma \ref{lem_alexo_hg}]
The hypothesis on $g$ implies
$$
\int_\bR g(x) dx =1 \quad\mbox{and}\quad
\int_\bR x^\ell g(x) dx =0 \ \mbox{for all}\ \ell \in \bN.
$$
Using the Taylor formula on $h$, we have
\begin{eqnarray*}
h*g_\delta (x)
&=&
\int_\bR h(x+\delta y) g(y) dy=
\int_\bR 
\left(\sum_{\ell=0}^{d-1} \frac{h^{(\ell)}(x)}{\ell!} (\delta y)^\ell +R_d(x,\delta y)\right)g(y) dy
\\
&=&
h(x) +\int_\bR 
R_d(x,\delta y) g(y) dy,
\end{eqnarray*}
where $R_d(x,\cdot)$ is the Taylor remainder of the function $h$ at $x$ of order $d$. 
We conclude easily with the following ($x$-independent) estimate for the remainder:
$$
|R_d(x,\delta y)|
\leq \frac{|\delta y|^d}{d!} \| h^{(d)}\|_\infty.
$$
\end{proof}

\subsection{Proof of Lemma \ref{lem_prop_mult_1}}
\label{subsec_pf_lem_prop_mult_1}

This section is devoted to proving Lemma \ref{lem_prop_mult_1}.
We will use the classical technics 
relying on estimates for the heat kernel,
see \cite{varo,alexo}. More precisely, we will follow closely the presentation of \cite{furioli+melzi+veneruso}.

Let $q\in \cD(G)$ vanishing at $e_G$ up to order $a-1\geq0$.

We fix a function  $f:[0,\infty)\to \bC$ with compact support in $[0,2]$. 
We assume that $f$ is regular enough, 
more precisely  in $\cC^d[0,\infty)$, 
that is, $d$-differentiable with $d$-th continuous derivatives.
 $d$ will be suitably chosen.

\textbf{Step 1:} 
For each $t>0$, we define the function $h_t:[0,\infty)\to\bC$ via
\begin{equation}
\label{pf_lem_prop_mult_1_def_ht}
h_t(\mu)=e^{-t\mu^2}f(t\mu^2),
\quad \mu\geq 0.
\end{equation}
We have
$$
\|h_t\|_\infty \leq e^2 \|f\|_\infty
\quad\mbox{and}\quad
f(t\lambda)=h_t(\sqrt \lambda ) e^{-t\lambda}.
$$
The spectral theorem implies easily  
$$
f(t\cL)\delta_e=h_t(\sqrt\cL) p_t
\quad\mbox{and}\quad
\|f(t\cL)\delta_e\|_{L^2(G)}\leq 
\|h_t \|_\infty \|p_t\|_{L^2(G)}.
$$
For the $L^2$-norm of the heat kernel, we use \eqref{eq_heat_kernel} 
and \eqref{eq_int_e-x2t}
to obtain 
$$
\|p_t\|_{L^2(G)} \leq C V(\sqrt t)^{-\frac12}.
$$
This implies  $f(t\cL)\delta_e \in L^2(G)$ with the following estimate:
\begin{equation}
\label{pf_lem_prop_mult_1_step1}
\|f(t\cL)\delta_e\|_{L^2(G)}\lesssim 
\|f\|_\infty V(\sqrt t)^{-\frac12}.
\end{equation}

\textbf{Step 2:} Let us show that the integral in the statement on a ball of radius $\sqrt t$ near the origin may be estimated by:
\begin{equation}
\int_{|x|<\sqrt t} 
| q(x) \ f(t\cL)\delta_e(x)| dx 
\lesssim_q \min (1, t^{\frac {a} 2})
\|f\|_\infty.
\label{pf_lem_prop_mult_1_step2}
\end{equation}
In order to show this, we first use  Cauchy-Schwartz' inequality:
$$
\int_{|x|<\sqrt t} 
| q(x) \ f(t\cL)\delta_e(x)| dx 
\leq
\|q\|_{L^2(B(\sqrt t))}
\|f(t\cL)\delta_e\|_{L^2(B(\sqrt t))}.
$$
The first $L^2$-norm of the right-hand side
may be estimated 
using Lemma \ref{lem_L2Brq}
and the second with \eqref{pf_lem_prop_mult_1_step1}:
$$
\|f(t\cL)\delta_e\|_{L^2(B(\sqrt t))}
\leq 
\|f(t\cL)\delta_e\|_{L^2(G)}
\lesssim 
\|f\|_\infty V(\sqrt t)^{-\frac12}.
$$
Hence $$
\int_{|x|<\sqrt t} 
| q(x) \ f(t\cL)\delta_e(x)| dx 
\lesssim_q \min (1,{\sqrt t} ^{a +\frac n2})
\|f\|_\infty V(\sqrt t)^{-\frac12}.
$$
Using the estimates for $V(r)$ in \eqref{eq_V(r)}, 
this shows the estimate in \eqref{pf_lem_prop_mult_1_step2}.

\textbf{Step 3:}  
For $t$ large, that is, if $\sqrt t$ is comparable with the radius $R_0$ of $G$, 
then the first part of Lemma \ref{lem_prop_mult_1} is proved.
Let us now consider the case of a multi-index $\beta\in \bN_0^n$, 
and still $\sqrt t$ comparable with the radius $R_0$ of $G$.
Proceeding as in Steps 1 and 2, we obtain
$$
\|X^\beta f(t\cL)\delta_e\|_{L^2(G)}
=
\| h_t(\sqrt\cL) \{X^\beta p_t\}\|_{L^2(G)}
\leq 
\|h_t \|_\infty \|X^\beta p_t\|_{L^2(G)},
$$
and 
$$
\|X^\beta p_t\|_{L^2(G)} \leq C, \quad \mbox{for} \ t\sim 1. 
$$
Hence 
$$
\|f(t\cL)\delta_e\|_{L^2(G)}\leq 
C e^2 \|f\|_\infty
$$
Thus the second part of Lemma \ref{lem_prop_mult_1} is proved for $t\sim 1$.
We therefore may assume that $t$ is small and consider the case of a multi-index $\beta\in \bN_0^n$.

\textbf{Step 4:}  
In order to finish the proof, it remains to show
\begin{equation}
\forall \sqrt t<\epsilon_0\qquad
\int_{|x|\geq \sqrt t} 
| q(x) X^\beta\{f(t\cL)\delta_e\}(x)| dx 
\lesssim
C_q  t^{\frac {a-|\beta|} 2}
\|f\|_{C^d}.
\label{pf_lem_prop_mult_1_step3}
\end{equation}
We will decompose the integrand using
$$
X^\beta \{f(t\cL)\delta_e\}=h_t(\sqrt\cL) X^\beta p_t
= 
h_t(\sqrt\cL)\sum_{j=0}^\infty 
 \{X^\beta p_t\} \ 1_{B(2^{j-1}\sqrt t)} 
 + \{X^\beta p_t\} \ 1_{B(2^{j-1}\sqrt t)^c)}.
$$
Here $1_{B(r)}$ and $1_{B(r)^c}$
denotes the indicatrix functions of the sets given by the ball $B(r)$ around the neutral element and by its complementary $B(r)^c$.
The function $h_t$ was defined earlier via \eqref{pf_lem_prop_mult_1_def_ht}.
Note that the sum over $j$
is finite but the number of terms is the smaller integer $J$ 
such that $2^{J+1} \sqrt t > R_0$, thus $J$ depends on $t$.
In order to obtain $t$-uniform estimates, we view this sum as infinite.
This decomposition yields
\begin{equation}
\label{pf_lem_prop_mult_1_step3_dec}
\int_{|x|\geq \sqrt t} 
| q(x) X^\beta\{f(t\cL)\delta_e\}(x)| dx 
\leq
\sum_{j=0}^\infty 
\int_{A_{t,j}} |q \ M_{t,j}^{(1)}|
+
\int_{A_{t,j}} |q \ M_{t,j}^{(2)}|,
\end{equation}
where 
$$
A_{t,j}
:=
\{x\in G:\ 2^j\sqrt t < |x| \leq  2^{j+1}\sqrt t\}
= B(2^{j+1}\sqrt t) \backslash B(2^{j}\sqrt t),
$$
and 
$$
M_{t,j}^{(1)}
:=
h_t(\sqrt \cL) \left\{X^\beta p_t \ 1_{B(2^{j-1}\sqrt t)}\right\}
\qquad\mbox{and}\qquad
M_{t,j}^{(2)}:=
h_t(\sqrt \cL) \left\{X^\beta p_t \ 1_{B(2^{j-1}\sqrt t)^c}\right\}.
$$

In both cases $i=1,2$, 
we will use Cauchy-Schwartz' inequality
$$
\int_{A_{t,j}} |q \ M_{t,j}^{(i)}|
\leq
\|q\|_{L^2(A_{t,j})}
\|M_{t,j}^{(i)}\|_{L^2(A_{t,j})}.
$$
For the first $L^2$-norm, we use Lemma \ref{lem_L2Brq}
(with $t$ small):
$$
\|q\|_{L^2(A_{t,j})}
\leq 
\|q\|_{L^2(B(2^{j+1}\sqrt t))}
\lesssim C_q  (2^{j+1} \sqrt t)^{a+\frac n 2}
$$ 

\textbf{Step 4a:}  
For the second $L^2$-norm, in the case $i=2$, we have
$$
\|M_{t,j}^{(2)}\|_{L^2(A_{t,j})}
\leq
\|M_{t,j}^{(2)}\|_{L^2(G)}
\leq
\|h_t(\sqrt \cL)\|_{\sL(L^2(G))} 
\|X^\beta p_t \ 1_{B(2^{j-1}\sqrt t)^c}\|_{L^2(G)}.
$$
On the one hand, we have 
by the spectral theorem 
$$
\|h_t(\sqrt \cL)\|_{\sL(L^2(G))}\leq 
\|h_t\|_\infty
\leq e^2 \|f\|_\infty.
$$
On the other hand, the estimate 
 for the heat kernel in \eqref{eq_heat_kernel_der} yields
\begin{eqnarray*}
\|X^\beta p_t \ 1_{B(2^{j-1}\sqrt t)^c}\|_{L^2(G)}^2
&\leq &
\sup_{|x|\geq 2^{j-1}\sqrt t} |X^\beta p_t(x)|
\int_G  |X^\beta p_t(x)| dx
\\
&\lesssim &
\sqrt t^{-n-|\beta|} e^{-\frac{2^{2(j-1)}}C}
\int_G \sqrt t^{-n-|\beta|} e^{-\frac{|x|^2}{C t}} dx
\\
&\lesssim &
\sqrt t^{-n-|\beta|} e^{-\frac{2^{2(j-1)}}C}\sqrt t^{-n-|\beta|} V(\sqrt t),
\end{eqnarray*}
by \eqref{eq_int_e-x2t}. Thus we have obtained
$$
\|X^\beta p_t \ 1_{B(2^{j-1}\sqrt t)^c}\|_{L^2(G)}
\lesssim
\sqrt t^{-\frac n2 -|\beta|} e^{-\frac{2^{2(j-1)}}C},
$$
and 
$$
\|M_{t,j}^{(2)}\|_{L^2(A_{t,j})}
\lesssim \|f\|_\infty  \sqrt t^{-\frac n2-|\beta|} e^{-\frac{2^{2(j-1)}}C}.
$$
Collecting the previous estimates yields:
\begin{eqnarray*}
\int_{A_{t,j}} |q \ M_{t,j}^{(2)}|
&\lesssim&
C_q  (2^{j+1} \sqrt t)^{a+\frac n 2}
\|f\|_\infty  \sqrt t^{-\frac n2-|\beta|} e^{-\frac{2^{2(j-1)}}C}
\\
&\lesssim& C_q  \|f\|_\infty \sqrt t^{a - |\beta|}
\ 
2^{(j+1)(a+\frac n2)} e^{-\frac{2^{2(j-1)}}C}.
\end{eqnarray*}
The exponential decay allows us to sum up over $j$ and to obtain:
\begin{equation}
\sum_{j=0}^\infty
\int_{A_{t,j}} |q \ M_{t,j}^{(2)}|
\lesssim
 C_q  \|f\|_\infty \sqrt t^{a - |\beta|}.
\label{pf_lem_prop_mult_1_step3a}
\end{equation}

\textbf{Step 4b:}  
The case of $i=1$, that is, the estimate of $\|M_{t,j}^{(1)}\|_{L^2(A_{t,j})}$, 
requires a more sophisticated argument.
The function $h_t$ is even and has compact support. 
Assuming  $f\in C^d[0,+\infty)$ with $d\geq 2$, 
the function $h_t\in C^d(\bR^d)$ admits an integrable Euclidean  Fourier transform of  $\widehat h_t\in L^1(\bR)$.
Hence the following formula holds for any $\mu\in \bR$
$$
h_t(\mu) = \frac 1{2\pi} \int_\bR \cos (s\mu)\ \widehat h_t(s)  ds, 
\quad \mu\in \bR,
$$
with a convergent integral.
The spectral theorem then implies 
$$
h_t(\sqrt \cL)= \frac 1{2\pi} \int_\bR \cos (s\sqrt \cL) \ \widehat h_t(s) ds
$$
and also
\begin{equation}
\label{pf_lem_prop_mult_1_step3_Mtj1}
M_{t,j}^{(1)}(x)=
\frac 1{2\pi} \int_\bR 
 \cos (s\sqrt \cL) 
\left\{X^\beta p_t \ 1_{B(2^{j-1}\sqrt t)}\right\}(x) \  \widehat h_t(s) ds.
\end{equation}

The operator $\cos (s\sqrt\cL)$ has finite unit propagation speed
\cite[ch. IV]{taylor_bk81}
in the sense that
$\supp \{\cos (s\sqrt\cL)\delta_e\} \subset B(|s|)$.
This implies
$$
x\in A_{t,j} \ \mbox{and}\  |s|\leq 2^{j-1}\sqrt t
\ \Longrightarrow\
\cos (s\sqrt \cL) 
\left\{X^\beta p_t \ 1_{B(2^{j-1}\sqrt t)}\right\}(x)
=0.
$$
We use this property in the following way. 
Let $g \in \cS (\bR)$ 
and $g_\delta=\delta^{-1} g(\delta^{-1}\cdot)$
be functions
as in Lemma \ref{lem_alexo_hg}.
As $\supp \ \widehat g_{(2^{j-1} \sqrt t)^{-1}} 
\subset 
[-2^{j-1}\sqrt t, 2^{j-1}\sqrt t]$,
the finite propagation speed property implies
$$
x\in A_{t,j} \ \Longrightarrow\
 \int_\bR 
 \cos (s\sqrt \cL) 
\left\{X^\beta p_t \ 1_{B(2^{j-1}\sqrt t)}\right\}(x) \  
\widehat h_t(s) \widehat g_{(2^{j-1} \sqrt t)^{-1}}  (s) ds = 0.
$$
Hence we can rewrite \eqref{pf_lem_prop_mult_1_step3_Mtj1}
for any $x\in A_{t,j}$ as
\begin{eqnarray*}
M_{t,j}^{(1)}(x)
&=&
\frac 1{2\pi} \int_\bR 
 \cos (s\sqrt \cL) 
\left\{X^\beta p_t \ 1_{B(2^{j-1}\sqrt t)}\right\}(x) 
\left(  \widehat h_t(s) - \widehat h_t(s) \widehat g_{(2^{j-1} \sqrt t)^{-1}}
\right) ds
\\
&=&
\left(h_t - h_t* g_{(2^{j-1} \sqrt t)^{-1}}\right)(\sqrt \cL)
\left\{X^\beta p_t \ 1_{B(2^{j-1}\sqrt t)}\right\}(x), 
\end{eqnarray*}
having used the spectral theorem and the inverse Fourier formula for even  functions on $\bR$.
Applying the $L^2$-norm on $A_{t,j}$, we obtain
\begin{eqnarray*}
\|M_{t,j}^{(1)}\|_{L^2(A_{t,j})}
&\leq&
\|\left(h_t - h_t* g_{(2^{j-1} \sqrt t)^{-1}}\right)(\sqrt \cL)
\left\{X^\beta p_t \ 1_{B(2^{j-1}\sqrt t)}\right\}\|_{L^2(G)}
\\
&\leq&
\| h_t - h_t* g_{(2^{j-1} \sqrt t)^{-1}}\|_\infty
\| X^\beta p_t \ 1_{B(2^{j-1}\sqrt t)}\|_{L^2(G)},
\end{eqnarray*}
by the spectral theorem.
We estimate the supremum norm with the result of Lemma \ref{lem_alexo_hg}:
$$
\| h_t - h_t* g_{(2^{j-1} \sqrt t)^{-1}}\|_\infty
\lesssim (2^{j-1} \sqrt t)^{-d} \|h_t^{(d)}\|_\infty,
$$ 
and one checks easily 
$$
\|h_t^{(d)}\|_\infty 
= t^{\frac d2} \|h_1^{(d)}\|_\infty 
\lesssim t^{\frac d2}  \max_{\ell =0,1,\ldots,d} 
\|f^{(\ell)}\|_{\infty}.
$$
For the $L^2$-norm, the estimates in  \eqref{eq_heat_kernel_der} for the heat kernel yields
$$
\| X^\beta p_t \ 1_{B(2^{j-1}\sqrt t)}\|_{L^2(G)}
\lesssim
\sqrt t^{-n -|\beta|} V(2^{j-1}\sqrt t)^{\frac 12}
 \lesssim
 \gamma_0^{\frac j2}  \sqrt t^{-\frac n2 -|\beta|}
$$
where we have set thanks to \eqref{eq_V(r)}:
$$
\gamma_0:=\sup_{r>0} \frac {V(2r)}{V(r)} \in (0,\infty).
$$
Hence we obtain
$$
\|M_{t,j}^{(1)}\|_{L^2(A_{t,j})}
\lesssim 
(2^{j-1} \sqrt t)^{-d} t^{\frac d2}  \max_{\ell =0,1,\ldots,d} 
\|f^{(\ell)}\|_{\infty}
 \gamma_0^{\frac j2}  \sqrt t^{-\frac n2 -|\beta|}
$$

We can now go back to 
\begin{eqnarray*}
\int_{A_{t,j}} |q \ M_{t,j}^{(1)}|
&\lesssim&
C_q  (2^{j+1} \sqrt t)^{a+\frac n 2}
(2^{j-1} \sqrt t)^{-d} t^{\frac d2}  \max_{\ell =0,1,\ldots,d} 
\|f^{(\ell)}\|_{\infty}
 \gamma_0^{\frac j2}  \sqrt t^{-\frac n2 -|\beta|}
 \\
&\lesssim&
C_q \max_{\ell =0,1,\ldots,d} 
\|f^{(\ell)}\|_{\infty}
2^{j(a+\frac n2 -d + \frac {\ln \gamma_0}2)} \sqrt t^{a-|\beta|}
\end{eqnarray*}
We choose $d$ to be the smallest positive integer such that 
$d>a+\frac n 2+ \frac {\ln \gamma_0}2$ so that we can sum up over $j$ 
to obtain
$$
\sum_{j=0}^\infty \int_{A_{t,j}} |q \ M_{t,j}^{(1)}|
\lesssim
C_q \max_{\ell =0,1,\ldots,d} 
\|f^{(\ell)}\|_{\infty}
\sqrt t^{a-|\beta|}.
$$
Using \eqref{pf_lem_prop_mult_1_step3_dec} and \eqref{pf_lem_prop_mult_1_step3a}, 
this shows \eqref{pf_lem_prop_mult_1_step3}.
This concludes the proof of  Lemma \ref{lem_prop_mult_1}.

\subsection{Proof of Proposition \ref{prop_app_mult_t}}
\label{subsec_pf_prop_app_mult_t}

\textbf{Reduction 1: in Proposition \ref{prop_app_mult_t},
we may assume $m<0$} for the following reasons.

Let  $f\in C^d[0,\infty)$ satisfying 
$\sup_{\substack{\lambda\geq 1\\ \ell =0,\ldots,d} }
\lambda^{-\frac m2 + \ell} |f^{(\ell)}(\lambda)|<\infty$.
Then $f_1(\lambda) = (1+\lambda)^{-N} f(\lambda)$ satisfies the same properties as $f$ but for $m_1=m-2N$ and we can choose $N$ large enough so that $m_1<0$.
As $f(\lambda) =f_1(\lambda)(1+ \lambda)^N$, 
we also have
$f(\lambda_\pi) = f_1(\lambda_\pi)(1+\lambda_\pi)^N$.
If we knew that $f_1$ satisfies the property described in 
Proposition \ref{prop_app_mult_t} for $m_1$ and any $q\in \cD(G)$
 then this together with Lemma \ref{lem_Deltaq_piXbeta} 
  would imply the property for functions $q$
yielding a collection $\Delta$
of RT-difference operators satisfying the Leibniz-like property
described in Definition \ref{def_leibniz}.
By Lemma \ref{lem_prop_indep_Delta}
and Theorem \ref{thm_Deltaeq+coincide} with Corollary \ref{cor_choice_Delta},
this would imply Proposition \ref{prop_app_mult_t}
for $f$ and any $q\in \cD(G)$.

\medskip

\textbf{Reduction 2: we may assume $f=0$ on $[0,1]$} 
as a consequence of the following property:

\begin{lemma}
\label{lem_prop_mult_2}
Let $m\in \bR$ and $a\in \bN_0$.
There exists $d=d_{a,m}\in \bN_0$ such that for any $q\in \cD(G)$
vanishing up to order $a-1$ there exists $C=C_{q,m}>0$
satisfying  
 for any function  $f\in C^d[0,\infty)$ with support in $[0,1]$: 
$$
\forall \pi\in\Gh, \ t\in (0,1)\qquad
\| \Delta_q f(t\lambda_\pi)\|_{\sL(\cH_\pi)} 
\leq C t^{\frac {m}  2}
(1+\lambda_\pi)^{\frac{m -a}2}
\max_{\ell =0,1,\ldots,d} 
\|f^{(\ell)}\|_{\infty}.
$$
\end{lemma}
\begin{proof}[Lemma \ref{lem_prop_mult_2}]
From the properties of the Laplace operator and its Sobolev spaces 
together with \eqref{eq_cF_L1}, we have:
\begin{eqnarray*}
&&(1+\lambda_\pi)^N\| \Delta_q f(t\lambda_\pi)\|_{\sL(\cH_\pi)} 
=
\|(1+\pi(\cL))^N \Delta_q f(t\lambda_\pi)\|_{\sL(\cH_\pi)} 
\\&& \qquad \leq
\int_G |(1+\cL)^N  q(x) \{(1+t\cL)^N f(t\cL)\delta_e\}(x)|dx
\\&& \qquad\lesssim \sum_{|\beta|\leq 2N}
\int_G | X^\beta q(x) \{ f(t\cL)\delta_e\}(x)|dx
\ \lesssim  \sum_{|\beta|\leq 2N} t^{\frac{a -|\beta|}2}
\max_{\ell =0,1,\ldots,d} 
\|f^{(\ell)}\|_{\infty},
\end{eqnarray*}
having used Lemma \ref{lem_prop_mult_1}
and Remark \ref{rem_lem_prop_mult_1} \eqref{item_rem_lem_prop_mult_1_Xq}.
Hence we have obtained
$$
\forall \pi\in\Gh, \ t\in (0,1)\qquad
\| \Delta_q f(t\lambda_\pi)\|_{\sL(\cH_\pi)} 
\leq C t^{\frac {a+m_1}  2}
(1+\lambda_\pi)^{\frac{m_1}2}
\max_{\ell =0,1,\ldots,d} 
\|f^{(\ell)}\|_{\infty}.
$$
for any $m_1=2N\in 2\bN$.
The properties of interpolation and duality of the Sobolev spaces imply the result for any $m_1\in \bR$.
We then choose $m_1=m-a$.
\end{proof}

\textbf{Strategy of the proof of Proposition \ref{prop_app_mult_t}:}
We may use the following notation:
$$
\|\kappa\|_* := \|T_\kappa\|_{\sL(L^2(G))}
=
\sup_{\pi\in \Gh} \|\cF_G \kappa(\pi)\|_{\sL(\cH_\pi)}.
$$
with the understanding that this quantity may be infinite.

Let $q\in \cD(G)$, $m<0$,
and  $f\in C^d[0,\infty)$ supported in $[1,\infty)$.
The properties of the Sobolev spaces imply that 
  it suffices to show
\begin{equation}
\label{eq_pf_prop_mult2}
\| \cL^{\frac b2} \{q\ f(t\cL)\delta_e\}\|_* \leq C 
t^{\frac{m}2}\!\!\sup_{\substack{\lambda\geq 1\\ \ell =0,\ldots,d} }\!\!
\lambda^{-\frac m2 + \ell} |\partial^\ell_\lambda f(\lambda)|,
\ \mbox{for}\ b=0,  -m+a.
\end{equation}
where $C=C_{b,\beta,q}>0$ and $a\in \bN_0$ is such that $q$ vanishes up to order $a-1$ at $e_G$.

Let us fix a dyadic decomposition,
that is, a function $\chi_1\in \cD(\bR)$  
satisfying
$$
0\leq \chi_1\leq 1,\quad
\chi_1\big|_{[\frac 34,\frac 32]}=1,\quad
\supp \chi_1 \subset [\frac 12,2],
$$
and 
$$
\forall \lambda \geq 1 \quad 
\sum_{j=1}^\infty \chi_j(\lambda)=1,
\qquad\mbox{where}\quad \chi_j(\lambda)=\chi(2^{-j}\lambda)\ \mbox{for} \ j\in \bN.
$$

We then set for $j\in \bN$ and $\lambda\geq 0$
$$
f_j(\lambda)
:=
\lambda^{-\frac m2}f(\lambda) \chi_j(\lambda)
\quad\mbox{and}\quad
g_j(\lambda)
:=
\lambda^{\frac m2} f_j(2^j \lambda).
$$
Note that, for any $j\in \bN_0$,  
$g_j$ is smooth, supported in $[\frac 12,2]$, 
and satisfies 
\begin{equation}
\label{eq_pf_prop_mult_bdgj}
\forall d\in \bN_0\quad
\|g_j^{(d)}\|_\infty \lesssim_m
\sup_{\substack {\lambda\geq 1 \\ \ell\leq d}}
\lambda^{-\frac m2+\ell}  |f^{(\ell)}(\lambda)|
\end{equation}

The sum $f(\lambda) =\sum_{j=1}^\infty 2^{j\frac m2} g_j(2^{-j}\lambda)$
is finite for any $\lambda\geq0$ and even locally finite on $[0,\infty)$.
Using \eqref{eq_pf_prop_mult_bdgj} 
and  $\sum_j 2^{j\frac m2}<\infty$ (recall that  $m<0$), 
we obtain
$$
\|f(t\cL)\|_{\sL(L^2(G))} 
\leq \sum_{j=1}^\infty 2^{j\frac m2} \| g_j(2^{-j}t\cL)\|_{\sL(L^2(G))}
 \lesssim_m \sup_{\lambda\geq 1} \lambda^{-\frac m2} |f(\lambda)|
 <\infty.
$$
Hence we can write 
$$
f(t\cL) 
=
\sum_{j=1}^\infty 2^{j\frac m2} g_j(2^{-j}t\cL) 
\quad\mbox{in}\ \sL(L^2(G)),
\quad\mbox{so}\quad
f(t\cL)\delta_e 
=
\sum_{j=1}^\infty 2^{j\frac m2} g_j(2^{-j}t\cL)\delta_e
\quad\mbox{in}\ \cD'(G),
$$
with each function $g_j(2^{-j}\cL)\delta_e$ being smooth, 
cf Remark \ref{rem_lem_prop_mult_1} \eqref{item_rem_lem_prop_mult_1_smooth}.
This justifies the estimates:
$$
\| X^\beta q f(t\cL)\delta_e\|_{L^1(G)}
\leq
\sum_{j=1}^\infty 2^{j\frac m2}
\| X^\beta q g_j(2^{-j}t\cL)\delta_e\|_{L^1(G)}
$$
By Lemma \ref{lem_prop_mult_1}
and Remark \ref{rem_lem_prop_mult_1} \eqref{item_rem_lem_prop_mult_1_Xq}, 
we have:
\begin{eqnarray}
\| X^\beta q g_j(2^{-j}t\cL)\delta_e\|_{L^1(G)}
&\lesssim_{q,\beta}&
(2^{-j}t )^{\frac {a-|\beta|}  2}
\max_{\ell=0,\ldots,d}
\|g_j^{(\ell)}\|_{\infty}
\nonumber
\\
&\lesssim_{q,\beta}&
(2^{-j}t )^{\frac {a-|\beta|}  2}
\sup_{\substack {\lambda\geq 1 \\ \ell\leq d}}
\lambda^{-\frac m2+\ell}  |f^{(\ell)}(\lambda)|,
\label{eq_pf_prop_mult_L1gj}
\end{eqnarray}
having used  \eqref{eq_pf_prop_mult_bdgj}.
This yields the (finite but crude) estimate:
\begin{eqnarray*}
\| X^\beta q f(\cL)\delta_e\|_{L^1(G)}
&\lesssim_{q,\beta,m} &
\sum_{j=1}^\infty 2^{j\frac m2}
(2^{-j}t )^{\frac {a-|\beta|}  2}
\sup_{\substack {\lambda\geq 1 \\ \ell\leq d}}
\lambda^{-\frac m2+\ell}  |f^{(\ell)}(\lambda)|
\\
&\lesssim_{q,\beta,m} &
t ^{\frac {a-|\beta|}  2}
\sup_{\substack {\lambda\geq 1 \\ \ell\leq d}}
\lambda^{-\frac m2+\ell}  |f^{(\ell)}(\lambda)|,
\end{eqnarray*}
as long as $m-a+|\beta|<0$. 
This rough $L^1$-estimate 
 implies the estimate in \eqref{eq_pf_prop_mult2} in the case $b=0$ but 
is not enough to prove the case $b=-m+a$. 
We now present an argument making us of the almost orthogonality of the decomposition of $f(\cL)$.
More precisely we will apply the Cotlar-Stein Lemma to the family of operators
$$
T_j:=2^{j \frac m2} T_{\cL^{\frac b2} \{q g_j(2^{-j}t\cL)\delta_e\} }, 
$$
where $b=-m+a$.
Note that the properties of the homogeneous Sobolev spaces imply
$$
\|\cL^{\frac b2} \{q g_j(2^{-j}t\cL)\delta_e\}\|_*
\leq 
\left(\|\cL^{\lceil \frac b2\rceil} \{q g_j(2^{-j}t\cL)\delta_e\} \|_*\right)^{\theta}
\left(\|\cL^{\lfloor \frac b2\rfloor} \{q g_j(2^{-j}t\cL)\delta_e\} \|_*
\right)^{1-\theta}
$$
with $\theta = \lfloor \frac b2\rfloor - \frac b2$
and we can bound the $\|\cdot\|_*$-norm 
with the $L^1$-norm given in  \eqref{eq_pf_prop_mult_L1gj}, summing up over $\beta$'s 
with $|\beta|=\lceil \frac b2\rceil$ or 
 $|\beta|=\lfloor \frac b2\rfloor$.
We obtain:
\begin{equation}
\label{eq_pf_prop_mult_bdTj}
\|\cL^{\frac b2} \{q g_j(2^{-j}t\cL)\delta_e\}\|_*
\lesssim_{q,b,m}
(2^{-j}t )^{\frac {a-b}  2}
\sup_{\substack {\lambda\geq 1 \\ \ell\leq d}}
\lambda^{-\frac m2+\ell}  |f^{(\ell)}(\lambda)|,
\end{equation}
and, as $q-b=m$, the operators $T_j$'s are uniformly bounded.
We also need to find a bound for the operator norm of 
$T_jT^*_k$ whose convolution kernel is
$$
2^{(j+k)\frac m2}
 \{\cL^{\frac b2} q g_j(2^{-j}t\cL)\delta_e\} * \{\cL^{\frac b2} q^* \bar g_k(2^{-k}t\cL)\delta_e\}.
 $$
 As the operator $\cL$ is central, this kernel may be also written  as 
$$
2^{(j+k)\frac m2}
\{\cL^{\frac {b+c}2}  \ q g_j(2^{-j}t\cL)\delta_e\} * 
\{\cL^{\frac {b-c}2}  q^* \bar g_k(2^{-k}t\cL)\delta_e\}
$$
for any real number $c$.
The estimate for $\|\cL^{\frac b2} \{q g_j(2^{-j}t\cL)\delta_e\}\|_*$
in  \eqref{eq_pf_prop_mult_bdTj}
holds in fact for any $b\geq 0$ and by duality for any $b\in \bR$.
Hence we can use it at $b\pm c$ to obtain
\begin{eqnarray*}
\|T_jT^*_k\|_{\sL(L^2(G))} 
&\leq&
2^{(j+k)\frac m2}
\|\cL^{\frac {b + c}2} \{q g_j(2^{-j}t\cL)\delta_e\}\|_*
\|\cL^{\frac {b -c }2} \{q g_k(2^{-k}t\cL)\delta_e\}\|_*
\\
&\lesssim_{q,b,c}&
2^{(j+k)\frac m2}
t^{a-b}
2^{-j\frac {a- (b +c)}  2}
2^{-k\frac {a-(b-c)}  2}
\left(\sup_{\substack {\lambda\geq 1 \\ \ell\leq d}}
\lambda^{-\frac m2+\ell}  |f^{(\ell)}(\lambda)|\right)^2
\\
&\lesssim_{q,b,c}&
2^{(j-k)\frac c 2}t^{a-b}
\left(\sup_{\substack {\lambda\geq 1 \\ \ell\leq d}}
\lambda^{-\frac m2+\ell}  |f^{(\ell)}(\lambda)|\right)^2,
\end{eqnarray*}
having used $b=-m+a$.
We choose $c$ to be the sign of $j-k$. 
This shows that the hypotheses of the Cotlar-Stein Lemma
\cite[Section VII.2]{stein_bk} are satisfied and this shows 
\eqref{eq_pf_prop_mult2} for $b=-m+a$.

This conclude the proof of  Proposition \ref{prop_app_mult_t}.

\section{A bilinear estimate}
\label{sec_bilinear}

This section is devoted to showing the following bilinear estimate 
which is used in the proof of the $L^2$-boundedness of pseudo-differential operators (cf. Lemma \ref{lem_L2bdd_rho}).

\begin{lemma}
\label{lem_bilinear}
For any $\gamma,s\in \bR$ with $2\gamma+s \leq 0$ and $s>n/2$, 
there exists $C=C_{s,\gamma,G}$ such that for any $\lambda,\mu\in \spec (\cL)$ with $\lambda\not=\mu$, 
for any 
$f\in \cH^{(\cL)}_{\lambda}$ and $g\in \cH^{(\cL)}_{\mu}$,
$$
\|(\id+\cL)^{\gamma} (fg)\|_{L^2}
\leq C
(1+|\mu-\lambda|)^{ (\gamma +\frac s2) } \| f \|_{L^2} \| g\|_{L^2}.
$$
\end{lemma}

Let us recall that $\cH_\lambda^{(\cL)}$ denotes the $\lambda$-eigenspace of $\cL$, see \eqref{eq_cHlambdacL}.
In the proof of Lemma \ref{lem_bilinear}, we will use the following properties of the Laplace-Beltrami operator obtained in relation with the theory of highest weight and representations:

\begin{lemma}
\label{lem_cL_rep}
Let $\lambda_1,\lambda_2\in \spec (\cL)$.
 If $f_i\in \cH_{\lambda_i}^{(\cL)}$, $i=1,2$, 
 then the point-wise product
$f_1 f_2$ is a function in $\oplus_{\lambda\leq \max(\lambda_1,\lambda_2)}
\cH_{\lambda}^{(\cL)}$.
\end{lemma}

\begin{proof}[Lemma \ref{lem_cL_rep}]
As is customary, we consider 
the highest weight theory on compact Lie groups
extended to the reductive case.
If $\pi\in \Gh$, denoting by $\tilde \pi$ its highest weight, 
the corresponding eigenvalue is \cite[Proposition 5.28]{knapp_bk}:
\begin{equation}
\label{eq_lambdapi_weight}
\lambda_\pi = |\tilde \pi+\rho_G|^2-|\rho_G|^2,
\end{equation}
where $\rho_G$ is the half-sum of the positive roots of the semi-simple part of  $\fg$.

By the Peter-Weyl theorem, for any $\pi\in \Gh$, 
the space
$L^2_{\pi}(G)$ decomposes as $d_\pi$ copies of the representation $\pi$, 
i.e. $L^2_{\pi}(G) \sim d_\pi V_\pi$ where $V_\pi$ is the abstract representation space of $\pi$, 
and any $f\in L^2_{\pi}(G)$ can be written as matrix coefficients of $\pi$. 
Hence if $f\in L^2_{\pi}(G)$ and $g\in L^2_{\tau}(G)$ then 
$fg$ is in the space which can be written as the abstract tensor product $(d_\pi V_\pi) \otimes (d_\tau V_\tau)$.
The highest weight among the irreducible components of 
$V_{\pi}\otimes V_{\tau}$ is of the form $\tilde \pi + \tilde \tau$
\cite[Proposition 9.72]{knapp_bk}. 
Naturally, $V_{\pi}\otimes V_{\tau}$ may contain other components with dominated weights, but, thanks to \eqref{eq_lambdapi_weight}, 
we always have 
\begin{equation}
\label{prop_cL_rep_pf}
\max\{\lambda_\omega:  \omega \in \Gh, V_\omega \subset V_{\pi}\otimes V_{\tau}\}
 \leq \lambda_\pi+\lambda_\tau.
\end{equation}
 Consequently,
 $fg\in \oplus_{\lambda\leq \lambda_\pi+\lambda_\tau} \cH_\lambda^{(\cL)}$ and 
the formulae in \eqref{eq_cHlambdacL1}
and \eqref{eq_lambdapi_weight}  imply 
the statement of Lemma \ref{lem_cL_rep}.
\end{proof}

\begin{proof}[Lemma \ref{lem_bilinear}]
Let $s,\gamma,\lambda,\mu,f,g$ be as in the statement. 
We may assume $\lambda<\mu$.
 
 The Plancherel formula and the Cauchy-Schwartz inequality easily imply:
\begin{eqnarray*}
\left| (\id+\cL)^{\gamma} (fg)(x)\right|
&=&
\big| \sum_{\lambda_\pi =\mu} d_\pi
\tr \left( (\id+\cL)_x^{\gamma} (f (x)\pi(x)) \widehat g(\pi)\right)
\big|
\\
&\leq&
\|g\|_{L^2}
\sqrt{\sum_{\lambda_\pi =\mu} d_\pi
\| (\id+\cL)_x^{\gamma} (f (x)\pi(x)) )\|_{HS(\cH_\pi)} ^2},
\end{eqnarray*}
Thus
$$
\| (\id+\cL)^{\gamma} (fg)\|_{L^2}^2
\leq
\|g\|_{L^2}^2
\sum_{\lambda_\pi =\mu} d_\pi
\int_G \| (\id+\cL)_x^{\gamma} (f (x)\pi(x)) )\|_{HS(\cH_\pi)} ^2dx ,
$$
We can easily rewrite these last integrals as
\begin{eqnarray*}
\int_G \| (\id+\cL)_x^{\gamma} (f (x)\pi(x)) )\|_{HS(\cH_\pi)} ^2dx 
=
\sum_{1\leq l,k\leq d_\pi}
\int_G | (\id+\cL)_x^{\gamma} (f (x)\pi_{l,k}(x)) )| ^2dx 
\\
=
\sum_{1\leq l,k\leq d_\pi}
\sum_{\tau\in \Gh} d_\tau (1+\lambda_\tau)^{2\gamma}
\| \tau^* \left(f \pi_{l,k}\right) \|_{HS(\cH_\tau)}^2.
\end{eqnarray*}
Now we notice that 
$$
\sum_{1\leq l,k\leq d_\pi}
\| \tau^* \left(f \pi_{l,k}\right) \|_{HS(\cH_\tau)}^2
=
\sum_{\substack{1\leq l,k\leq d_\pi \\1\leq l',k'\leq d_\tau }}
| [\tau^* \left(f \pi_{l,k}\right)]_{l',k'} |^2,
$$
and that 
$$
[\tau^* \left(f \pi_{l,k}\right)]_{l',k'}
=
\int_G f(x) \pi_{l,k}(x) \tau_{l',k'}(x) dx
=
[\pi^*(f\tau_{l',k'})]_{l,k},
$$
thus 
$$
\sum_{1\leq l,k\leq d_\pi}
\| \tau^* \left(f \pi_{l,k}\right) \|_{HS(\cH_\tau)}^2
=
\sum_{1\leq l',k'\leq d_\tau}
\| \pi(f\tau_{l',k'})^* \|_{HS(\cH_\pi)}^2.
$$
We have therefore obtained:
\begin{eqnarray}
\| (\id+\cL)^{\gamma} (fg)\|_{L^2}^2
&\leq&
\|g\|_{L^2}^2
\sum_{\lambda_\pi =\mu} d_\pi 
\sum_{\tau\in \Gh} d_\tau (1+\lambda_\tau)^{2\gamma}
\sum_{1\leq l',k'\leq d_\tau}
\| \pi(f\tau_{l',k'})^* \|_{HS(\cH_\pi)}^2
\nonumber\\
&\leq&
\|g\|_{L^2}^2
\sum_{\tau\in \Gh} d_\tau (1+\lambda_\tau)^{2\gamma}
\sum_{1\leq l',k'\leq d_\tau}
\| 1_{\mu}(\cL) (f\tau_{l',k'})\|_{L^2(G)}^2,
\label{eq_pf_lem_bilinear1}
\end{eqnarray}
by the Plancherel formula,
where  $1_{\mu}(\cL)$ denotes the orthogonal projection onto $\cH_{\mu}^{(\cL)}$.

As $f\in \cH_\lambda^{(\cL)}$ and $\tau_{l',k'}\in \cH_{\lambda_\tau}^{(\cL)}$, 
by Lemma \ref{lem_cL_rep}, 
$f\tau_{l',k'}\in \oplus_{\lambda' \leq \lambda + \lambda_\tau } 
\cH_{\lambda'}^{(\cL)}$.
Thus 
if $\lambda+\lambda_\tau <\mu$ then 
$1_{\mu}(\cL) (f\tau_{l',k'})=0$.
If $\lambda+\lambda_\tau \geq\mu$, then we use
$$
\sum_{1\leq l',k'\leq d_\tau}
\| 1_{\mu}(\cL) (f\tau_{l',k'})\|_{L^2(G)}^2
\leq
\sum_{1\leq l',k'\leq d_\tau}
\| f\tau_{l',k'}\|_{L^2(G)}^2
=d_\tau
\|f\|_{L^2(G)}^2.
$$
Inserting this in \eqref{eq_pf_lem_bilinear1}, we obtain:
\begin{eqnarray*}
\| (\id+\cL)^{\gamma} (fg)\|_{L^2}^2
&\leq&
\|g\|_{L^2}^2\|f\|_{L^2}^2
\sum_{\lambda_\tau \geq \mu -\lambda} 
 d_\tau^2 (1+\lambda_\tau)^{2\gamma}\\
&\leq&
C_s \|g\|_{L^2}^2\|f\|_{L^2}^2
(1+\mu -\lambda)^{2\gamma+s},
 \end{eqnarray*}
where $C_s:=\sum_{\tau\in \Gh} 
 d_\tau^2 (1+\lambda_\tau)^{-s}=\|\cB_s\|_{L^2(G)}$
 is finite for any $s>n/2$ 
 by Lemma \ref{lem_sob_embedding}.
 This concludes the proof of Lemma \ref{lem_bilinear}.
\end{proof}

\end{document}